\documentclass[a4paper,10pt]{article}

\usepackage{exscale,anysize,url}
\marginsize{2cm}{2cm}{2cm}{2cm}
\usepackage[centertags]{amsmath}
\usepackage{color,amscd}
\usepackage{amssymb}
\usepackage{amsthm}
\usepackage{mathtools}
\usepackage{dsfont,verbatim}
\usepackage[all]{xy}
\usepackage{tikz-cd}
\usepackage{enumitem}
\usepackage{hyperref}
\usepackage{extarrows}

\newtheorem{definition}{Definition}[section]
\newtheorem{theorem}[definition]{Theorem}
\newtheorem{proposition}[definition]{Proposition}
\newtheorem{lemma}[definition]{Lemma}
\newtheorem{corollary}[definition]{Corollary}

\newtheorem{remark}[definition]{Remark}

\newcommand{\nd}{\noindent}

\newcommand{\dG}{{\mathds G}}
\newcommand{\dV}{{\mathds V}}
\newcommand{\dR}{{\mathds R}}
\newcommand{\dC}{{\mathds C}}
\newcommand{\dQ}{{\mathds Q}}

\newcommand{\dN}{{\mathds N}}

\newcommand{\dZ}{{\mathds Z}}
\newcommand{\dP}{{\mathds P}}

\newcommand{\dA}{{\mathbb A}}

\newcommand{\dD}{{\mathds D}}
\newcommand{\dE}{{\mathds E}}

\newcommand{\bDelta}{\mathbf{\Delta}}

\newcommand{\kk}{k}
\newcommand{\op}{\mathrm{op}}
\newcommand{\Mod}[2]{\mathrm{Mod}(#1_{#2})}
\newcommand{\Modrc}[2]{\mathrm{Mod}_{\rc}(#1_{#2})}
\newcommand{\Modc}[2]{\mathrm{Mod}^\mathrm{c}(#1_{#2})}
\newcommand{\Ind}[2]{\mathrm{I}(#1_{#2})}

\newcommand{\Db}[2]{\mathrm{D}^\mathrm{b}(#1_{#2})}
\newcommand{\DbL}[1]{\mathrm{D}^\mathrm{b}(L_{#1})}
\newcommand{\DbI}[2]{\mathrm{D}^\mathrm{b}(\mathrm{I}#1_{#2})}

\newcommand{\EbI}[2]{\mathrm{E}^\mathrm{b}(\mathrm{I}#1_{#2})}
\newcommand{\EbRcI}[2]{\mathrm{E}^\mathrm{b}_{\rc}(\mathrm{I}#1_{#2})}
\newcommand{\EbIC}[1]{\mathrm{E}^\mathrm{b}(\mathrm{I}\dC_{#1})}
\newcommand{\DbIL}[1]{\mathrm{D}^\mathrm{b}(\mathrm{I}L_{#1})}
\newcommand{\EbIL}[1]{\mathrm{E}^\mathrm{b}(\mathrm{I}L_{#1})}
\newcommand{\sP}{\mathsf{P}}
\newcommand{\RR}{\mathrm{R}}
\newcommand{\EE}{\mathrm{E}}
\newcommand{\Dbhol}[1]{\mathrm{D}^\mathrm{b}_\mathrm{hol}(\mathcal{D}_{#1})}
\newcommand{\SolE}[1]{\mathcal{S}ol^\mathrm{E}_{#1}}

\newcommand{\Gmu}{\mathds{G}_{m,u}}

\newcommand{\Sol}[1]{\mathcal{S}ol_{#1}}
\newcommand{\DE}[1]{\mathrm{D}^\mathrm{E}_{#1}}

\newcommand{\cA}{\mathcal{A}}

\newcommand{\cD}{\mathcal{D}}
\newcommand{\cE}{\mathcal{E}}
\newcommand{\cF}{\mathcal{F}}
\newcommand{\cG}{\mathcal{G}}
\newcommand{\cH}{\mathcal{H}}
\newcommand{\cI}{\mathcal{I}}

\newcommand{\cL}{\mathcal{L}}
\newcommand{\cM}{\mathcal{M}}
\newcommand{\cN}{\mathcal{N}}
\newcommand{\cO}{\mathcal{O}}
\newcommand{\cP}{\mathcal{P}}

\newcommand{\cR}{\mathcal{R}}
\newcommand{\cS}{\mathcal{S}}

\newcommand{\cX}{\mathcal{X}}
\newcommand{\cY}{\mathcal{Y}}

\newcommand{\fF}{\mathfrak{F}}

\newcommand{\D}{\displaystyle}

\DeclareMathOperator{\Spec}{\textup{Spec}\,}

\DeclareMathOperator{\Aut}{\textup{Aut}}

\DeclareMathOperator{\DR}{\mathit{DR}}

\DeclareMathOperator{\im}{\textup{im}}
\DeclareMathOperator{\FL}{\textup{FL}}
\DeclareMathOperator{\id}{\textup{id}}

\DeclareMathOperator{\Conv}{\textup{Conv}}

\newcommand{\Supp}{\textup{Supp}}

\newcommand{\Gal}{\textup{Gal}}
\newcommand{\Gm}{\dG_{m,q}}

\newcommand{\indlim}[1]{\mathop{``\varinjlim"}\limits_{#1}}

\long\def\inhibe#1\endinhibe{\relax}

\newcommand{\wt}{\widetilde}

\newcommand{\conv}{\mathbin{\overset{+}{\otimes}}}
\newcommand{\Ihomp}{\cI hom^+}
\newcommand{\Ihom}{\cI hom}
\newcommand{\rc}{\dR\textnormal{-c}}
\newcommand{\real}{\mathop{\mathrm{Re}}}

\begin{document}
\title{Betti structures of hypergeometric equations}
\author{Davide Barco, Marco Hien, Andreas Hohl and Christian Sevenheck}

\renewcommand{\thefootnote}{}
\footnotetext{The first and fourth authors are partially supported by the DFG project SE 1114/3-2. The second author was partially supported by the DFG research fellowship HO 6925/1-1.}
\renewcommand{\thefootnote}{\arabic{footnote}}

\maketitle

\begin{abstract}
We study Betti structures in the solution complexes of confluent hypergeometric equations. We use the framework of enhanced ind-sheaves and the irregular Riemann-Hilbert correspondence of D'Agnolo--Kashiwara. The main result is a group theoretic criterion that ensures that enhanced solutions of such systems are defined
over certain subfields of $\dC$. The proof uses a description of the hypergeometric systems as exponentially twisted Gau\ss--Manin systems of certain Laurent polynomials.
\end{abstract}

\section{Introduction}
\label{sec:Introduction}

The aim of this paper is to study Betti structures in solutions of certain univariate hypergeometric $\cD$-modules. These differential systems have a long history, starting from work of  Euler and Gau\ss. In modern language, they are the most basic examples of rigid $\cD$-modules. Recent applications include quantum differential equations of toric manifolds (\cite{Giv7}, \cite{Ir2,Ir-Periods} or \cite{ReiSe, ReiSe2}) and an in-depth study of the Hodge theoretic properties of such systems (\cite{Sa15, SevCast, SevCastReich, Sa16}). From the Hodge theoretic point of view, but also for many other applications, it is important to find subfields of the complex numbers over which the solutions of these equations can be defined. This also fits to the more general program to understand Betti structures in  holonomic $\cD$-modules, see e.g.\ \cite{Mo7}. A first example of a result in this direction is a theorem of Fedorov (\cite[Theorem 2]{Fedorov}), which was conjectured by Corti and Golyshev (\cite{CortiGolyshev})
and which gives arithmetic conditions on the exponents for a \emph{regular} (i.e. non-confluent) hypergeometric system to underly a \emph{real} variation of polarized Hodge structures. The rigidity property of these systems is used in an essential way in his argument.

In the present paper, we take up the question of the existence of Betti structures in the solution spaces of univariate hypergeometric differential equations. We replace Fedorov's proof by a geometric argument, linking these $\cD$-modules to Gau\ss--Manin systems of certain Laurent polynomials. Our approach is distinct from Fedorov's in at least two ways:
\begin{enumerate}
\item We consider more generally not necessarily regular hypergeometric equations (these irregular ones are usually called \emph{confluent}) using the theory of enhanced solutions
of D'Agnolo--Kashiwara (\cite{DK16}).
\item We discuss a more general setup, where we consider any finite Galois extension $L/K$ such that the solutions of the given hypergeometric systems are a priori defined over $L$ (e.g.\ if all exponents of the system are rational, then $L$ is a cyclotomic field), and we establish and prove a group theoretic criterion for them to be defined over $K$.
\end{enumerate}
As an application, we get a general criterion for the enhanced solutions to be defined over the real numbers (similar in shape to the one from \cite{Fedorov}), and, if all exponents are rational, we determine when the solutions are defined over $\dQ$. One can deduce in a rather straightforward way similar statements for the perverse sheaf of (classical) solutions, and we obtain in particular, in the non-confluent case, a new proof of Fedorov's result.

The paper is organized as follows:
We first discuss in Section \ref{sec:BettiConj} the following question: Given a Galois extension $L/K$ and an enhanced ind-sheaf defined over $L$, what are criteria to ensure that it comes from (i.e.\ is obtained by extension of scalars from) an object defined over $K$? This is applied in Section \ref{sec:Laurent}, where we introduce hypergeometric modules and prove a geometric realization of them, relying on earlier work of Schulze--Walther (\cite{SchulWalth2}) and Reichelt (\cite{Reich2}). Our main result is then Theorem \ref{theo:MainTheo2}, which gives the group theoretic criterion alluded to above. In order to illustrate this result, we formulate here a shortened version, which covers the cases we are mostly interested in.
\begin{theorem}\label{theo:MainTheoIntro}
Let numbers $\alpha_1,\ldots,\alpha_n,\beta_1,\ldots,\beta_m \in [0,1)$ be given,
where $n\geq m$ and where we assume that $\alpha_i\neq \beta_j$ for all $i\in\{1,\ldots,n\}$ and all $j\in \{1,\ldots,m\}$. Consider the differential operator in one variable
$$
P:=\prod_{i=1}^{n}\left(q\partial_q-\alpha_i\right)-q\cdot \prod_{j=1}^{m}\left(q\partial_q-\beta_j\right)
$$
and let $\cH(\alpha;\beta):=\cD_{\Gm}/\cD_{\Gm} \cdot P \in \mathrm{Mod}_\mathrm{hol}(\cD_{\Gm})$ be the corresponding irreducible hypergeometric module on the one-dimensional torus $\Gm$. Suppose that $L\subset \dC$ is a field such that $e^{2\pi i \alpha_i}$, $e^{2\pi i \beta_j}\in L$ for all $i,j$. Let $K\subset L$ be such that $L/K$ is a finite Galois extension. Then if $\Gal(L/K)$ induces actions on
$\{e^{2\pi i\alpha_1},\ldots,e^{2\pi i\alpha_n}\}$ and on $\{e^{2\pi i\beta_1},\ldots,e^{2\pi i\beta_m}\}$,
the enhanced ind-sheaf $\SolE{\Gm}(\cH(\alpha;\beta))$ (see section \ref{sectionRH} below) is defined over $K$, that is, comes from an enhanced ind-sheaf over $K$ by extension of scalars.\end{theorem}
The proof of this theorem will be given right after the proof of Theorem \ref{theo:MainTheo2} on page \pageref{page:proofMainTheoIntro}.
As an example (see Theorem \ref{theo:RealStruct})
the above criterion applies if those of the numbers $\alpha_i$ and $\beta_j$ which are non-zero are symmetric around $\frac{1}{2}$, in which case (taking $L=\dC$ and $K=\dR$) we obtain that the enhanced solutions of the system $\cH(\alpha;\beta)$ are defined over $\dR$. A similar reasoning leads to a criterion (Theorem \ref{theo:RatStruct}) showing the existence of rational structures.
Finally, in  Section \ref{sec:Stokes}, we draw some consequences for Stokes matrices associated to the irregular singular point of confluent hypergeometric equations.\vspace{0.3cm}

\textbf{Acknowledgements.} We are grateful to Takuro Mochizuki for some explanations about his work on Betti structures and enhanced ind-sheaves, and for pointing out some ideas on the proofs in Section~\ref{sec:Stokes}. We also thank Andrea D'Agnolo for useful correspondence during the preparation of this work.

Finally, we thank the anonymous referees for their very careful reading of the paper, and for suggesting a number of improvements to the text.

\section{Betti structures and enhanced ind-sheaves}
\label{sec:BettiConj}

The functors of complex conjugation on the category of complex vector spaces and complexification of real vector spaces are well-known. In this chapter, we develop a theory of Galois conjugation and extension of scalars for more general field extensions $L/K$ in the context of sheaves, ind-sheaves and enhanced ind-sheaves. Since the latter are the topological counterpart of holonomic $\cD$-modules via the irregular Riemann--Hilbert correspondence, this will produce a framework for studying the question of when solutions of a differential system (a priori defined over $\dC$) admit a structure over a subfield $K$ of $\dC$. In this case, we will say that the differential system carries a $K$-Betti structure (or $K$-structure for short, see Definition \ref{def:KStructure} below).
In particular, we will prove that in certain cases, the fact that an object over $L$ is isomorphic to all its Galois conjugates (in a compatible way) implies that this object already comes from an object defined over $K$.

\subsection{Enhanced ind-sheaves}\label{sect:enhanced}
In \cite{DK16}, the authors introduced the category of enhanced ind-sheaves, which we will briefly recall here.
We assume all topological spaces to be \emph{good} in this section (i.e.\ Hausdorff, locally compact, second countable and of finite flabby dimension).

Let $k$ be an arbitrary field, and let $X$ be a topological space. We denote by $\Mod{\kk}{X}$ the category of sheaves of $\kk$-vector spaces and by $\Db{\kk}{X}$ its bounded derived category with the six Grothendieck operations $\RR\cH om_{\kk_X}$, $\otimes$, $\RR f_*$, $f^{-1}$, $\RR f_!$ and $f^!$.

The category $\Ind{\kk}{X}$ of \emph{ind-sheaves} on $X$ is the category of ind-objects for $\Modc{\kk}{X}$, the category of compactly supported sheaves of $\kk$-vector spaces on $X$ (see \cite{KS01} and \cite{KS06} for the theory of ind-sheaves and -objects). In other words, an object $F\in\Ind{\kk}{X}$ is of the form
$$F=\indlim{i\in I} \cF_i,$$
where the $\cF_i\in\Modc{\kk}{X}$ form a small filtrant inductive system and $\indlim{}$ denotes\vspace{-0.3cm} the inductive limit in the category of functors $\Modc{\kk}{X}^\op\to \mathrm{Mod}(\dZ)$ (i.e.\ one considers the sheaves $\cF_i$ after Yoneda embedding).

There is a fully faithful and exact embedding from the category of (not necessarily compactly supported) sheaves into the category of ind-sheaves
$\iota_X\colon \Mod{\kk}{X}\hookrightarrow \Ind{\kk}{X}$. If there is no confusion, the functor $\iota_X$ will often be omitted in the notation. This embedding has an exact left adjoint, denoted by $\alpha_X$, which in turn has an exact left adjoint $\beta_X$.  Moreover, one has a formalism of six Grothendieck operations on $\Ind{\kk}{X}$, where one denotes the proper direct image by $\RR f_{!!}$ to distinguish it from the operation $\RR f_!$ for sheaves, since it is not compatible with $\iota_X$. The derived category of $\Ind{\kk}{X}$ is denoted by $\DbI{\kk}{X}$.

Now, let $\cX=(X,\widehat{X})$ be a bordered space, i.e.\ a pair of topological spaces such that $X\subseteq \widehat{X}$ is an open subset. One defines the category of enhanced ind-sheaves on $\cX$ by two successive quotients (we refer to \cite{DK16} and \cite{EnhPerv} for details on this construction):
Denote by $\sP=\mathrm{P}^1(\dR)$ the real projective line and define the bordered space $\dR_\infty\vcentcolon=(\dR,\sP)$. Then
$$\DbI{\kk}{\cX\times\dR_\infty}\vcentcolon= \DbI{\kk}{\widehat{X}\times\sP}/\DbI{\kk}{(\widehat{X}\times\sP)\setminus (X\times\dR)},\qquad \DbI{\kk}{\cX}\vcentcolon= \DbI{\kk}{\widehat{X}}/\DbI{\kk}{\widehat{X}\setminus X}$$
and
$$\EbI{\kk}{\cX}\vcentcolon= \DbI{\kk}{\cX\times\dR_\infty}/\pi^{-1}\DbI{\kk}{\cX},$$
where $\pi\colon \cX\times\dR_\infty\to \cX$ is the morphism of bordered spaces induced by the projection. The total quotient functor $Q\colon \DbI{\kk}{\widehat{X}\times\sP}\to \EbI{\kk}{\cX}$ has a fully faithful left adjoint, which we will denote by $\cL$ (it is denoted by $\RR j_{\cX\times\dR_\infty !!}\mathrm{L}^\EE$ in \cite{EnhPerv}).

The category of enhanced ind-sheaves still comes with the six-functor formalism and these operations are denoted by $\RR\Ihomp$, $\conv$, $\EE f_*$, $\EE f^{-1}$, $\EE f_{!!}$ and $\EE f^!$ (for a morphism $f$ of bordered spaces). One also has a duality functor, denoted by $\DE{\cX}$. In addition, for an object $\cF\in \Db{\kk}{X}$, one has the operation
\begin{equation}\label{eq:pitens}
\EbI{\kk}{\cX}\to \EbI{\kk}{\cX}, \quad H\mapsto \pi^{-1}\cF \otimes H
\end{equation}
induced by the tensor product on $\DbI{\kk}{\widehat{X}\times\sP}$. We will in particular abbreviate $H_V\vcentcolon=\pi^{-1}k_V\otimes H$ for a subset $V\subset X$, and we recall that $H_V\cong \mathrm{E}i_{V_\infty!!}\mathrm{E}i_{V_\infty}^{-1}H$, where $i_{V_\infty}\colon V_\infty=(V,\overline{V})\to \cX$ is the embedding and $\overline{V}$ denotes the closure of $V$ in $\widehat{X}$) (see \cite[Lemma 2.7.6]{EnhPerv}).

We will often encounter the objects
\begin{equation}\label{eq:DefExp}
\kk^\EE_\cX\vcentcolon= \indlim{a\to\infty} \kk_{\{t\geq a\}}\in\EbI{\kk}{\cX}\qquad \textnormal{and} \qquad \dE^{\phi}_\kk\vcentcolon= \kk^\EE_\cX\conv \kk_{\{t+\phi\geq 0\}}\cong \indlim{a\to\infty}\kk_{\{t+\phi \geq a\}}\in\EbI{\kk}{\cX},
\end{equation}
where $\{t\geq a\}\vcentcolon= \{(x,t)\in\widehat{X}\times\sP;x\in X, t\in \dR, t\geq a\}$, and for a continuous function $\phi\colon U\to \dR$ on an open subset $U\subseteq X$ we set $\{t+\phi\geq a\}\vcentcolon= \{(x,t)\in\widehat{X}\times\sP; x\in U, t\in \dR, t+\phi(x)\geq a\}$.

Moreover, for each $H\in\EbI{\kk}{\cX}$, we set
\begin{equation}\label{eq:Sh}
\mathsf{sh}_{\cX}(H)=
\alpha_{X} j^{-1} \mathrm{R}\overline{\pi}_* \mathrm{R}\Ihom (\kk_{\{t\geq 0\}}\oplus \kk_{\{t\leq 0\}},\mathcal{L}H)\in \Db{\kk}{X},
\end{equation}
where $j\colon X\hookrightarrow \widehat{X}$ is the embedding and $\overline{\pi}\colon \widehat{X}\times\mathsf{P}\to \widehat{X}$ is the projection. One calls $\mathsf{sh}_\cX$ the \emph{sheafification functor} for enhanced ind-sheaves on the bordered space (see \cite{DKsheafification} for a detailed study).

On a real analytic manifold $X$, one has the notions of $\dR$-constructible sheaves and subanalytic ind-sheaves on $\widehat{X}$ (see \cite{KS90} and \cite{KS01}, where subanalytic ind-sheaves were called ``ind-$\dR$-constructible ind-sheaves''). This gives the full subcategories $\Modrc{k}{X}\subset\Mod{k}{X}$, $\mathrm{D}^{\mathrm{b}}_{\rc}(k_{X})\subset\Db{k}{X}$ and $\mathrm{I}_{\mathrm{suban}}(k_{X})$, $\mathrm{D}^{\mathrm{b}}_{\mathrm{suban}}(\mathrm{I}k_X)$.
In the case where $\cX$ is a real analytic bordered space (i.e.\ $\widehat{X}$ is a real analytic manifold and $X\subset \widehat{X}$ is a subanalytic subset), the full subcategory $\EbRcI{\kk}{\cX}\subset\EbI{\kk}{\cX}$ consisting  of $\dR$-constructible enhanced ind-sheaves was introduced in \cite{DK16, EnhPerv}.

The \emph{standard t-structure} (cf.\ \cite[§3.4]{DK16}, \cite{EnhPerv}) on $\EbI{\kk}{\cX}$ is the one induced by the standard t-structure on the derived category $\DbI{\kk}{\widehat{X}\times\sP}$, i.e.
\begin{align*}
\mathrm{E}^{\leq n}(\mathrm{I}\kk_\cX)=\{ H\in\EbI{\kk}{\cX}; \cL(H)\in \mathrm{D}^{\leq n}(\mathrm{I}\kk_{\widehat{X}\times\sP}) \}\\
\mathrm{E}^{\geq n}(\mathrm{I}\kk_\cX)=\{ H\in\EbI{\kk}{\cX}; \cL(H)\in \mathrm{D}^{\geq n}(\mathrm{I}\kk_{\widehat{X}\times\sP}) \}
\end{align*}
for $n\in\dZ$. Its heart $\mathrm{E}^0(\mathrm{I}\kk_\cX)$ therefore consists of those objects $H\in\EbI{\kk}{\cX}$ such that $\cL(H)$ is concentrated in degree zero. This is an abelian category. The associated cohomology functors are denoted by $\cH^n$.

In \cite{EnhPerv}, the authors define generalized t-structures $({}^p\EE_{\rc}^{\leq c}(\mathrm{I}k_\cX),{}^p\EE_{\rc}^{\geq c}(\mathrm{I}k_\cX))_{c\in\dR}$ on the category $\EbRcI{k}{\cX}$ for so-called perversities $p\colon \mathbb{Z}_{\geq 0}\to\dR$. Its heart ${}^p\EE_{\rc}^0(\mathrm{I}k_\cX)\vcentcolon={}^p\EE_{\rc}^{\leq 0}(\mathrm{I}k_\cX)\cap{}^p\EE_{\rc}^{\geq 0}(\mathrm{I}k_\cX)$ is a quasi-abelian category. In particular, they introduce the \emph{middle perversity (generalized) t-structure} for $p(n)=-\frac{n}{2}$, which is denoted by $({}^{1/2}\EE_{\rc}^{\leq c}(\mathrm{I}k_\cX),{}^{1/2}\EE_{\rc}^{\geq c}(\mathrm{I}k_\cX))_{c\in\dR}$. We refer to loc.~cit. for details.

\subsection{Galois descent for enhanced ind-sheaves}
In this section, we generalize two constructions, which are well-known for vector spaces, to the category of enhanced ind-sheaves: For a given finite Galois extension $L/K$, we consider, firstly, conjugation of an enhanced ind-sheaf over $L$ with respect to elements of the Galois group and, secondly, extension of scalars on an enhanced ind-sheaf over $K$. We establish some properties of these functors, and we will then show that the existence of suitable isomorphisms between an object over $L$ and its conjugates implies that the object comes from an object over the subfield $K$ by extension of scalars. This procedure is often called ``Galois descent'' and classical references for this construction on vector spaces can, for example, be found in \cite[§10.2]{Jac}, \cite[§AG.14]{BorelLAG}, \cite[§17]{Waterhouse}, \cite[§3.2]{Winter}. Let us also point out the exposition \cite{Conrad}.

Let $L/ K$ be a Galois extension and $G$ its Galois group.
For an element $g\in G$ and an $L$-vector space $V$, one can define the \emph{$g$-conjugate of $V$}, denoted by $\overline{V}^g$, as follows: As a $K$-vector space, $\overline{V}^g=V$, and the action of $L$ is given by $l\cdot v\vcentcolon= g(l)v$.

One easily checks that this construction (by applying it to sections over any open set) defines $g$-conjugation functors for sheaves of $L$-vector spaces on a topological space $M$:
\begin{equation}\label{eq:ConjSheaves}
\mathrm{Mod}(L_M)\to \mathrm{Mod}(L_M), \cF\mapsto \overline{\cF}^g
\end{equation}
(The restriction morphisms of $\overline{\cF}^g$ are the same as those of $\cF$.)
This functor is exact for any $g\in G$ and hence induces a functor on the derived category of sheaves
$$\Db{L}{M}\to\Db{L}{M}, \cF^\bullet \mapsto \overline{\cF^\bullet}^g.$$

The following lemma shows that conjugation is a very ``tame'' operation.
\begin{lemma} \label{lemmaConjCompatShv}
Let $f\colon X\to Y$ be a morphism of topological spaces and let $\cF,\cF_1,\cF_2\in\Db{L}{X}$, $\cG\in\Db{L}{Y}$. Then we have isomorphisms for any $g,h\in G$
\begin{itemize}
\item[(i)] $\overline{\mathrm{R}f_*\cF}^g\cong \mathrm{R}f_*\overline{\cF}^g$ and $\overline{\mathrm{R}f_!\cF}^g\cong \mathrm{R}f_!\overline{\cG}^g$,
\item[(ii)] $\overline{\mathrm{R}f^{-1}\cG}^g\cong \mathrm{R}f^{-1}\overline{\cG}^g$ and $\overline{\mathrm{R}f^!\cG}^g\cong \mathrm{R}f^!\overline{\cG}^g$,
\item[(iii)] $\overline{\cF_1\otimes\cF_2}^g\cong \overline{\cF_1}^g\otimes\overline{\cF_2}^g$ and $\RR\cH om_{L_X}(\overline{\cF_1}^g,\overline{\cF_2}^g)\cong \overline{\RR\cH om_{L_X}(\cF_1,\cF_2)}^g$,
\item[(iv)] $\overline{\overline{\cF}^g}^h\cong \overline{\cF}^{gh}$.
\end{itemize}
\end{lemma}
\begin{proof}
\begin{description}
\item[\textnormal{(i)}]
By definition, for the underived direct image functor we have $(f_*\cF)(U)=\cF(f^{-1}(U))$ for any open subset $U\subseteq X$. Since conjugation for sheaves is defined on sections, this yields a  natural isomorphism $\overline{f_*\cF}^g\cong f_*\overline{\cF}^g$. The functor $\overline{(\bullet)}^g$ is exact, and hence one obtains the first statement by derivation. The second isomorphism follows similarly: We note that $f_!\cF$ is a subsheaf of $f_*\cF$ and that the above isomorphism induces an isomorphism $\overline{f_!\cF}^g\cong f_!\overline{\cF}^g$, and conclude again by deriving the composition of functors.
\item[\textnormal{(ii)}] It is not difficult to see that
$$\mathrm{Hom}_{\Db{L}{X}}\big(\overline{\cF_1}^g,\cF_2\big) \cong \mathrm{Hom}_{\Db{L}{X}}\big(\cF_1,\overline{\cF_2}^{g^{-1}}\big)$$
for $\cF_1,\cF_2\in\Db{L}{X}$. Using this, the statements in (ii) follow from (i) by adjunction.
\item[\textnormal{(iii) and (iv)}] follow directly from the corresponding statements for vector spaces (for the tensor product, note that $g$-conjugation is compatible with sheafification, or derive it by adjunction from the statement about $\mathrm{R}\cH om$).
\end{description}
\end{proof}

By the general theory of ind-sheaves (cf.\ \cite[p.\ 7]{KS01}), $g$-conjugation further extends to a functor on ind-sheaves
$$\Ind{L}{M}\to\Ind{L}{M}, F=\indlim{i} \cF_i \mapsto \overline{F}^g=\indlim{i}\overline{\cF_i}^g.$$
Since this functor is still exact (cf.\ \cite[p.\ 11]{KS01}), we get a $g$-conjugation on the derived category of ind-sheaves
$$\DbIL{M}\to\DbIL{M}, F^\bullet \mapsto \overline{F^\bullet}^g.$$
A corresponding statement like Lemma \ref{lemmaConjCompatShv} holds for ind-sheaves.
Moreover, conjugation behaves nicely with respect to the functors $\iota_X$, $\alpha_X$ and $\beta_X$ between sheaves and ind-sheaves, as the following result shows.
\begin{lemma}
Let $\cF\in\DbL{X}$, $F\in\DbIL{X}$ and $g\in G$. Then there are isomorphisms
\begin{itemize}
	\item[(i)] $\overline{\iota_X\cF}^g\cong \iota_X\overline{\cF}^g$,
	\item[(ii)] $\overline{\alpha_X F}^g\cong \alpha_X\overline{F}^g$,
	\item[(iii)] $\overline{\beta_X\cF}^g\cong \beta_X\overline{\cF}^g$.
\end{itemize}
\end{lemma}

\begin{proof}
	Let us prove the statements in the non-derived case, i.e.\ for $\cF\in\Mod{L}{X}$ and $F\in \mathrm{I}(L_X)$. This is enough since the functors $\iota_X$, $\alpha_X$ and $\beta_X$ are exact.
	
	The functor $\iota_X$ is given by $$\iota_X\cF=\indlim{U\subset\subset X} \cF_U,$$ where $U$ ranges over the relatively compact open subsets of $X$. Therefore, by definition of conjugation for ind-sheaves, we obtain
	$$\overline{\iota_X\cF}^g=\indlim{U\subset\subset X}\overline{\cF_U}^g\cong \indlim{U\subset\subset X}(\overline{\cF}^g)_U=\iota_X(\overline{\cF}^g).$$
	This proves (i). Using this and the fact that $\alpha_X$ is left adjoint to $\iota_X$, we get for any $\cG\in \Mod{L}{X}$
	\begin{align*}
	\mathrm{Hom}_{\Mod{L}{X}}\big(\overline{\alpha_X F}^g,\cG\big)&\cong \mathrm{Hom}_{\Mod{L}{X}}\Big(\alpha_X F,\overline{\cG}^{g^{-1}}\Big)\cong \mathrm{Hom}_{\mathrm{I}(L_X)}\Big(F,\iota_X\overline{\cG}^{g^{-1}}\Big)\\
	&\cong \mathrm{Hom}_{\mathrm{I}(L_X)}\Big(F,\overline{\iota_X\cG}^{g^{-1}}\Big)\cong \mathrm{Hom}_{\mathrm{I}(L_X)}\Big(\overline{F}^g,\iota_X\cG\Big)\\
	&\cong \mathrm{Hom}_{\Mod{L}{X}}\big(\alpha_X\overline{F}^g,\cG\big),
	\end{align*}
	hence (ii) follows. Accordingly, one proves (iii), using that $\beta_X$ is the left adjoint of $\alpha_X$.
\end{proof}

Now, the category $\EbIL{\cX}$ of enhanced ind-sheaves on a bordered space $\cX=(X,\widehat{X})$ is a quotient category of $\DbIL{\widehat{X}\times\sP}$, and it can be checked that the above conjugation functor for ind-sheaves on $M=\widehat{X}\times\sP$ induces a well-defined functor
$$\EbIL{\cX}\to\EbIL{\cX}, H\mapsto \overline{H}^g.$$

The following lemma is not difficult to prove, and it shows that conjugation functors are compatible with many of the standard operations.

\begin{lemma}\label{lemmaConjCompat}
Let $\cX=(X,\widehat{X})$ and $\cY=(Y,\widehat{Y})$ be bordered spaces and $f\colon \cX\to \cY$ a morphism. Let $H,H_1,H_2\in \EbIL{\cX}$, $\cF\in\Db{L}{X}$ and $J\in \EbIL{\cY}$. Then for any $g,h\in G$ we have the following isomorphisms:
\begin{itemize}
\item[(i)] $\overline{\mathrm{E}f_*H}^g\cong \mathrm{E}f_*\overline{H}^g$ and $\overline{\mathrm{E}f_{!!}H}^g\cong \mathrm{E}f_{!!}\overline{H}^g$,
\item[(ii)] $\overline{\mathrm{E}f^{-1}J}^g\cong \mathrm{E}f^{-1}\overline{J}^g$ and $\overline{\mathrm{E}f^!J}^g\cong \mathrm{E}f^!\overline{J}^g$,
\item[(iii)] $\overline{H_1\conv H_2}^g\cong \overline{H_1}^g\conv \overline{H_2}^g$, $\overline{\RR\Ihomp(H_1,H_2)}^g\cong \RR\Ihomp(\overline{H_1}^g,\overline{H_2}^g)$ and $\overline{\pi^{-1}\cF\otimes H}^g\cong \pi^{-1}\overline{\cF}^g\otimes \overline{H}^g$,
\item[(iv)] $\overline{\overline{H}^g}^h\cong \overline{H}^{gh}$
\end{itemize}
\end{lemma}
\begin{proof}
(i)--(iv) Since the operations on enhanced ind-sheaves are induced by operations on ind-sheaves, this follows from Lemma \ref{lemmaConjCompatShv} and the corresponding statements for ind-sheaves.
\end{proof}

Conjugation associates to an object over $L$ a different object over $L$. On the other hand, given an arbitrary field extension $L/K$, one can extend scalars, starting from objects defined over $K$: Classically, if $V$ is a $K$-vector space, then $L\otimes_K V$ is a vector space over $L$. This construction naturally extends to sheaves: Let $\cF\in \Mod{K}{X}$ and denote by $L_X$ (resp.\ $K_X$) the constant sheaf with stalk $L$ (resp. $K$), then $L_X\otimes_{K_X}\cF\in \Mod{L}{X}$. Since tensor products over fields are exact, we obtain a functor
\begin{equation}\label{eq:extensionSh}
    \Db{K}{X}\to\Db{L}{X}, \quad \cF\mapsto L_X\otimes_{K_X}\cF.
\end{equation}
Similarly, if $F=\indlim{i\in I}\cF_i\in\Ind{K}{X}$ for $\cF_i\in\Modc{K}{X}$, then $L_X\otimes_{K_X}F\cong\iota_X(L_X)\otimes_{K_X}F\indlim{i\in I}(L_X\otimes_{K_X}\cF_i)\in \Ind{L}{X}$. (Note that $L_X$ means the ind-sheaf $\iota_X(L_X)$ here.) This again extends to a functor between derived categories, and it is compatible with some basic operations: We have the following lemma for ind-sheaves (where (i) and (ii) hold correspondingly for sheaves).
\begin{lemma}\label{IndlemmaExtensionCompatibility}
Let $L/K$ be a field extension and $f\colon X\to Y$ a morphism of topological spaces. Let $F\in\DbI{K}{X}$, $\cF\in \Db{K}{X}$ and $G\in\DbI{K}{Y}$. Then we have isomorphisms
\begin{itemize}
    \item[(i)] $ Rf_{!!}(L_X\otimes_{K_X}F)\cong L_Y\otimes_{K_Y} Rf_{!!}F $,
    \item[(ii)] $ f^{-1}(L_Y\otimes_{K_Y}G)\cong L_X\otimes_{K_X} f^{-1}G $,
    \item[(iii)] $\iota_X(L_X\otimes_{K_X}F)\cong L_X\otimes\iota_X(F)$,
    \item[(iv)]$\alpha_X(L_X\otimes_{K_X}\cF)\cong L_X\otimes\alpha_X(\cF)$,
    \item[(v)]$\beta_X(L_X\otimes_{K_X}F)\cong L_X\otimes\beta_X(F)$.
\end{itemize}
\end{lemma}
\begin{proof}
\begin{itemize}
    \item[(i)] Noting that $L_X\cong f^{-1}L_Y$, this follows from the projection formula for ind-sheaves (see the first isomorphism in \cite[Theorem 5.2.7]{KS01}).
\item[(ii)] This follows as above, using the second isomorphism in \cite[Proposition 4.3.2]{KS01}.
\item[(iii)--(v)] These follow from the commutation between tensor product and $\iota_X,\alpha_X,\beta_X$ (see \cite[Proposition 4.2.3, Proposition 4.2.12]{KS01}). Moreover, note that $\alpha_X\circ\iota_X=\id$ and $\beta_X(L_X)=\iota_X(L_X)$ (cf.\ \cite[p.\ 50]{KS01}).
\end{itemize}
\end{proof}

Using notation as in \eqref{eq:pitens}, the functor of extension of scalars for ind-sheaves on $\widehat{X}$ induces a functor
\begin{equation}\label{eq:extension}
\EbI{K}{\cX}\to\EbI{L}{\cX}, \quad H\mapsto \pi^{-1}L_X\otimes_{\pi^{-1}K_X} H
\end{equation}
for any bordered space $\cX=(X,\widehat{X})$. (We emphasize here the field over which we take the tensor product.)

\begin{definition}\label{def:KStructure}
We will say that an object $H\in\EbI{L}{\cX}$ has a $K$-structure if it is contained in the essential image of the functor \eqref{eq:extension}.\\ Similarly, we say that an object $\cF\in\Db{L}{X}$ has a $K$-structure if it is contained in the essential image of the functor \eqref{eq:extensionSh}.
\end{definition}

We will give a statement about compatibility of the functor \eqref{eq:extension} with the six operations on enhanced ind-sheaves below (Lemma~\ref{IndlemmaExtensionCompatibility}).
If we restrict our focus to $\dR$-constructible enhanced ind-sheaves, we can also prove compatibility of extension of scalars with the sheafification functor.

As a preparation, we prove two lemmas about compatibility with certain Hom functors.
The following statement was suggested to us by Takuro Mochizuki, to whom we are very grateful for this.

\begin{lemma}\label{lemma:MocRc}
Let $X$ be a real analytic manifold and let $\cF_1,\cF_2\in\mathrm{Mod}_{\rc}(K_X)$ be $\dR$-constructible sheaves with compact support. Then there is a natural isomorphism of $L$-vector spaces
$$L\otimes_K \mathrm{Hom}_{K_X}(\cF_1,\cF_2) \overset{\cong}{\longrightarrow} \mathrm{Hom}_{L_X}(L_X\otimes_{K_X}\cF_1,L_X\otimes_{K_X}\cF_2).$$
\end{lemma}
\begin{proof}
We only sketch the idea of proof here: $\dR$-constructible sheaves on real analytic manifolds can be considered as constructible sheaves on simplicial complexes (see \cite[Chap.\ VIII]{KS90}). By compactness of the support, finitely many strata are sufficient to describe the sheaves. One starts by proving the statement for constant sheaves on single strata and concludes by induction on the dimension of the strata.
\end{proof}
Thanks to the lemma above, it is easy to prove the following.

\begin{lemma}\label{lemmaIhomExtensions}
		Let $X$ be a real analytic manifold, $\cF\in\mathrm{D}^\mathrm{b}_{\rc}(L_X)$ and $G\in\mathrm{D}^{\mathrm{b}}_{\mathrm{suban}}(\mathrm{I}L_X)$ such that $\cF\cong L_X\otimes_{K_X}\cF_K$ for some $\cF_K\in\mathrm{D}^\mathrm{b}_{\rc}(K_X)$ and $G\cong L_X\otimes_{K_X}G_K$ for some $G_K\in\mathrm{D}^{\mathrm{b}}_{\mathrm{suban}}(\mathrm{I}K_X)$, then we have an isomorphism in $\DbI{L}{X}$
	$$\mathrm{R}\cI hom(\cF,G)\cong L_X\otimes_{K_X} \mathrm{R}\cI hom(\cF_K,G_K).$$
	In particular, for $\cF, \cG\in\mathrm{D}^\mathrm{b}_{\rc}(L_X)$ such that $\cF\cong L_X\otimes_{K_X}\cF_K$ and $\cG\cong L_X\otimes_{K_X}\cG_K$ for some $\cF_K, \cG_K\in\mathrm{D}^\mathrm{b}_{\rc}(K_X)$, there is an isomorphism in $\Db{L}{X}$
	$$\mathrm{R}\cH om(\cF,\cG)\cong L_X\otimes_{K_X} \mathrm{R}\cH om(\cF_K,\cG_K).$$
\end{lemma}

\begin{proof}
	Let us first assume that $\cF_K,\cG_K$ are concentrated in degree $0$. We can write $G_K=\indlim{i} \cG_i$ for some $\cG_i\in\mathrm{Mod}_{\rc}^c(K_X)$ (by definition of $\mathrm{I}_{\mathrm{suban}}(K_X)$) and hence $G=\indlim{i} L_X\otimes_{K_X} \cG_i$. Then by \cite[Corollary 4.2.8(iii) and Proposition 4.2.4]{KS01} we have
	\begin{equation}\label{IhomC}\mathcal{I}hom(\cF,G)\cong \indlim{i} \mathcal{I}hom(\cF,L_X\otimes_{K_X}\cG_i)\cong \indlim{i} \mathcal{H}om(L_X\otimes_{K_X}\cF_K,L_X\otimes_{K_X}\cG_i).\end{equation}
	Let us show
	\begin{equation}\label{HomC} \mathcal{H}om(L_X\otimes_{K_X}\cF_K,L_X\otimes_{K_X}\cG_i)\cong L_X\otimes_{K_X}\mathcal{H}om(\cF_K,\cG_i)\end{equation}
	Let $U$ be a suffiently small ball in $X$, then, denoting by $C$ the (compact) support of $\cG_i$, we have
	\begin{align*}
		\Gamma(U;\mathcal{H}om(L_X\otimes_{K_X}\cF_K,L_X\otimes_{K_X}\cG_i))&=\mathrm{Hom}_{L_U}(L_U\otimes_{K_U}\cF_K|_U,L_U\otimes_{K_U}\cG_i|_U)\\
		&\cong \mathrm{Hom}_{L_X}(L_X\otimes_{K_X}(\cF_K)_{C\cap U},L_X\otimes_{K_X}(\cG_i)_U)\\
		&\cong L\otimes_K \mathrm{Hom}_{K_X}((\cF_K)_{C\cap U},(\cG_i)_U)\\
		&\cong L\otimes_K \mathrm{Hom}_{K_U}(\cF_K|_U,\cG_i|_U)
	\end{align*}
where the second line follows by full faithfulness of extension by zero (from $U$ to $X$) and the third line from Lemma~\ref{lemma:MocRc}, since the sheaves $(\cF_K)_{C\cap U}$ and $(\cG_i)_U$ are compactly supported and $\dR$-constructible.
On the other hand, the tensor product $L_{X}\otimes_{K_X}\mathcal{H}om(\cF_K,\cG_i)$ is the presheaf associated to
$$U\mapsto L_X(U)\otimes_{K}\mathrm{Hom}_{K_U}(\cF_K|_U,\cG_i|_U)$$
On a basis of the topology consisting of small open balls $U$, we have $L_X(U)=L$ and this coincides with the above, so \eqref{HomC} follows.
Together with \eqref{IhomC} and the fact that $\indlim{}$ commutes with tensor products, this proves the non-derived version of the first isomorphism in lemma.

By deriving functors (note that $L_X\otimes_{K_X}(\bullet)$ is exact), one gets the first statement of the lemma in the derived case.

The second assertion follows now by applying the functor $\alpha_X$ since we have $\alpha_X\circ \RR \cI hom \cong \cH om$ and $\alpha_X$ commutes with tensor products (cf.\ \cite[Propositions 4.2.3 and 4.2.4]{KS01}
\end{proof}

We can now state the above-mentioned compatibility of the sheafification functor with extension of scalars.

\begin{corollary}\label{corSheafificationExtension}
Let $\cX$ be a real analytic bordered space and $H\in\EbRcI{L}{\cX}$ such that $H\cong \pi^{-1}L_X\otimes_{\pi^{-1}K_X} H_K$ for some $H_K\in\EbRcI{K}{\cX}$. Then
$$
\mathsf{sh}_\cX(H)\cong L_X\otimes_{K_X}\mathsf{sh}_\cX(H_K).
$$
\end{corollary}
\begin{proof}
By definition (see formula \eqref{eq:Sh}), one has
$$\mathsf{sh}_\cX(H)=\alpha_X j^{-1}\mathrm{R}\overline{\pi}_*\mathrm{R}\cI hom\big(\pi^{-1}L_X\otimes_{\pi^{-1}K_X} (K_{\{t\geq 0\}}\oplus K_{\{t\leq 0\}}),\pi^{-1}L_X\otimes_{\pi^{-1}K_X} \cL H_K\big).$$
By Lemma~\ref{lemmaIhomExtensions}, the functor $\mathrm{R}\cI hom$ is compatible with $K$-structures (since there is only a usual sheaf in the first component). Moreover, the functor $\alpha_X$ commutes with tensor products and $\mathrm{R}\overline{\pi}_*=\mathrm{R}\overline{\pi}_{!!}$ (note that $\overline{\pi}\colon \widehat{X}\times\mathsf{P}\to\widehat{X}$ is proper) as well as $j^{-1}$ commute with extension of scalars (similarly to Lemma~\ref{lemmaExtensionCompatibility}). This concludes the proof.
\end{proof}

Let us now study relations between extension of scalars and the seix operations for enhanced ind-sheaves: We have similar compatibilities as in Lemma~\ref{IndlemmaExtensionCompatibility}(i)--(ii) for the case of enhanced ind-sheaves. In the setting we are mostly interested in, namely when $\cX=(X,\widehat{X})$ is a bordered space attached to a complex algebraic variety (i.e.\ $\widehat{X}$ is a compactification of $X$) and all objects involved are $\dR$-constructible, we can also prove compatibilities for the functors $\EE f_*$, $\EE f^!$ and $\RR \cI hom^+$ by duality.

\begin{lemma}\label{lemmaExtensionCompatibility}
Let $L/K$ be a field extension and $f\colon \cX=(X,\widehat{X})\to \cY=(Y,\widehat{Y})$ a morphism of bordered spaces. Let $F,F_1,F_2\in\EbI{K}{\cX}$ and $G\in\EbI{K}{\cY}$. Then we have isomorphisms
\begin{itemize}
    \item[(i)] $\EE f_{!!}(\pi^{-1}L_X\otimes_{\pi^{-1}K_X}F)\cong \pi^{-1}L_Y\otimes_{\pi^{-1}K_Y}\EE f_{!!}F $,
    \item[(ii)] $\EE f^{-1}(\pi^{-1}L_Y\otimes_{\pi^{-1}K_Y}G)\cong \pi^{-1}L_X\otimes_{\pi^{-1}K_X}\EE f^{-1}G $,
    \item[(iii)] $(\pi^{-1}L_X\otimes_{\pi^{-1}K_X}F_1)\conv (\pi^{-1}L_X\otimes_{\pi^{-1}K_X}F_2)\cong \pi^{-1}L_X\otimes_{\pi^{-1}K_X} (F_1\conv F_2)$.
\end{itemize}
Assume now in addition that $X$ and $Y$ are complex manifolds, $X\subset \widehat{X}$ and $Y\subset \widehat{Y}$ are relatively compact. Let moreover $F,F_1,F_2\in\EbRcI{K}{\cX}$ and $G\in\EbRcI{K}{\cY}$. Then we have isomorphisms
\begin{itemize}
	\item[(iv)] $\DE{\cX}(\pi^{-1}L_X\otimes F)\cong \pi^{-1}L_X\otimes \DE{\cX}F$,
	\item[(v)] $\EE f_* (\pi^{-1}L_Y\otimes F)\cong \pi^{-1}L_X\otimes \EE f_* F$,
	\item[(vi)] $\EE f^! (\pi^{-1}L_Y\otimes G)\cong \pi^{-1}L_X\otimes \EE f^! G$,
	\item[(vii)] $\RR \cI hom^+ (\pi^{-1}L_X\otimes_{\pi^{-1}K_X}F_1, \pi^{-1}L_X\otimes_{\pi^{-1}K_X}F_2)\cong \pi^{-1}L_X\otimes_{\pi^{-1}K_X}\RR\cI hom^+(F_1,F_2)$.
\end{itemize}
\end{lemma}
\begin{proof}
\begin{itemize}
    \item[(i)] Since all the operations involved are induced by operations on $\DbI{K}{\cX\times\dR_\infty}$, it is enough to prove an isomorphism
$$\RR \tilde{f}_{!!}(\pi^{-1}L_X\otimes_{\pi^{-1}K_X}F)\cong \pi^{-1}L_Y\otimes_{\pi^{-1}K_Y}\RR \tilde{f}_{!!}F $$
for any $F\in \DbI{K}{\cX\times\dR_\infty}$, where $\tilde{f}=f\times\id_{\dR_\infty}$. Noting that $\pi^{-1}L_X\cong \tilde{f}^{-1}\pi^{-1}L_Y$, this follows from the projection formula for ind-sheaves on bordered spaces (see the first isomorphism in \cite[Proposition 3.3.13]{DK16}).
\item[(ii)] Similarly, here we need to prove an isomorphism
$$\RR \tilde{f}^{-1}(\pi^{-1}L_Y\otimes_{\pi^{-1}K_Y}G)\cong \pi^{-1}L_X\otimes_{\pi^{-1}K_X}\RR \tilde{f}^{-1}G $$
for any $G\in\DbI{K}{\cY\times\dR_\infty}$, and this follows as above, using the second isomorphism in \cite[Proposition 3.3.13]{DK16}.
\item[(iii)] The convolution product is defined on $\DbI{K}{\cX\times\dR_\infty}$ by
$$F_1\conv F_2\vcentcolon= \mathrm{R}\mu_{!!}(q_1^{-1}F_1\otimes q_2^{-1}F_2),$$
where the maps $q_1,q_2,\mu\colon \cX\times\dR_\infty^2\to\cX\times\dR_\infty$ are given by the projections and by addition of the real variables, respectively. Hence, one concludes as in (i) and (ii).

\item[(iv)] Since $X\subset \widehat{X}$ is relatively compact and $F$ is $\dR$-constructible, there exists $\cF\in \mathrm{D}^\mathrm{b}_{\rc}(K_{X\times \dR})$ such that $F\cong K^\EE_\cX\conv \cF$. Then, by \cite[Lemma 2.8.3]{EnhPerv} and the definition of the duality functor, we get (with $a$ denoting the involution $(x,t) \mapsto (x,-t) $):
	\begin{align*}
		\DE{\cX} (\pi^{-1}L_X\otimes F) &\cong \DE{\cX}\big(L^\EE_\cX \conv (L_{X\times\dR}\otimes \cF)\big)\cong L^\EE_\cX \conv a^{-1}\mathrm{D}_{X\times \dR}(L_{X\times\dR}\otimes \cF)\\
		&\cong L^\EE_\cX \conv a^{-1}\mathrm{R}\mathcal{H}om_{L_{X\times\dR}}(L_{X\times\dR}\otimes \cF,\omega_{X\times\dR}^L)\\
		&\cong L^\EE_\cX \conv a^{-1}\mathrm{R}\mathcal{H}om_{L_{X\times\dR}}(L_{X\times\dR}\otimes \cF,L_{X\times\dR}\otimes \omega_{X\times\dR}^K)\\
		&\cong L^\EE_\cX \conv a^{-1}\big(L_{X\times\dR}\otimes \mathrm{R}\mathcal{H}om_{K_{X\times\dR}}(\cF,\omega_{X\times\dR}^K)\big)\\
		&\cong \pi^{-1}L_\cX\otimes ( K^\EE_\cX \conv a^{-1} \mathrm{D}_{X\times\dR}\cF) \cong \pi^{-1}L_\cX\otimes \DE{\cX}F.
	\end{align*}
Here, $\omega_{X\times\dR}^L$ denotes the dualizing complex in $\DbL{X\times\dR}$ (and similarly for $K$), and the fourth isomorphism follows from the fact that $\omega_{X\times\dR}^L\cong L_{X\times\dR}[2d_X]$ (and similarly for $K$), where $d_X$ is the complex dimension of $X$, because $X$ is orientable. The fifth isomorphism follows from Lemma~\ref{lemmaIhomExtensions}.
\item[(v)] This follows from (i) and (iv), since under the assumption that $X\subset\widehat{X}$ and $Y\subset\widehat{Y}$ are relatively compact, $f$ is semi-proper (in the sense of \cite[Definition 2.3.5]{EnhPerv}) and hence $\EE f_*\cong \DE{\cY}\EE f_*\DE{\cX}$ and $\EE f_*$ preserves $\dR$-constructibility (see \cite[Proposition 3.3.3(iv)]{EnhPerv}).
\item[(vi)] This follows from (ii) and (iv), using $\EE f^!\cong \DE{\cX}\EE f^{-1}\DE{\cY}$ and the fact that $\EE f^{-1}$ preserves $\dR$-constructibility (see \cite[Proposition 3.3.3(iii)]{EnhPerv}).
\item[(vii)] This follows from (iii) and (iv), since $\RR\cI hom^+(\bullet,\bullet)\cong \DE{\cX}(\,\bullet\,\conv\DE{\cX}(\bullet))$ and the functor $\RR\cI hom^+$ preserves $\dR$-constructibility (see \cite[Proposition 4.9.13]{DK16}).
\end{itemize}
\end{proof}

Extension of scalars also has good properties in connection with the standard and perverse t-structures on enhanced ind-sheaves introduced in \cite{EnhPerv}.

\begin{proposition}\label{prop:tExactnessScalarExt}
	Let $\cX=(X,\widehat{X})$ be a bordered space. The functor
	$$\Phi_{L/K}\colon \EbI{K}{\cX}\to\EbIL{\cX}, H\mapsto \pi^{-1}L_X\otimes_{\pi^{-1}K_X} H$$
	is t-exact with respect to the standard t-structure on enhanced ind-sheaves, i.e.\ we have
	\begin{align*}
		\Phi_{L/K}(\EE^{\leq c}(\mathrm{I}K_\cX))\subset \EE^{\leq c}(\mathrm{I}L_\cX) \qquad \text{and}\qquad
		\Phi_{L/K}(\EE^{\geq c}(\mathrm{I}K_\cX))\subset \EE^{\geq c}(\mathrm{I}L_\cX).
	\end{align*}
	Moreover, if $X$ is a complex manifold and $X\subset\widehat{X}$ is relatively compact, then this functor is t-exact with respect to the perverse (generalized) t-structure on $\dR$-constructible enhanced ind-sheaves, i.e.\ it satisfies
	\begin{align*}
		\Phi_{L/K}({}^p\EE_{\rc}^{\leq c}(\mathrm{I}K_\cX))\subset {}^p\EE_{\rc}^{\leq c}(\mathrm{I}L_\cX) \qquad \text{and}\qquad
		\Phi_{L/K}({}^p\EE_{\rc}^{\geq c}(\mathrm{I}K_\cX))\subset {}^p\EE_{\rc}^{\geq c}(\mathrm{I}L_\cX).
	\end{align*}
\end{proposition}
\begin{proof}
	The compatibility with the standard t-structure is clear on the level of ind-sheaves (tensor products over fields are exact). Therefore, it also commutes with the induced standard t-structure on enhanced ind-sheaves since extension of scalars commutes with the functor $\cL$ (which is just a convolution product, $\cL(H)=K_{t\geq 0}\conv H$).
	
	For the second assertion, we recall the definition of the perverse (generalized) t-structure from \cite{EnhPerv}:
	Let $H\in\EbRcI{K}{\cX}$. Then one has $H\in {}^p\EE_{\rc}^{\leq c}(\mathrm{I}K_\cX)$ if and only if for any $k\in\dZ$ we have
	\begin{align*}
		&\EE i_{(X\setminus Z)_\infty}^{-1}H\in\EE^{\leq c+p(k)}(\mathrm{I}K_{(X\setminus Z)_\infty})&&\textnormal{for some closed subanalytic subset $Z\subset\cX$ of dimension smaller than $k$}\\
		&\EE i_{Z_\infty}^!\DE{\cX}H\in\EE^{\geq -c-\frac{1}{2}-p(k)-k}(\mathrm{I}K_{Z_\infty})&&\textnormal{for any closed subanalytic subset $Z\subset\cX$ of dimension at most $k$}
	\end{align*}
Here, $i_{(X\setminus Z)_\infty}$ and $i_{Z_\infty}$ are the embeddings and we refer to \cite{EnhPerv} for more details.
From this definition, we see that, since $\Phi_{L/K}$ commutes with duality and inverse images by Lemma~\ref{lemmaExtensionCompatibility} (note that $H$ is $\dR$-constructible), the statement reduces to the above compatibility with the standard t-structure. The definition of $H\in {}^p\EE_{\rc}^{\geq c}(\mathrm{I}K_\cX)$ is similar and the proof works along the same lines.
\end{proof}

In particular, the heart ${}^p\EE_{\rc}^0(\mathrm{I}K_\cX)\vcentcolon= {}^p\EE_{\rc}^{\leq 0}(\mathrm{I}K_\cX)\cap {}^p\EE_{\rc}^{\geq 0}(\mathrm{I}K_\cX)$ is preserved by extension of scalars (i.e.\ it is sent to the heart ${}^p\EE_{\rc}^0(\mathrm{I}L_\cX)$). Recall that these hearts are quasi-abelian categories.

\begin{corollary}\label{cor:scalarExt-kernel}
	Let $f$ be a strict morphism in ${}^p\EE_{\rc}^0(\mathrm{I}K_\cX)$, then we have
	\begin{align*}
		\Phi_{L/K}(\ker f)&\cong \ker \Phi_{L/K}(f),& \Phi_{L/K}(\mathop{\mathrm{coker}} f)&\cong \mathop{\mathrm{coker}} \Phi_{L/K}(f),\\
		\Phi_{L/K}(\im f)&\cong \im \Phi_{L/K}(f),& \Phi_{L/K}(\mathop{\mathrm{coim}} f)&\cong \mathop{\mathrm{coim}} \Phi_{L/K}(f).
	\end{align*}
\end{corollary}
\begin{proof}
	As follows from \cite[Lemma 4.3]{Br07} and Proposition~\ref{prop:tExactnessScalarExt}, the functor $\Phi_{L/K}$ (being triangulated and t-exact) transforms strict short exact sequences in ${}^p\EE_{\rc}^0(\mathrm{I}K_\cX)$ into strict short exact sequences in ${}^p\EE_{\rc}^0(\mathrm{I}L_\cX)$. This means that it is left and right exact in the sense of \cite{Sch99}, and hence that it preserves kernels and cokernels of strict morphisms (see Definitions 1.1.12, 1.1.17 and 1.1.18 in loc.~cit.). The statement about images and coimages is then easily derived since kernels and cokernels always determine strict morphisms (see \cite[Remark 1.1.2]{Sch99}).
\end{proof}

We now want to use the concept of $g$-conjugation in order to describe when an object defined over a field $L$ actually comes from an object over a subfield $K$ by extension of scalars. To this end, we introduce the following notion. This is inspired by the corresponding results for Galois descent of vector spaces in \cite{Conrad} (and the references therein), where the notion of $G$-structure on an $L$-vector space was (equivalently) formulated in terms of a semilinear action of the Galois group on the underlying $K$-vector space.

From now on, let $L/K$ be a finite Galois extension with Galois group $G$.

\begin{definition}\label{def:GStructure}
Let $X$ be a topological space and let $\cX$ be a bordered space. A $G$-structure on an object $F\in \Db{L}{X}$ (resp.\ $F\in\DbI{L}{X}$, $F\in \EbI{L}{\cX}$) consists of the data of an isomorphism in $\Db{L}{X}$ (resp.\ $\DbI{L}{X}$, $\EbI{L}{\cX}$)
$$\varphi_g\colon F\overset{\cong}{\longrightarrow} \overline{F}^g$$ for each $g\in G$ such that for any $g,h\in G$, we have $\varphi_{gh}=\overline{\varphi_g}^h\circ \varphi_{h}$. Here, $\overline{\varphi_g}^h\colon \overline{F}^h\overset{\cong}{\longrightarrow}\overline{\overline{F}^g}^h=\overline{F}^{gh}$ denotes the isomorphism induced by $\varphi_g$ via the $h$-conjugation functor $\overline{(\bullet)}^h$.
\end{definition}

The following three statements show that the existence of such a $G$-structure on objects concentrated in one degree (with respect to the standard t-structures, i.e.\ sheaves and ind-sheaves rather than complexes thereof) often implies the existence of a structure over the subfield $K$.

\begin{lemma}\label{lemmaGstructureSheaves}
	Let $X$ be a topological space and let $\cF\in\mathrm{Mod}(L_X)$ be a sheaf equipped with a $G$-structure. Then there exists $\cF_K\in\mathrm{Mod}(K_X)$ such that $\cF\cong L_X\otimes_{K_X} \cF_K$. Moreover, if $\cF$ is a local system of finite rank (resp.\ $\dR$-constructible), then $\cF_K$ is a local system of finite rank (resp.\ $\dR$-constructible).
\end{lemma}
\begin{proof}
	Recall that, for all $g\in G$, the underlying sheaf of $K$-vector spaces of $\cF$ and $\overline{\cF}^g$ is the same. Hence, each $\varphi_g\colon \cF\overset{\cong}{\longrightarrow}\overline{\cF}^g$ defines a $K$-linear automorphism of $\cF$. More precisely, we get a $K$-linear automorphism $\varphi_g^U\colon \cF(U)\overset{\cong}{\longrightarrow}\cF(U)$ for any open $U\subseteq X$ which in addition is $g$-semilinear, meaning that $\varphi_g^U(l\cdot v)=g(l)\cdot \varphi_g^U(v)$ for any $l\in L$ and $v\in\cF(U)$.
	
	Now we set for any open $U\subseteq X$
	$$\cF_K(U)\vcentcolon=\cF(U)^G\vcentcolon=\{v\in \cF(U)\mid \varphi_g^U(v)=v \; \text{for any $g\in G$}\}.$$
	This clearly defines a subsheaf of $\cF$. Moreover, the natural morphisms
	$$L\otimes_K \cF_K(U)\to \cF(U)$$
	are isomorphisms (see e.g.\ \cite[Theorem 2.14]{Conrad}). Consequently, the natural morphism of sheaves
	$$L_X\otimes_{K_X} \cF_K\to \cF$$
	is also an isomorphism (noting that $L_X\otimes_{K_X} \cF_K$ is the sheaf associated to the presheaf $U\mapsto L\otimes \cF_K(U)$). This proves the first assertion.
	
	Assume now that $\cF$ was a local system of finite rank $n$. Let $x\in X$ be an arbitrary point and choose a basis $v_1,\ldots,v_n$ of the stalk $(\cF_K)_x$. Denote by $\widetilde{v}_1,\ldots,\widetilde{v}_n$ the induced elements of the stalk $\cF_x$. Then, by the definition of the stalk and since $\cF$ is a local system, there exists a neighbourhood $U$ of $x$ such that the $v_i$ (resp.\ the $\widetilde{v}_i$) can be viewed as sections of $\cF_K$ (resp.\ $\cF$) on $U$ and such that the $\widetilde{v}_i$ induce a basis of $\cF_y$ for any $y\in U$. Consequently, the $v_i$ also form a basis of $(\cF_K)_y$ for all $y\in U$, and hence $\cF_K$ is a local system of rank $n$.
	
	The statement for $\dR$-constructible sheaves follows immediately since $\dR$-constructible sheaves are local systems on the elements of a stratification, and one can take the same stratification for $\cF_K$ and $\cF$ (recall that extension of scalars behaves nicely with respect to inverse images).
\end{proof}

\begin{lemma}\label{lemmaGstructureIndsheaves}
Let $X$ be a real analytic manifold and let $F\in\mathrm{I}_{\mathrm{suban}}(L_X)$ be a subanalytic ind-sheaf equipped with a $G$-structure. Then there exists $F_K\in\mathrm{I}_{\mathrm{suban}}(K_X)$ such that $F\cong L_X\otimes_{K_X} F_K$.
\end{lemma}
\begin{proof}
By \cite[Theorem 6.3.5]{KS01} (note that $\mathrm{I}_{\mathrm{suban}}(L_X)$ is denoted by $\mathrm{I}_{\rc}(L_X)$ in loc.~cit.), there is an equivalence of categories
$$\mathrm{I}_{\mathrm{suban}}(L_X)\simeq \Mod{L}{X_\mathrm{sa}},$$
where $X_{\mathrm{sa}}$ is the subanalytic site associated to $X$ (whose open sets are the subanalytic open subsets of $X$ and whose coverings are required to be locally finite) and $\Mod{L}{X_{\mathrm{sa}}}$ is the category of sheaves of $L$-vector spaces on this site. This equivalence is compatible with $g$-conjugation (where $g$-conjugation on subanalytic sheaves is -- similarly to sheaves on the ordinary topology -- defined by considering $g$-conjugation on sections over any subanalytic open set). Hence, via this equivalence $F$ corresponds to a sheaf $\fF\in \mathrm{Mod}(L_{X_{\mathrm{sa}}})$ of $L$-vector spaces on the subanalytic site, and we have a $G$-structure on $\fF$. Now, with the same arguments as in the proof of Lemma~\ref{lemmaGstructureSheaves} (considering for $U$ the open subanalytic subsets of $X$), we can show that there is a sheaf $\fF_K\in\mathrm{Mod}(K_{X_{\mathrm{sa}}})$ such that $\fF\cong L_X\otimes_{K_X}\fF_K$. Finally, we obtain $F_K$ as the ind-sheaf corresponding to $\fF_K$ via the equivalence
$$\mathrm{I}_{\mathrm{suban}}(K_X)\simeq \mathrm{Mod}(K_{X_{\mathrm{sa}}}).$$
Since this equivalence is compatible with tensor products (cf.\ \cite[Proposition 1.3.1]{Prelli}), we can conclude that $F\cong L_X\otimes_{K_X}F_K$.
\end{proof}

\begin{proposition}\label{prop:KStructure}
	Let $\cX=(X,\widehat{X})$ be a real analytic bordered space such that $X\subset \widehat{X}$ is relatively compact. Let $H\in \mathrm{E}^0(\mathrm{I}L_{\cX})\cap\EbRcI{L}{\cX}$ be equipped with a $G$-structure. Then there exists $H_K\in \mathrm{E}^0(\mathrm{I}K_{\cX})\cap\EbRcI{K}{\cX}$ such that $H\cong \pi^{-1}L_X\otimes_{\pi^{-1}K_X}H_K$.
\end{proposition}
	Let us remark that this proposition shows in particular that $H$ has a $K$-structure in the sense of Definition \ref{def:KStructure}, but it also shows that $H_K$ can be chosen to be $\dR$-construcible.
	
\begin{proof}
	We are given isomorphisms $\varphi_g\colon H\to \overline{H}^g$ for each $g\in G$ satisfying the compatibility conditions from Definition~\ref{def:KStructure}.
	We can therefore identify $H$ with the image of the morphism
	\begin{equation}\label{eq:prodConj}
		\bigoplus_{g\in G}\varphi_g\colon H\to \bigoplus_{g\in G} \overline{H}^g.
	\end{equation}
	
	Due to our assumptions that $X\subset \widehat{X}$ is relatively compact and $H$ is $\dR$-constructible (cf.\ \cite[Definition 3.3.1]{EnhPerv} and recall also \cite[Lemma 4.6.3]{DK16}), we can write $H=L^\EE_\cX\conv \cF$ for some $\cF\in \Modrc{L}{\cX\times \dR_\infty}$ satisfying $L_{\{t\geq 0\}}\conv \cF\cong \cF$.
	
	Since $X\subset \widehat{X}$ is relatively compact, we know from \cite[Proposition 4.7.9]{DK16} that a morphism $\varphi_g\colon L^\EE_\cX\conv \cF\to \overline{L^\EE_\cX\conv \cF}^g\cong L^\EE_\cX\conv \overline{\cF}^g$ (where we have used the natural $G$-structure on $L^\EE_\cX$ in the last isomorphism) comes from a morphism $\varphi'_{g,a}\colon L_{\{t\geq -a\}}\conv \cF\to \overline{\cF}^g$ after applying the functor $L^\EE_\cX\conv\bullet$ (and noting that $L^\EE_\cX\conv L_{\{t\geq -a\}}\cong L^\EE_\cX$) for some sufficiently large $a\geq 0$. Note that if $\varphi'_{g,a}$ induces $\varphi_g$ and $b\geq a$, then the induced map $\varphi'_{g,b}\colon L_{\{t\geq -b\}}\conv \cF\to L_{\{t\geq -a\}}\conv \cF\to \overline{\cF}^g$ also induces $\varphi_g$. Moreover, if two morphisms $\varphi'_{g,a}, \widetilde{\varphi}'_{g,a}\colon L_{\{t\geq -a\}}\conv \cF\to \overline{\cF}^g$ induce $\varphi_g$, then there exists $b\geq a$ such that the induced morphisms $\varphi'_{g,b}$ and $\widetilde{\varphi}'_{g,b}$ coincide. Consequently, if $a\geq 0$ is sufficiently large, then for all $g\in G$, there exists $\varphi'_{g,a}\colon L_{\{t\geq -a\}}\conv \cF\to\overline{\cF}^g$ inducing $\varphi_g$ and satisfying \begin{equation}\label{eq:compatPhi}
		\varphi'_{gh,2a}=\overline{\varphi'_{g,a}}^h\circ \varphi'_{h,a}
	\end{equation}
	Let us fix such an $a\geq 0$.
	
	For any $b\geq a$, the morphism \eqref{eq:prodConj} is therefore induced by
	\begin{equation}\label{eq:prodConj2}
		\varphi'_b\vcentcolon=\bigoplus_{g\in G} \varphi'_{g,b}\colon L_{\{t\geq -b\}}\conv \cF\to \bigoplus_{g\in G} \overline{\cF}^g.
	\end{equation}
	Let us define $\cP\vcentcolon=\im(\varphi'_b)\in\Modrc{L}{\cX\times\dR_\infty}$ and note that this image does not depend on the choice of $b\geq a$. Hence, we have $H\cong L^\EE_\cX\conv \cP$.
	
	We will now show that $\cP$ has a $G$-structure: Clearly, there is a natural $G$-structure on the right-hand side of \eqref{eq:prodConj2} given by an isomorphism for each $h\in G$
	$$\psi_h\colon \bigoplus_{g\in G} \overline{\cF}^g\overset{\sim}{\longrightarrow} \bigoplus_{g\in G} \overline{\cF}^{gh}\cong \overline{\bigoplus_{g\in G} \overline{\cF}^g}^h,$$
	where the first isomorphism is given by a permutation of summands and the second one is induced by the natural identification $\overline{\cF}^{gh}\cong \overline{\overline{\cF}^g}^h$.
	
	Let $h\in G$ and consider the morphism
	$$\psi_h^{-1}\circ \overline{\varphi'_b}^h\colon L_{\{t\geq -b\}}\conv \overline{\cF}^g\longrightarrow \bigoplus_{g\in G} \overline{\overline{\cF}^g}^h\overset{\sim}{\longrightarrow} \bigoplus_{g\in G} \overline{\cF}^g.$$
	Clearly, $\psi_h$ induces an isomorphism $\im(\psi_h^{-1}\circ \overline{\varphi'_b}^h)\cong \im(\overline{\varphi'_b}^h)=\overline{\cP}^h$. To obtain a $G$-structure on $\cP$, it therefore suffices to show that $\cP=\im(\psi_h^{-1}\circ \overline{\varphi'_b}^h)$.
	
	Starting from the compatibility \eqref{eq:compatPhi}, one can show that $\varphi'_{2a}=\psi_h^{-1}\circ \overline{\varphi'_a}^h\circ \varphi'_{h,a}$. Therefore, one gets $\cP=\im(\varphi'_{2a})\subseteq \im(\psi_h^{-1}\circ \overline{\varphi'_a}^h)$. Similarly, from \eqref{eq:compatPhi} (applied for the product $gh^{-1}$), we obtain $\varphi'_{2a}=\psi_{h^{-1}}^{-1}\circ \overline{\varphi'_a}^{h^{-1}}\circ \varphi'_{h^{-1},a}$, and this yields $\psi_{h}^{-1}\circ\overline{\varphi'_{2a}}^h= \varphi'_a \circ \overline{\varphi'_{h^{-1},a}}^{h}$. This gives the inclusion $\im(\psi_{h}^{-1}\circ\overline{\varphi'_{a}}^h)=\im(\psi_{h}^{-1}\circ\overline{\varphi'_{2a}}^h)\subset \im(\varphi'_a)=\cP$, as desired.
	
	It follows that $\cP$ has a $G$-structure and hence, by Lemma~\ref{lemmaGstructureSheaves}, there exists $\cP_K\in \Mod{K}{X\times\dR}$ such that $\cP\cong L_X\otimes_{K_X} \cP_K$. Note that one has $\cP_K\in\Modrc{K}{\cX\times\dR_\infty}$. In particular, $H_K\vcentcolon= K^\EE_\cX\conv \cP_K\in\mathrm{E}^0(\mathrm{I}K_{\cX})\cap\EbRcI{K}{\cX}$ is an object satisfying $H\cong \pi^{-1}L_X\otimes_{\pi^{-1}K_X}H_K$ and this proves the proposition.
\end{proof}

\subsection{$\cD$-modules and the enhanced Riemann--Hilbert correspondence}\label{sectionRH}

Let $X$ be a smooth complex algebraic variety. We denote by $\cD_X$ the sheaf of algebraic differential operators on $X$ and by $\mathrm{Mod}_{\mathrm{hol}}(\cD_X)$ the category of holonomic $\cD_X$-modules. Moreover, we denote by $\Dbhol{X}$ the subcategory of the derived category of $\cD_X$-modules consisting of complexes with holonomic cohomologies. For a morphism $f\colon X\to Y$, we will denote the direct and inverse image operations on $\cD_X$-modules by $f_+$ and $f^+$, respectively. The duality functor for $\cD_X$-modules is denoted by $\dD_X$.

Recall that the classical Riemann--Hilbert correspondence for regular holonomic $\cD$-modules gives an equivalence
$$\Sol{X}\colon \mathrm{D}^\mathrm{b}_{\mathrm{reghol}}(\cD_X)^\op\xlongrightarrow{\sim} \mathrm{D}^\mathrm{b}_{\dC\textrm{-c}}(\dC_X)$$
between the derived category of regular holonomic $\cD$-modules and the derived category of $\dC$-constructible sheaves of complex vector spaces on a smooth algebraic variety $X$ (and similarly on a complex manifold).

In \cite{DK16}, the authors established a generalization of this result to (not necessarily regular) holonomic $\cD$-modules, where the category of enhanced ind-sheaves serves as a target category for the Riemann--Hilbert functor. Concretely, for $X$ a complex manifold, they introduce a fully faithful functor
$$\SolE{X}\colon \Dbhol{X}^\op\hookrightarrow \EbRcI{\dC}{X}.$$

The theories of algebraic and analytic $\cD$-modules are often parallel, but differ in certain aspects. Analytification gives a way of associating to an algebraic $\cD$-module on a smooth algebraic variety $X$ an analytic $\cD$-module on the corresponding complex manifold $X^\mathrm{an}$. However, one generally needs to extend the algebraic $\cD$-module to a completion of $X$ first in order not to lose information during this procedure.

The details on an algebraic version of the Riemann--Hilbert correspondence for holonomic $\cD$-modules have been given in \cite{Ito} and we will briefly recall the construction here.

Let $X$ be a smooth complex algebraic variety. Then by classical results of Hironaka there exists a smooth completion $\wt{X}$, i.e.\ a smooth complete algebraic variety containing $X$ as an open subvariety, such that $\wt{X}\setminus X\subset \wt{X}$ is a normal crossing divisor. We denote by $X_\infty=(X,\widetilde{X})$ the (algebraic) bordered space thus defined and by $j\colon X\hookrightarrow\wt{X}$ the inclusion. Although the space $\widetilde{X}$ is not unique, the bordered space $X_\infty$ is determined up to isomorphism. Furthermore, we write $j_\infty^\mathrm{an}\colon X_\infty^\mathrm{an}=(X^\mathrm{an},\wt{X}^\mathrm{an})\to \wt{X}^\mathrm{an}=(\wt{X}^\mathrm{an},\wt{X}^\mathrm{an})$ for the natural morphism of bordered spaces given by the embedding.

\begin{theorem}[{cf.\ \cite[Theorem 3.12]{Ito}}]
Let $X$ be a smooth complex algebraic variety and $\widetilde{X}$ a smooth completion as above. Then the functor
\begin{align*}
\SolE{X_\infty}\colon \Dbhol{X}^\op&\longrightarrow \EbRcI{\dC}{X_\infty^\mathrm{an}}\\
\cM &\longmapsto \EE (j_\infty^\mathrm{an})^{-1} \SolE{\wt{X}^\mathrm{an}}\big((j_+\cM)^\mathrm{an}\big)
\end{align*}
is fully faithful.
\end{theorem}

We will often write $\EbRcI{\dC}{X_\infty}$ for $\EbRcI{\dC}{X_\infty^\mathrm{an}}$, since the category of enhanced ind-sheaves was only defined for analytic bordered spaces (algebraic vartieties are not good topological spaces), and hence there is no risk of confusion. If $f\colon X\to Y$ is a morphism of smooth complex algebraic varieties, we denote by $f_\infty\colon X_\infty^\mathrm{an}\to Y_\infty^\mathrm{an}$ the induced morphism of (analytic) bordered spaces. The enhanced solution functor satisfies many convenient compatibilities, which we summarize in the following lemma. We refer to \cite[Corollary 9.4.10]{DK16} (see also \cite[Proposition 3.13]{Ito}) for statements (i)--(iii). The fourth point easily follows combining \cite[Theorem 9.1.2(iv), Corollary 9.4.9, Lemma 4.3.2, Proposition 4.9.13]{DK16} (see also \cite[Corollary 7.7.8]{KSRegIrr}).

\begin{lemma}\label{lemmaAlgRHCompat}
Let $f\colon X\to Y$ be a morphism of smooth complex algebraic varieties.
\begin{itemize}
    \item[(i)] Let $\cM\in\Dbhol{X}$, then there is an isomorphism in $\EbI{\dC}{Y_\infty}$
    $$\SolE{Y_\infty}(f_+\cM)\cong \EE f_{\infty !!} \SolE{X_\infty}(\cM)[d_X-d_Y],$$ where $d_X$ and $d_Y$ are the (complex) dimensions of $X$ and $Y$, respectively.
    \item[(ii)] Let $\cN\in\Dbhol{Y}$, then there is an isomorphism in $\EbI{\dC}{X_\infty}$
    $$\SolE{X_\infty}(f^+\cN)\cong \EE f_\infty^{-1}\SolE{Y_\infty}(\cN).$$
    \item[(iii)] Let $\cM\in\Dbhol{X}$, then there is an isomorphism in $\EbIC{X_\infty}$
    $$\SolE{X_\infty}(\dD_X\cM)\cong \DE{X_\infty}\SolE{X_\infty}(\cM)[-2d_X],$$
    where $d_X$ is the (complex) dimension of $X$.
    \item[(iv)] Let $\cM\in\Dbhol{X}$ and $\cR\in\mathrm{D}^\mathrm{b}_\mathrm{reghol}(\cD_X)$, then there is an isomorphism in $\EbI{\dC}{X_\infty}$
    $$\SolE{X_\infty}(\cR\otimes_{\cO_X}\cM)\cong \pi^{-1}\Sol{X}(\cR)\otimes \SolE{X_\infty}(\cM).$$
\end{itemize}
\end{lemma}

Let $f$ be an algebraic function on $X$. Denote by $\cE^f\in\mathrm{Mod}_{\mathrm{hol}}(\cD_X)$ the algebraic $\cD_X$-module associated to differential operators with solutions $e^f$. Then there is an isomorphism
$$\SolE{X_\infty}(\cE^f)\cong \dE^{\real f}_\dC.$$

Moreover, the usual holomorphic solutions of an algebraic holonomic $\cD_X$-module $\cM$ are recovered from its enhanced solutions via the sheafification functor (see \cite[Lemma 3.16]{Ito}):
$$\mathsf{sh}_{X_\infty}\big(\SolE{X_\infty}(\cM)\big)\cong \Sol{X}(\cM).$$

For later use, we state the following lemma, applying a result of T.\ Mochizuki about the enhanced solutions of meromorphic connections to the algebraic enhanced solution functor.
\begin{lemma}\label{lemmaMochizukiDegreeZero}
Let $X$ be a smooth complex algebraic variety and let $\cM\in\mathrm{Mod}_{\mathrm{hol}}(\cD_X)$ be an integrable connection on $X$ (i.e.\ a holonomic $\cD_X$-module, locally free as an $\cO_X$-module). Then $\SolE{X_\infty}(\cM)\in \mathrm{E}^0(\mathrm{I}\dC_{X_\infty})$.
\end{lemma}
\begin{proof}
If $\cM$ is such a module, then $(j_+\cM)^\mathrm{an}\in \mathrm{Mod}_{\mathrm{hol}}(\cD_{\widetilde{X}^\mathrm{an}})$ is a meromorphic connection with poles on $D\vcentcolon=\widetilde{X}^\mathrm{an}\setminus X^\mathrm{an}$., i.e. $\mathrm{SingSupp}((j_+\cM)^\mathrm{an})\subseteq D$ and $(j_+\cM)^\mathrm{an}\cong (j_+\cM)^\mathrm{an}\otimes_{\cO_{\widetilde{X}^\mathrm{an}}}\cO_{\widetilde{X}^\mathrm{an}}(*D)$. Combining results from \cite[§9, Corollary 5.21]{MocCurveTest}, it follows that $\SolE{\widetilde{X}^\mathrm{an}}((j_+\cM)^\mathrm{an})\in \mathrm{E}^0(\mathrm{I}\dC_{\widetilde{X}^\mathrm{an}})$. Since $\EE (j_\infty^\mathrm{an})^{-1}$ is exact with respect to the standard t-structures on $\EbI{\dC}{\widetilde{X}^\mathrm{an}}$ and $\EbI{\dC}{X_\infty}$ (see \cite[Proposition 2.7.3(iv)]{EnhPerv}), the assertion follows.
\end{proof}

Finally, let us recall that the functor $\SolE{X_\infty}(\bullet)[d_X]$ (where $d_X$ is the complex dimension of $X$) is exact with respect to the standard t-structure on $\Dbhol{X}$ and the middle perversity generalized t-structure on $\EbRcI{\dC}{X_\infty}$. (Recall the notation from Section \ref{sect:enhanced}.) In particular, it sends $\mathrm{Mod}_{\mathrm{hol}}(\cD_X)$ to ${}^{1/2}\EE_{\rc}^0(\mathrm{I}\dC_{X_\infty})$.

We prove the following lemma for later use. Recall that a morphism $h$ in a quasi-abelian category is called \emph{strict} if the canonical morphism $\mathop{\mathrm{coim}} h \to \mathop{\mathrm{im}} h$ is an isomorphism. (We refer to \cite{Sch99} for a detailed study of quasi-abelian categories.)

\begin{lemma}\label{lemma:SolEstrictness}
	Let $g\colon \cM\to\cN$ be a morphism of holonomic $\cD_X$-modules. Then the morphism $$\SolE{X_\infty}(g)[d_X]\colon \SolE{X_\infty}(\cN)[d_X]\to\SolE{X_\infty}(\cM)[d_X]$$
	is strict in the quasi-abelian category ${}^{1/2}\EE_{\rc}^0(\mathrm{I}\dC_{X_\infty})$.
\end{lemma}
\begin{proof}
	Since $\SolE{X_\infty}(\bullet)[d_X]$ is exact with respect to the standard t-structure on $\Dbhol{X}$ and the middle perversity generalized t-structure on $\EbRcI{\dC}{X_\infty}$, it follows (cf.\ \cite[Lemma 4.3]{Br07}) that it sends strict short exact sequences in $\mathrm{Mod}_{\mathrm{hol}}(\cD_X)$ to strict short exact sequences in ${}^{1/2}\EE_{\rc}^0(\mathrm{I}\dC_{X_\infty})$. In particular, it is exact in the sense of \cite{Sch99} and hence preserves kernels and cokernels of arbitrary morphisms in $\mathrm{Mod}_{\mathrm{hol}}(\cD_X)^{\op}$, because the latter category is abelian and hence any morphism is strict. It follows that it also preserves images and coimages.
	
	Now let $g$ be as above. Then the natural isomorphism $\mathop{\mathrm{coim}} g\to \mathop{\mathrm{im}} g$ in $\mathrm{Mod}_{\mathrm{hol}}(\cD_X)$ is mapped to a morphism
	$$\SolE{X_\infty}(\mathop{\mathrm{im}} g)[d_X]\to \SolE{X_\infty}(\mathop{\mathrm{coim}} g)[d_X],$$
	which is still an isomorphism and coincides with the natural morphism
	$$\mathop{\mathrm{coim}}\SolE{X_\infty}(g)[d_X]\to \mathop{\mathrm{im}}\SolE{X_\infty}(g)[d_X].$$
	Hence, $\SolE{X_\infty}(g)[d_X]$ is strict.
\end{proof}

\section{Hypergeometric $\cD$-modules and families of Laurent polynomials}
\label{sec:Laurent}

We briefly recall a few fundamental facts on one-dimensional hypergeometric $\cD$-modules. Standard references are \cite{Ka}, etc. (one may follow \cite[section 2]{SevCast}).
As a matter of notation, we denote by $\dG_m$ a $1$-dimensional algebraic torus, if we want to fix a coordinate, say, $q$, on it, we also write $\Gm$.
\begin{definition}
Let $m,n\in \dZ_{\geq 0}$
and let $\alpha_1,\ldots,\alpha_n,\beta_1,\ldots,\beta_m\in \dC$ be given. Consider the differential operator in one variable
$$
P:=\prod_{i=1}^{n}\left(q\partial_q-\alpha_i\right)-q\cdot \prod_{j=1}^{m}\left(q\partial_q-\beta_j\right),
$$
and the left $\cD_{\Gm}$-module
$$
\cH(\alpha;\beta):=\cD_{\Gm}/\cD_{\Gm} \cdot P \in \mathrm{Mod}_\mathrm{hol}(\cD_{\Gm}).
$$
$\cH(\alpha;\beta)$  is called one-dimensional (or univariate) hypergeometric $\cD$-module.
\end{definition}
\begin{remark}
We will
suppose from now on that all $\alpha_i$ and all $\beta_j$ are real numbers. Although this is not strictly necessary for what follows, it simplifies some arguments, and corresponds to the cases of interest, specifically if one studies Hodge properties of hypergeometric systems.
\end{remark}
We will mainly be concerned with the case where $n\neq m$ and where the system $\cH(\alpha;\beta)$ is irreducible. In this case, we have the following important fact due to Katz.
\begin{proposition}[{see \cite[Proposition 2.11.9, Proposition 3.2]{Ka}}]\label{prop:PropHyp}
Suppose that $n\neq m$. Let $\alpha_i,\beta_j\in \dR$ be given, and consider the $\cD_{\Gm}$-module $\cH(\alpha;\beta)$ as defined above. Then:
\begin{enumerate}
    \item $\cH(\alpha;\beta)$ is irreducible if and only if
    for all $i\in\{1,\ldots,n\}$
and $j\in \{1,\ldots,m\}$ we have $\beta_j-\alpha_i\notin \dZ$.
    \item If $\cH(\alpha;\beta)$ is irreducible, then for any $k,l\in \dZ$, for any $i\in\{1,\ldots,n\}$
and for any $j\in \{1,\ldots,m\}$ we have that
$$
\cH(\alpha;\beta) \cong
\cH(\alpha_1,\ldots,\alpha_i+k,\ldots,\alpha_n;
\beta_1,\ldots,\beta_j+l,\ldots,\beta_m).
$$
Hence, for irreducible hypergeometric modules $\cH(\alpha;\beta)$  we may assume, up to isomorphism, that $\alpha_i,\beta_j\in [0,1)$ with $\beta_j\neq\alpha_i$ for all $i,j$.
\end{enumerate}
\end{proposition}

A major step in our approach to the existence of Betti structures is a geometric realization of $\cH(\alpha;\beta)$ via a family of Laurent polynomials.
We will prove our main results in several steps, first under some special assumptions on the numbers $\alpha_i$ and $\beta_j$. Namely, suppose that  $n > m$ and
that $\alpha_1=0$. This latter hypothesis is
a technical but crucial assumption to obtain the geometric realization of $\cH(\alpha;\beta)$. We write
$\alpha=(0,\alpha_2,\ldots,\alpha_n)\in \dR^n$,
$\beta=(\beta_1,\ldots,\beta_m)\in \dR^m$ and we define $\gamma=(\gamma_1,\ldots,\gamma_{N-1}):=(\beta_1,\ldots,\beta_m,\alpha_2,\ldots,\alpha_n)\in \dR^{N-1}$,
where $N=n+m$.

Here and later we will use twisted structure sheaves on algebraic tori, defined as follows.
Let $k,l\in \dN_0$ be arbitrary
(but excluding $(k,l)=(0,0)$), and consider any vector
$\gamma=(\gamma_1,\ldots,\gamma_k)\in \dR^k$. We consider the torus $\dG_m^{k+l}$ and the module
\begin{equation}\label{eq:DefOG}
\cO_{\dG_m^{k+l}}^\gamma:=\cO_{\dG_m^k}^\gamma\boxtimes\cO_{\dG_m^l},
\quad \quad \quad \textup{where}
\quad \quad \quad \quad
\cO_{\dG_m^k}^\gamma := \cD_{\dG_m^k}/\left(\partial_{x_i} x_i +\gamma_i\right)_{i=1,\ldots, k}
\end{equation}
In particular, when $N=n+m$ is as above, we put $\dG=\dG_m^{N-1}\times \Gm$ and we consider the sheaves
$\cO_{\dG_m^{N-1}}^\gamma$ and $
\cO_\dG^\gamma=\cO_{\dG^{N-1}}^\gamma\boxtimes \cO_{\Gm}$.
Then the following holds
\begin{proposition}\label{prop:HypGMSystem}
Let $f:=x_1+\ldots+x_m+\frac{1}{x_{m+1}}+\ldots+ \frac{1}{x_{N-1}}+q\cdot x_1\cdot\ldots\cdot x_{N-1}\in \cO_\dG$. Consider the
elementary irregular module $\cE^{\gamma,f}:=\cO^\gamma_\dG\cdot e^f=\cO^\gamma_\dG\otimes_{\cO_\dG}\cE^f$, where
$\gamma=(\beta_1,\ldots,\beta_m,\alpha_2,\ldots,\alpha_n)\in[0,1)^{N-1}$ is as above (i.e., $n>m$ and such that $\alpha_1=0$ and that $\alpha_i\neq \beta_j$ for all $i\in\{1,\ldots,n\}$ and $j\in\{1,\ldots,m\}$).
Write $p:\dG\twoheadrightarrow\Gm$
for the projection to the last factor.

Then we have $\cH^i p_+ \cE^{\gamma,f} =0$ for all $i\neq 0$ and
\begin{equation}\label{eq:FourierHypergeom}
\cH(\alpha;\beta)=\cH(0,\alpha_2,\ldots,\alpha_n;\beta_1,\ldots,\beta_m)\cong\kappa^+\cH^0 p_+ \cE^{\gamma,f}
=\kappa^+  p_+ \cE^{\gamma,f},
\end{equation}
where $\kappa:\Gm \stackrel{\cong}{\longrightarrow} \Gm$ sends
$q$ to $(-1)^m\cdot q$.
\end{proposition}
As a preliminary step towards the proof of this statement, we have the following result on a torus embedding.
\begin{lemma}\label{lem:ApplySchulzeWalther}
Let $\gamma\in \dR^{N-1}$ be arbitrary, and consider the torus embedding
$$
\begin{array}{rcl}
h:\dG^{N-1}_m & \longrightarrow & \dA^N \\
(x_1,\ldots,x_{N-1}) & \longmapsto &
(x_1\cdot\ldots\cdot x_{N-1},\frac{1}{x_{m+1}},\ldots,\frac{1}{x_{N-1}},x_1,\ldots,x_m)=:(y_1,\ldots,y_N).
\end{array}
$$
Then there
exists an integer vector $c\in\dZ^{N-1}$ such that, writing $\widetilde{\gamma}:=\gamma+c$, we have an isomorphism of left
$\cD_{\dA^N}$-modules:
$$
\begin{array}{rcl}
h_+ \cO^{\widetilde{\gamma}}_{\dG_m^{N-1}} &\cong&
\frac{\D \cD_{\dA^N}}{\D \left(\Box, (E_{m+i-1}+\widetilde{\gamma}_{m+i-1})_{i=2,\ldots,n},
(E_j+\widetilde{\gamma}_j)_{j=1,\ldots,m}\right)} \\ \\
&=&
\frac{\D \cD_{\dA^N}}{\D \left(\Box, (E_{m+i-1}+\alpha_i+c_{m+i-1})_{i=2,\ldots,n},
(E_j+\beta_j+c_j)_{j=1,\ldots,m}\right)},
\end{array}
$$
where
$$
\begin{array}{rclccl}
\Box & := & y_1\cdot\ldots\cdot y_n-y_{n+1}\cdot\ldots\cdot y_N \\
E_{m+i-1} & := & \partial_{y_1}y_1 - \partial_{y_i} y_i &\quad& i=2,\ldots,n\\
E_j & := & \partial_{y_1}y_1+\partial_{y_{n+j}}y_{n+j}&\quad& j=1,\ldots,m
\end{array}
$$
\end{lemma}
\begin{proof}
Since $h$ is an affine map and an embedding, the functor $h_+$ is exact. Moreover, notice that $h$ factors as $h=h_2\circ h_1$, where
$h_1:\dG_m^{N-1}\hookrightarrow \dG_m^N$ is the closed (monomial) embedding defined exactly as $h$, i.e.\
sending $(x_1,\ldots,x_{N-1})$ to
$(x_1\cdot\ldots\cdot x_{N-1},1/x_{m+1},\ldots,1/x_{N-1},x_1,\ldots,x_m)$,
and where $h_2:\dG_m^N\hookrightarrow \dA^N$ is the canonical open embedding. Then it is an easy exercise using the definition of the direct image (i.e.\ the explicit expression via transfer modules) to show that
$$
h_{1,+} \cO^{\widetilde{\gamma}}_{\dG_m^{N-1}} \cong
\frac{\D \cD_{\dG_m^N}}{\D \left(\Box, (E_{m+i-1}+\alpha_i+c_{m+i-1})_{i=2,\ldots,n},
(E_j+\beta_j+c_j)_{j=1,\ldots,m}\right)}.
$$
Notice that this isomorphism holds for any integer vector $c\in \dZ^{N-1}$ and moreover, the modules $h_{1,+} \cO^{\widetilde{\gamma}}_{\dG_m^{N-1}}$ are all isomorphic (i.e., independent of the choice of $c$) since we have
$\cO^{\gamma}_{\dG_m^{N-1}}\cong \cO^{\widetilde{\gamma}}_{\dG_m^{N-1}}$ for any $c\in \dZ^{N-1}$.

It thus remains to show that
\begin{equation}\label{eq:DirImOpEmbed}
\begin{array}{c}
h_{2,+}
\left(
\frac{\D \cD_{\dG_m^N}}{\D \left(\Box, (E_{m+i-1}+\alpha_i+c_{m+i-1})_{i=2,\ldots,n},
(E_j+\beta_j+c_j)_{j=1,\ldots,m}\right)}
\right)
\\ \\ \cong
\frac{\D \cD_{\dA^N}}{\D \left(\Box, (E_{m+i-1}+\alpha_i+c_{m+i-1})_{i=2,\ldots,n},
(E_j+\beta_j+c_j)_{j=1,\ldots,m}\right)}.
\end{array}
\end{equation}
Since $h_2$ is an open embedding (over the complement
$\dA^N\backslash \dG_m^N=\left\{y_1\cdot\ldots\cdot y_N=0\right\}$), in order to prove equation \eqref{eq:DirImOpEmbed}, it suffices to show that for all $k\in\{1,\ldots, N\}$, left multiplication with $y_k$ is invertible on the module
$$
\frac{\D \cD_{\dA^N}}{\D \left(\Box, (E_{m+i-1}+\alpha_i+c_{m+i-1})_{i=2,\ldots,n},
(E_j+\beta_j+c_j)_{j=1,\ldots,m}\right)}.
$$
For this we use \cite[Theorem 3.6.]{SchulWalth2}, together with \cite[Lemma 1.10]{Reich2}. Consider the matrix
\begin{equation}\label{eq:MatrixA}
A=\left(\begin{array}{c|c|c}
\underline{1}_m & \underline{0}_{m\times(n-1)} & \operatorname{Id}_m\\[3pt]
\hline
 & & \vspace{-10pt}\\
\underline{1}_{n-1} & -\operatorname{Id}_{n-1} & \underline{0}_{(n-1)\times m}\end{array}
\right)
\end{equation}
and notice that its columns are the exponents of the monomial components of the map $h_1$. We will write below $\dN A$ resp. $\dR_{\geq 0}A$ for the monoid resp. the cone generated in $\dZ^{N-1}$ resp. in $\dR^{N-1}$ by the columns of the matrix $A$.

It follows then from
\cite[Theorem 3.6.]{SchulWalth2} and \cite[Lemma 1.10]{Reich2} that there is a vector $\delta_A\in \dN A$ such that for all $\gamma'\in \delta_A+\dR_{\geq 0}A$, multiplication with
$y_k$ for $k\in\{1,\ldots,N\}$ is invertible on the module $
\cD_{\dA^N}/(\Box,(E_{m+i-1}+\gamma'_{m+i-2}),(E_j+\gamma'_j))$. Now since $\dR_{\geq 0} A$ is a cone in $\dR^{N-1}$,
it is clear by an elementary topological argument that we can find a $c\in \dZ^{N-1}$ such that
$\widetilde{\gamma}=\gamma+c \in \delta_A+\dR_{\geq A}$. Notice that it follows from the normality of the semigroup $\dN A$, using \cite[Lemma 1.11]{Reich2},
that we can actually take $\delta_A$ to be zero, but
we will not use this fact here.

We thus obtain that for this choice of $c\in \dZ^{N-1}$, multiplication by $y_k$ for $k\in\{1,\ldots,N\}$ is invertible on
$ \cD_{\dA^N}/\left(\Box, (E_{m+i-1}+\widetilde{\gamma}_{m+i-1})_{i=2,\ldots,n},
(E_j+\widetilde{\gamma}_j)_{j=1,\ldots,m}\right)$,
which implies that
$$
h_+ \cO_{\dG_m^{N-1}}^{\widetilde{\gamma}}
\cong
\frac{\D \cD_{\dA^N}}{\D \left(\Box, (E_{m+i-1}+\widetilde{\gamma}_{m+i-1})_{i=2,\ldots,n},
(E_j+\widetilde{\gamma}_j)_{j=1,\ldots,m}\right)}.
$$

\end{proof}
We need another preparation concerning an important property of the Laurent polynomial $f$ introduced in the statement of Proposition \ref{prop:HypGMSystem}. In order to formulate it, consider more generally the function $F:=\sum_{i=1}^N
\lambda_i\cdot\underline{x}^{\underline{a}_i}$, where
$\underline{a}_1,\ldots,\underline{a}_N$ are the columns of the matrix $A$ from equation \eqref{eq:MatrixA}. Notice that then $f\in \cO_\dG$ is the specialization of $F$ by setting
$\lambda_1=\ldots=\lambda_{N-1}=1$ and $\lambda_N=q$.
\begin{lemma}\label{lem:NonDeg}
For any $\underline{\lambda}\in \dG_m^N$, the function $F(-,\underline{\lambda}) \in \cO_{\dG_m^{N-1}}$ is non-degenerate in the sense of \cite[p.\ 274]{Adolphson}. In particular, $f$ is non-degenerate for any $q\in \Gm$.
\end{lemma}
\begin{proof}
Let us recall the notion of non-degenerateness: Let $\Delta:=\Conv\left(\underline{0},\underline{a}_1,\ldots,\underline{a}_N\right)$ be the convex hull of the vectors $\underline{a}_1,\ldots,\underline{a}_N$ as well as the origin. Then if $\tau\subset \Delta$ is any proper face of $\Delta$ that does not contain $\underline{0}$, we have to show that the Laurent polynomial
$$
F_\tau:=\sum_{\underline{a}_i\in\tau} \lambda_i \cdot \underline{x}^{\underline{a_i}}
$$
has no critical points on $\dG_m^{N-1}$.
Notice that in \cite[Definition 3.8]{ReiSe2}, this property is referred to as having no \emph{bad singularities at infinity}.

In order to prove that this property holds for the functions $F(-,\underline{\lambda})$, we employ an argument from toric geometry: First notice that the vectors $\underline{a}_1,\ldots,\underline{a}_N$ are the primitive integral generators of the rays of the fan of the (non-compact) $(N-1)$-dimensional toric variety $Y_\Sigma:=\dV(\cO_{\dP^{n-1}}(-1)^m)$ (the total space of the
vector bundle which is the direct sum of $m$ copies of the line bundle $\cO_{\dP^{n-1}}(-1)$ on $\dP^{n-1}$).
First notice that since the duals of the bundles $\cO_{\dP^{n-1}}(-1)$ are ample (nef would be sufficient), it follows that $\Supp(\Sigma)$ is convex. This implies
(see \cite[Lemma 5.3]{ReiSe2}) that $\Supp(\Sigma)=\dR_{\geq 0} A$.
In particular, consider the cone $C(\tau)$ over $\tau$, that is
$$
C(\tau):=\dR_{\geq 0} \tau:=\sum_{\underline{a}_i\in \tau} \dR_{\geq 0} \underline{a}_i,
$$
then we have $C(\tau)\subset \Supp(\Sigma)$ (here we use that $\underline{0}$ is not contained in $\tau$).
It follows that $C(\tau)$ is a union of cones of the fan $\Sigma$. We claim that it is actually equal to a single cone of $\Sigma$. Namely, since we have $n>m$, the variety $Y$ is Fano, i.e.\ its anti-canonical class is ample (see e.g.\ \cite[Section 4.2]{ReiSe2} for a summary of the statements needed here). In terms of the toric data of $Y$ (i.e.\ the fan $\Sigma$), this property means that the piecewiese linear function $\psi^{\Sigma}_{-K_Y}$ on $\textup{Supp}(\Sigma)\subset \dR^{N-1}$ defined by $-K_Y$ (which is the sum of the toric invariant divisors on $Y$) is strictly convex.
This implies that no cone $\sigma\in \Sigma$ can be  strictly contained in $C(\tau)$, hence $C(\tau)$ must itself be a cone in $\Sigma$.

Now the argument proceeds as in \cite[Lemma 2.8]{ReiSe}: Suppose that $\tau$ is a face of $\Delta$ (not containing $\underline{0}$) of dimension $s-1$, so that $C(\tau)$ is an $s$-dimensional cone. Write $\underline{a}_{\tau_1},\ldots,\underline{a}_{\tau_s}$ for those vectors $\underline{a}_i$ that appear in $\tau$.
Then since $\Sigma$ is a smooth fan, we know that
$\underline{a}_{\tau_1},\ldots,\underline{a}_{\tau_s}$ are linearly independent. Now if $(x_1,\ldots,x_{N-1})\in \dG_m^{N-1}$ was a critical point of $F_\tau(\underline{x},\underline{\lambda})$,
then the equation
$$
A_\tau\cdot
\begin{pmatrix}
\lambda_{\tau_1}\cdot \underline{x}^{\underline{a}_{\tau_1}} \\
\vdots \\
\lambda_{\tau_s}\cdot \underline{x}^{\underline{a}_{\tau_s}} \\
\end{pmatrix}
=0
$$
(where $A_\tau=\left(\underline{a}_{\tau_1}\,|\ldots|\,
\underline{a}_{\tau_s}\right)$)
would have a non-trivial solution, but this is impossible
since the matrix $A_\tau$ has maximal rank but the entries  $\lambda_{\tau_j}\underline{x}^{\underline{a}_{\tau_j}}$ lie in $\dG_m$, i.e.\ they are non-zero.
\end{proof}
With these preparations, we can give a proof of the proposition by applying some twisted version of a Fourier--Laplace transformation.
\begin{proof}[Proof of Proposition \ref{prop:HypGMSystem}]

We first show the isomorphism
$$
\kappa^+\cH^0 p_+ \cE^{\gamma,f} \cong\cH(\alpha;\beta).
$$
Notice that given $\gamma=(\beta_1,\ldots,\beta_m,\alpha_2,\ldots,\alpha_n)\in [0,1)^{N-1}$, we can find an integer vector $c\in \dZ^{N-1}$ as in the previous lemma \ref{lem:ApplySchulzeWalther}, and consider
$$
\widetilde{\gamma}=\gamma+c=:
(\widetilde{\beta}_1,\ldots,\widetilde{\beta}_m,\widetilde{\alpha}_2,\ldots,\widetilde{\alpha}_n).
$$
Then since $\cH(\alpha;\beta)$ is irreducible, we have an isomorphism of $\cD_{\Gm}$-modules
$$
\cH(\alpha;\beta) \cong \cH(\widetilde{\alpha};\widetilde{\beta})
$$
by Proposition \ref{prop:PropHyp}. Hence, in order to prove the statement, we can replace $\cH(\alpha;\beta)$ by $\cH(\widetilde{\alpha};\widetilde{\beta})$ or, in other words, assume from the very beginning that
the statement of the previous lemma \ref{lem:ApplySchulzeWalther} holds
for $\gamma$ itself.

Consider the following diagram
$$
\begin{tikzcd}
\dG=\dG^{N-1}_m\times\Gm \ar{rr}{h\times \id_{\Gm}} \ar[swap]{dd}{p_1} \ar[bend left=40]{ddrrrr}{p} && \dA^N\times \Gm \ar[swap]{dd}{\pi_1} \ar{rrdd}{\pi_2} \\ \\
\dG_m^{N-1} \ar[hook]{rr}{h} && \dA^N&& \Gm
\end{tikzcd}
$$
where the square is Cartesian. Consider also the exponential module $\cE_{\dA^N\times \Gm}^\psi \in \mathrm{Mod}_\mathrm{hol}(\cD_{\dA^N\times \Gm})$, where $\psi=q\cdot y_1+y_2+\ldots+y_N\in\cO_{\dA^N\times \Gm}$.
Then we have
$$
\begin{array}{rclcl}
\FL^\psi(h_+\cO^\gamma_{\dG_m^{N-1}}) & :=&
\cH^0 \pi_{2,+}\left([\pi_1^+ h_+\cO^\gamma_{\dG_m^{N-1}}]
\otimes_{\cO_{\dA^N\times \Gm}} \cE_{\dA^N\times \Gm}^\psi\right)
&\quad\quad&\textup{Definition of  }\FL
\\ \\
&\cong &
\cH^0 \pi_{2,+}\left(
[(h\times\id_{\Gm})_+p_1^+\cO^\gamma_{\dG_m^{N-1}}]
\otimes_{\cO_{\dA^N\times \Gm}} \cE_{\dA^N\times \Gm}^\psi\right)
&&\textup{base change}
\\ \\
&\cong &
\cH^0 \pi_{2,+}\left(
[(h\times\id_{\Gm})_+\cO^\gamma_{\dG}]
\otimes_{\cO_{\dA^N\times \Gm}} \cE_{\dA^N\times \Gm}^\psi\right)
&&p_1^+\cO^\gamma_{\dG_m^{N-1}}\cong\cO^\gamma_\dG
\\ \\
&\cong &\cH^0 \pi_{2,+}
(h\times\id_{\Gm})_+(\cO^\gamma_{\dG}
\otimes_{\cO_{\dG}} \cE_{\dG}^f)
&&\textup{projection formula and }
\\ &&&&(h\times\id_{\Gm})^+\cE_{\dA^N\times \Gm}^\psi=\cE^f_\dG \\ \\
&\cong&  \cH^0 p_+ \cE^{\gamma,f}
&&\cO^\gamma_{\dG}
\otimes_{\cO_{\dG}} \cE_{\dG}^f=\cE^{\gamma,f} \\
&&&& \textup{and }p=\pi_2\circ (h\times\id_{\Gm}).
\end{array}
$$

It thus remains to prove that
$$
\kappa^+\FL^\psi(h_+\cO^\gamma_{\dG_m^{N-1}})  \cong \cH(\alpha;\beta),
$$
using the explicit expression for $h_+\cO^\gamma_{\dG_m^{N-1}}$ from Lemma \ref{lem:ApplySchulzeWalther} (as well as the remark made at the beginning of the current proof). We have
$$
\begin{array}{rcl}
\pi_1^+h_+\cO^\gamma_{\dG_m^{N-1}}
&\cong&
\frac{\D \cD_{\dA^N}}{\D \left(\Box, (E_{m+i-1}+\alpha_i)_{i=2,\ldots,n},
(E_j+\beta_j)_{j=1,\ldots,m}\right)}\boxtimes
\cO_{\Gm}\\ \\
&\cong&
\frac{\D \cD_{\dA^N\times\Gm}}{\D \left(\Box, (E_{m+i-1}+\alpha_i)_{i=2,\ldots,n},
(E_j+\beta_j)_{j=1,\ldots,m},\partial_q\right)}.
\end{array}
$$
The exponential module $ \cE_{\dA^N\times \Gm}^\psi\cong
\cD_{\dA^N\times\Gm}/ \left(\partial_q-y_1,\partial_{y_1}-q,(\partial_{y_k}-1)_{k=2,\ldots,N}\right)$ is $\cO_{\dA^N\times\Gm}$-locally free of rank $1$, generated by the formal symbol $e^\psi$. Hence, the tensor product
$[\pi_1^+ h_+\cO^\gamma_{\dG_m^{N-1}}]
\otimes_{\cO_{\dA^N\times \Gm}} \cE_{\dA^N\times \Gm}^\psi$ equals $\pi_1^+ h_+\cO^\gamma_{\dG_m^{N-1}}$ as
$\cO_{\dA^N\times\Gm}$-module
(and we denote it by $\pi_1^+ h_+\cO^\gamma_{\dG_m^{N-1}}\cdot e^\psi$), and its $\cD_{\dA^N\times\Gm}$-structure is given by the product rule. We thus have for any $n\in \pi_1^+ h_+\cO^\gamma_{\dG_m^{N-1}}$ and for any $k\in\{2,\ldots, N\}$ that
$$
\begin{array}{rcl}
(\partial_{y_1} y_1  \cdot n) \otimes e^\psi &=& \partial_{y_1} \left(y_1  \cdot n \otimes e^\psi\right)-  y_1\cdot n\otimes q \cdot e^\psi
 = \left(\partial_{y_1} -q \right) y_1 \cdot (n\otimes e^\psi)
 \\ \\
\left(\partial_{y_k} y_k \cdot n\right)  \otimes e^\psi &=& \partial_{y_k} \left(y_k  \cdot n \otimes e^\psi\right)-  y_k \cdot n\otimes e^\psi
=\left(\partial_{y_k}-1\right)y_k\cdot\left(n\otimes e^\psi\right)\\ \\
\left(\partial_q \cdot n\right)\otimes e^\psi & = &
\partial_q\left(n\otimes e^\psi\right)-n\otimes y_1\cdot e^\psi =
\left(\partial_q-y_1\right)\cdot\left(n\otimes e^\psi\right),
\end{array}
$$
from which it follows that
\begin{equation}\label{eq:ExprExpTw}
\begin{array}{l}
[\pi_1^+ h_+\cO^\gamma_{\dG_m^{N-1}}]
\otimes_{\cO_{\dA^N\times \Gm}} \cE_{\dA^N\times \Gm}^\psi\cong
\\ \\ \frac{\D \cD_{\dA^N\times\Gm}}{\D \left(\Box, (E_{m+i-1}+\alpha_i)_{i=2,\ldots,n}, (E_j+\beta_j)_{j=1,\ldots,m},\partial_q\right)}\otimes_{\cO_{\dA^N\times \Gm}} \cE_{\dA^N\times \Gm}^\psi  \cong
\\ \\
\frac{\D \cD_{\dA^N\times\Gm}}{\D \left(\Box, ((\partial_{y_1}-q)y_1-(\partial_{y_i}-1)y_i+\alpha_i)_{i=2,\ldots,n},
((\partial_{y_1}-q)y_1+(\partial_{y_{n+j}}-1)y_{n+j}+\beta_j)_{j=1,\ldots,m},\partial_q-y_1\right)} \cong
\\ \\
\frac{\D \cD_{\dA^N\times\Gm}}{\D \left(\Box, (\partial_{y_1}y_1-q\partial_q-\partial_{y_i}y_i+y_i+\alpha_i)_{i=2,\ldots,n},
(\partial_{y_1}y_1-q\partial_q+\partial_{y_{n+j}}y_{n+j}-y_{n+j}+\beta_j)_{j=1,\ldots,m},\partial_q-y_1\right)}.
\end{array}
\end{equation}
Consider the subalgebra
$$
\widetilde{\cD}
:=\cD_{\Gm}\langle \partial_{y_1},\partial_{y_1}y_1,\partial_{y_2},\partial_{y_2}y_2,\ldots,\partial_{y_N},\partial_{y_N}y_N\rangle
$$
of $\cD_{\dA^N\times \Gm}$. Then there is an isomorphism of $\widetilde{\cD}$-modules $$
[\pi_1^+ h_+\cO^\gamma_{\dG_m^{N-1}}]
\otimes_{\cO_{\dA^N\times \Gm}} \cE_{\dA^N\times \Gm}^\psi\cong
\widetilde{\cD}/(\widetilde{\Box})
$$
where
$$
\widetilde{\Box} =
\partial_q\cdot \prod_{i=2}^n\left(-\partial_{y_1}y_1+q\partial_q+\partial_{y_i}y_i-\alpha_i\right)
-\prod_{j=1}^m(\partial_{y_1}y_1-q\partial_q+\partial_{y_{n+j}}y_{n+j}+\beta_j)
$$
by expressing $y_1$ by $\partial_q$, $y_i$ as $-\partial_{y_1}y_1+q\partial_q+\partial_{y_i}y_i-\alpha_i$ for $i=2,\ldots, n$ and $y_{n+j}$ by $\partial_{y_1}y_1-q\partial_q+\partial_{y_{n+j}}y_{n+j}+\beta_j$ for $j=1,\ldots, m$ which is possible due to the denominator of the right hand side of the isomorphism \eqref{eq:ExprExpTw}.
Now since the map $\pi_2$ is affine, the top cohomology $\cH^0 \pi_{2,+}\left([\pi_1^+ h_+\cO^\gamma_{\dG_m^{N-1}}]
\otimes_{\cO_{\dA^N\times \Gm}} \cE_{\dA^N\times \Gm}^\psi\right)$ is nothing but the $N$-th cohomology of the complex
$\pi_{2,*}\DR_{\dA^N\times\Gm/\Gm}^\bullet\left([\pi_1^+ h_+\cO^\gamma_{\dG_m^{N-1}}]
\otimes_{\cO_{\dA^N\times \Gm}} \cE_{\dA^N\times \Gm}^\psi\right)$.
The latter is the cokernel of the morphism given by left multiplication
by $\partial_{y_k}$ for $k=1,\ldots N$. All monomials $\partial_{y_k}$ and $\partial_{y_k} y_k$ are zero in this cokernel, in particular,
the (class of the) operator $-\partial_{y_1}y_1+q\partial_q+\partial_{y_i}y_i-\alpha_i$ resp. $\partial_{y_1}y_1-q\partial_q+\partial_{y_{n+j}}y_{n+j}+\beta_j$ equals (the class of)
$q\partial_q-\alpha_i$ resp. $-q\partial_q+\beta_j$.

Hence, we obtain
$$
\begin{array}{rcl}
\cH^0 \pi_{2,+}\left([\pi_1^+ h_+\cO^\gamma_{\dG_m^{N-1}}]
\otimes_{\cO_{\dA^N\times \Gm}} \cE_{\dA^N\times \Gm}^\psi\right)
&\cong &\frac{\D \cD_{\Gm}}{\D
\left(\partial_q\prod_{i=2}^n(q\partial_q-\alpha_i)-(-1)^m\prod_{j=1}^m(q\partial_q-\beta_j)\right)} \\ \\
\cong\frac{\D \cD_{\Gm}}{\D
\left((q\partial_q)\prod_{i=2}^n(q\partial_q-\alpha_i)-(-1)^m q \prod_{j=1}^m(q\partial_q-\beta_j)\right)}
&\cong&  \kappa^+\cH(\alpha;\beta)
\end{array}
$$
Since $\kappa$ is an involution, the statement we are after follows.

It remains to prove the vanishing of $\cH^i p_+ \cE^{\gamma,f} $ for $i\neq 0$. First note that the complex $p_+ \cE^{\gamma,f}\in D^b_h(\cD_{\Gm})$ can alternatively be calculated as follows:
\begin{equation}\label{eq:ForFLExp}
p_+\cE^{\gamma,f}\cong \iota^+ \FL_{\Gm}(\phi_+\cO^\gamma_\dG).
\end{equation}
where $\iota: \Gm\hookrightarrow \widehat{\dA}^1 \times \Gm$,
$q\mapsto (1,q)$ is the embedding, where $\FL_{\Gm}:\textup{Mod}(\cD_{\dA^1\times \Gm})\rightarrow
\textup{Mod}(\cD_{\widehat{\dA}^1\times \Gm})$ is the partial Fourier transformation with respect to the first factor, and where we write $\phi=(f,p):\dG\rightarrow \dA^1\times\Gm$.
This is well-known, but let us reprove it here for the convenience of the reader: We have the following diagram, the leftmost part of which is cartesian,
$$
\begin{tikzcd}
& \dG\times\dA^1\times\widehat{\dA}^1 \ar{ld}{\pi_{\dG}} \ar{rd}{\widetilde{\phi}
=(f,\pi_{\widehat{\dA}^1},\pi_{\Gm})} \ar[bend left=40]{rrdd}{\widetilde{\pi}_{\widehat{\dA}^1\times \Gm}}\\
\dG=\dG_m^{N-1}\times\Gm \ar{rd}{\phi}&&\dA^1\times\widehat{\dA}^1\times\Gm \ar{ld}{\pi_{\dA^1\times\Gm}} \ar[swap]{rd}{\pi_{\widehat{\dA}^1\times\Gm}}\\
&\dA^1\times\Gm&& \widehat{\dA}^1\times\Gm & \Gm \ar[hook',swap]{l}{\iota}
\end{tikzcd}
$$
where $\pi_{\dG}, \pi_{\Gm},\pi_{\dA^1}, \pi_{\dA^1\times \Gm}, \pi_{\widehat{\dA}^1}, \pi_{\widehat{\dA}^1\times \Gm}$ and $\widetilde{\pi}_{\widehat{\dA}^1\times \Gm}$
denote the obvious projections. We now have (denoting the coordinates on $\dA^1$ resp.
$\widehat{\dA}^1$ by $t$ resp. by $\tau$):
$$
\begin{array}{rclcl}
\iota^+ \FL_{\Gm}(\phi_+\cO^\gamma_\dG) & = &
\iota^+ \pi_{\widehat{\dA}^1\times\Gm,+}\left((\pi_{\dA^1\times\Gm}^+\phi_+\cO^\gamma_\dG)\otimes \cE^{\tau t}\right)
&\quad\quad&\textup{definition of }\FL\\ \\
&\cong&
\iota^+ \pi_{\widehat{\dA}^1\times\Gm,+}\left((\widetilde{\phi}_+\pi_{\dG}^+\cO^\gamma_\dG)\otimes \cE^{\tau t}\right)
&\quad\quad&\textup{base change}\\ \\
&\cong&
\iota^+ \pi_{\widehat{\dA}^1\times\Gm,+}\widetilde{\phi}_+\left(\pi_{\dG}^+\cO^\gamma_\dG\otimes \cE^{\tau f}\right)
&\quad\quad&\textup{projection formula}\\ \\
&\cong&
\iota^+ \widetilde{\pi}_{\widehat{\dA}^1\times\Gm,+}\left(\pi_{\dG}^+\cO^\gamma_\dG\otimes \cE^{\tau f}\right) &\quad\quad& \widetilde{\pi}_{\widehat{\dA}^1\times\Gm}=
\pi_{\widehat{\dA}^1\times\Gm}\circ\widetilde{\phi}\\ \\
&\cong&
p_+\left(\cO_\dG^\gamma\otimes\cE^f\right)\cong p_+  \cE^{\gamma,f}
&\quad\quad&\textup{base change},
\end{array}
$$
and this shows the statement of formula \eqref{eq:ForFLExp}.
Now consider the following diagram
$$
\begin{tikzcd}
& \widetilde{\dG} \ar{dd}{\pi}   \ar[bend left=40]{dddd}{\widetilde{\Phi}}\\ \\
\dG \ar{r}{j} \ar[swap]{rdd}{\phi} \ar{ruu}{\widetilde{j}} & \widehat{\dG} \ar{dd}{\Phi}\\ \\
& \dA^1\times\Gm
\end{tikzcd}
$$
where $\Phi:\widehat{\dG}\rightarrow \dA^1\times \Gm$ is a partial
compactification of $\phi$, i.e.\ $\widehat{\dG}$ is
the toric compactification of the graph of $\phi$, and $\pi:\widetilde{\dG}\rightarrow \widehat{\dG}$ is a resolution of singularities such that $\widetilde{\dG}\backslash \dG$ is a simple normal crossing divisor. In particular, $\Phi$, $\pi$ and $\widetilde{\Phi}$ are projective morphisms, and the non-degeneracy property of Lemma \ref{lem:NonDeg} translates into the fact that $\Phi$ is stratified smooth on $\widehat{\dG} \backslash \dG$, that is, there is a (toric) stratification of $\widehat{\dG} \backslash \dG$ such that
the restriction of $\Phi$ to each of its strata has no critical points. Moreover, since $\phi$ has isolated critical points, the hypotheses of \cite[Corollary 7.6]{SchapSchnei} are satisfied for (the analytification of) $\widetilde{\Phi}$ and for $(\widetilde{j}_+\cO_\dG^\gamma)^{\mathrm{an}}$. It follows then from loc.cit. that
$\cH^i \widetilde{\Phi}^{\mathrm{an}}_+ (\widetilde{j}_+\cO_\dG^\gamma)^{\mathrm{an}}$ is smooth for $i<0$, but since $\widetilde{\Phi}$ is proper we conclude that the same holds
for $\cH^i\phi_+\cO_\dG^\gamma$.

In particular, for $i<0$, $\cH^i \phi_+\cO^\gamma_\dG$ can be seen as a free $\cO_{\dA^1}$-module (of infinite rank), but then
its partial Fourier transform
$$
\FL( \cH^i(\phi_+\cO^\gamma_\dG))=
\cH^i \FL(\phi_+\cO^\gamma_\dG)
$$
(notice that $\FL$ is an exact functor) is supported at
$\{\tau=0\}\times \Gm\subset \widehat{\dA}^1\times \Gm$.
Consequently, we have
$$
\cH^j\iota^+\cH^i\,\FL(\phi_+\cO_\dG^\gamma) =0
$$
for all $j$ and for all $i<0$.

The statement we are after is thus proved once we know that the embedding $\iota$ is non-characteristic for the module $\cH^0\FL(\phi_+\cO_\dG^\gamma)=
\FL(\cH^0\phi_+\cO_\dG^\gamma)$. This follows from \cite[Theorem 1.11 (2)]{DS} (applying it to $M=\cH^0\phi_+\cO_\dG^\gamma$), notice that the condition (NC) in loc.cit. holds by a general argument from \cite[Lemma 3.13]{ReiSe2}, but can also be seen directly, namely, if $\Delta(\phi)$ is the discriminant of the map $\phi$ (containing the singular locus of $M$), then the restriction of the projection
$\dA^1\times\Gm \twoheadrightarrow \Gm$
to $\Delta(\phi)$ satisfies the condition (NC).

\end{proof}

We will now use the presentation of the hypergeometric system given in Proposition~\ref{prop:HypGMSystem} in order to study $K$-structures on the enhanced solutions of $\cH(\alpha;\beta)$ for subfields $K\subset \dC$.

Let $L$ be a subfield of $\dC$ such that $L \supset \dQ(e^{2\pi i \gamma_1},\ldots, e^{2\pi i\gamma_{N-1}})$. Consider a field $K\subset L$ such that $L/ K$ is a finite Galois extension. We write $G:=\Gal(L/K)$, so that we are in the situation considered in Section \ref{sec:BettiConj}. Let $(\Gm)_\infty$ and $\dG_\infty$ be the bordered spaces associated to $\Gm$ and $\dG$, respectively, by a smooth completion as in
Section \ref{sectionRH}. Our main objective is to find criteria such that the object $\SolE{(\Gm)_\infty}(\cH(\alpha;\beta))$ has a $K$-structure, i.e.\ is defined over $K$ in the sense of Proposition \ref{prop:KStructure}.

Using the properties of the algebraic enhanced solution functor and Proposition \ref{prop:HypGMSystem}, we can write
\begin{align*}
\SolE{(\Gm)_\infty}\big(\cH(\alpha;\beta)\big)&\cong \SolE{(\Gm)_\infty}\big(\kappa^+ p_+ \cE^{\gamma,f}\big)\\
&\cong \EE\kappa_\infty^{-1} \EE p_{\infty !!} \big(\pi^{-1}\Sol{\dG}(\cO^\gamma)\otimes \SolE{\dG_\infty}(\cE^f)[N-1])\\
&\cong \EE\kappa_\infty^{-1} \EE p_{\infty !!} \big(\pi^{-1}\Sol{\dG}(\cO^\gamma)\otimes \dE_\dC^{\real f}[N-1]\big)\in \EbI{\dC}{(\Gm)_{\infty}}.
\end{align*}

Note that $\Sol{\dG}(\cO^\gamma)\in\Mod{\dC}{\dG}$ is a local system with semi-simple monodromy with eigenvalues contained in $L$, so one can find a local system $\cF^\gamma\in \Mod{L}{\dG}$ such that $\Sol{\dG}(\cO^\gamma)\cong \dC_\dG \otimes_{L_\dG} \cF^\gamma$. On the other hand, we clearly have $\dE^{\real f}_\dC\cong \pi^{-1}\dC_\dG\otimes_{\pi^{-1}L_\dG} \dE^{\real f}_L$ (indeed, $\dE^{\real f}_\dC$ admits a structure over any subfield of $\dC$). Consequently, in view of Lemma \ref{lemmaExtensionCompatibility}, we obtain
\begin{equation}\label{eq:SolEHypergeom}
\SolE{(\Gm)_\infty}\big(\cH(\alpha;\beta)\big)\cong \pi^{-1}\dC_{\Gm}\otimes_{\pi^{-1}L_{\Gm}}\EE \kappa_\infty^{-1} \EE p_{\infty !!} \big(\pi^{-1}\cF^\gamma\otimes \dE_L^{\real f}[N-1]\big),
\end{equation}
so this object naturally carries a structure over $L$.
We cannot, however, expect to find a structure over $K$ in general. However, under suitable assumptions on $\gamma$, we will show that the above object carries a $K$-structure (even if $\cF^\gamma$ itself does not).

To this end, we introduce the following group theoretic condition on the exponent vector we are interested in.
\begin{definition}\label{def:Ggood}
Let $k\in \dN$ be arbitrary, and consider any vector
$\gamma=(\gamma_1,\ldots,\gamma_k)\in \dR^k$. Let $L\subset \dC$ such that $L\supset \dQ(e^{2\pi i \gamma_1},\ldots,
e^{2\pi i \gamma_k})$. For any $l\in \dN_0$, let
$$
\cO_{\dG_m^{k+l}}^\gamma:=\cO_{\dG_m^k}^\gamma\boxtimes\cO_{\dG_m^l},
$$
be defined as  in formula \eqref{eq:DefOG}. We let $\cF^\gamma$ be the $L$-local system on $\dG_m^{k+l}$ such that $\Sol{\dG_m^{k+l}}(\cO_{\dG_m^{k+l}}^\gamma)=\dC_{\dG_m^{k+l}} \otimes_{L_{\dG_m^{k+l}}}\cF^\gamma$. Let $K\subset L$ be a subfield such that $L/K$ is finite Galois and let $G:=\Gal(L/K)$ be its Galois group. Then we say that $\gamma$ is $G$-good
if there is an action
\begin{align*}
\varrho:G&\longrightarrow \Aut(\dG^k_m) \subset \Aut(\dG_k^{k+l})\\
g&\longmapsto  \varrho_g
\end{align*}
such that
for all $g\in G$ we have an isomorphism
\begin{equation}\label{eq:ConjugationCond}
\psi_g:{\varrho}_{g !} \cF^\gamma \stackrel{\cong}{\longrightarrow} \overline {\cF^\gamma}^g
\end{equation}
such that for all $g,h\in G$, the following diagram commutes:
\begin{equation}\label{eq:Cocycle}
\begin{tikzcd}
\varrho_{g !} \varrho_{h !}\cF^\gamma \arrow{rr}{\varrho_{g !}\psi_h}\arrow[equal]{drr}{\sim} && \varrho_{g !} \overline{\cF^\gamma}^h \cong \overline{\varrho_{g !} \cF^\gamma}^h\arrow{r}{\overline{\psi_g}^h} & \overline{\overline{\cF^\gamma}^g}^h\cong &[-27pt]\overline{\cF^\gamma}^{gh} \\
&&\varrho_{gh !}\cF^\gamma\arrow{rru}[swap]{\psi_{gh}}.
\end{tikzcd}
\end{equation}
\end{definition}
Notice that by requiring that the action of $G$ on $\dG$ factors over $\dG_m^k$ we impose that $G$ acts trivially on the last $l$ factors of $\dG_m^{k+l}$.

This notion of $G$-goodness will be the crucial point for the hypergeometric systems  to have a $G$-structure in the sense of Definition \ref{def:GStructure} and consequently to be defined over the field $K$. To this end, we apply the previous definition for the case $k=N-1$ and $l=1$, that is,
we let $\cF^\gamma$ be the $L$-local system on $\dG$ such that $\Sol{\dG}(\cO^\gamma_\dG) = \dC_\dG\otimes \cF^\gamma$. Moreover, for any function $f\in \cO_\dG$, we write $\dE^{\real f}\vcentcolon= \dE^{\real f}_L$ and $\dE^{\gamma,f}\vcentcolon=\pi^{-1}\cF^\gamma\otimes \dE^{\real f}\in\EbI{L}{\dG_\infty}$ for short.

\begin{proposition}\label{prop:ConjProperty}
For $N\in \dN$ and $\gamma\in\dR ^{N-1}$ let $L\subset \dC$ be as above. Choose a subfield $K\subset L$, such that
\begin{enumerate}
    \item $L/K$ is a finite Galois extension with $G=\Gal(L/K)$,
    \item $\gamma$ is $G$-good.
\end{enumerate}
Then for any $G$-invariant function $f$, i.e.\ for any $f\in\cO_\dG^{\im(\varrho)}$
(where $\varrho$ is the action from the previous definition), the object $\mathrm{E}p_{\infty !!}\dE^{\gamma,f}=\mathrm{E}p_{\infty !!}(\pi^{-1}\cF^\gamma\otimes\dE^{\real f})\in \EbI{L}{(\Gm)_\infty}$
has a $G$-structure.
\end{proposition}
\begin{proof}
By definition, the condition that $\gamma$ is $G$-good means that we have isomorphisms
$$\psi_g\colon \varrho_{g !}\cF^\gamma\xlongrightarrow{\cong} \overline{\cF^\gamma}^g$$
for any $g\in G$ (satisfying the compatibility condition given by diagram \eqref{eq:Cocycle}).

Furthermore, there are isomorphisms for any $g\in G$
\begin{equation}\label{eq:canonicalMorphExp}
\dE^{\mathrm{Re}\, f}\xlongrightarrow{\cong} \overline{\dE^{\mathrm{Re}\, f}}^g\quad\text{and}\quad \mathrm{E}(\varrho_{g})_{\infty}^{-1}\dE^{\mathrm{Re}\, f} \cong \dE^{\mathrm{Re}\, (f\circ \varrho_g)} \cong \dE^{\mathrm{Re}\, f},
\end{equation}
where the first isomorphism is given by the action of $g$ on $L$ (and hence we will denote it simply by $g$) and the second isomorphism follows from \cite[Remark 3.3.21]{DK16} and the fact that $f$ is invariant under the action of $\varrho_g$.

Then, we can conclude using the projection formula:
\begin{align}\label{eq:IsosToConjugate}
\overline{\mathrm{E}p_{\infty !!}(\pi^{-1}\cF^\gamma\otimes \dE^{\real f})}^g&\cong
\mathrm{E}p_{\infty !!}\big(\pi^{-1}\overline{\cF^\gamma}^g\otimes \overline{\dE^{\mathrm{Re}\,f}}^g\big)\notag\\
&\xleftarrow{\sim} \mathrm{E}p_{\infty !!}\big(\pi^{-1}\varrho_{g !}\cF^\gamma\otimes \dE^{\mathrm{Re}\,f}\big)\notag\\
&\cong \mathrm{E}(p\circ \varrho_g )_{\infty !!}\big(\pi^{-1}\cF^\gamma\otimes \mathrm{E}(\varrho_g)_\infty^{-1}\dE^{\mathrm{Re}\,f}\big)\notag\\ &\cong\mathrm{E}p_{\infty !!}(\pi^{-1}\cF^\gamma \otimes \dE^{\real f}),
\end{align}
where we used $p\circ\varrho_g=p$ in the last isomorphism and the second isomorphism is induced by $\psi_g$ and the action of $g$ on the first and second factor of the tensor product, respectively.

Let us call these isomorphisms $\varphi_g: \mathrm{E}p_{\infty !!}\dE^{\gamma,f}\to \overline{\mathrm{E}p_{\infty !!}\dE^{\gamma,f}}^g$ (from right to left in \eqref{eq:IsosToConjugate}). It remains to check that these isomorphisms satisfy the compatibilities from Definition \ref{def:GStructure}, i.e.\ that for $g,h\in G$ the isomorphism
\begin{align*}
\varphi_{gh}: \mathrm{E}p_{\infty !!}\dE^{\gamma,f}\to \overline{\mathrm{E}p_{\infty !!}\dE^{\gamma,f}}^{gh}
\end{align*}
coincides with the isomorphism
\begin{align*}\mathrm{E}p_{\infty !!}\dE^{\gamma,f}\xrightarrow{\varphi_h} \overline{\mathrm{E}p_{\infty !!}\dE^{\gamma,f}}^h \xrightarrow{\overline{\varphi_g}^h} \overline{\overline{\mathrm{E}p_{\infty !!}\dE^{\gamma,f}}^{g}}^h \cong \overline{\mathrm{E}p_{\infty !!}\dE^{\gamma,f}}^{gh}.
\end{align*}
This is mostly a matter of checking that ``resolving'' the $gh$-conjugation is equivalent to resolving the $g$- and $h$-conjugations one after another in \eqref{eq:IsosToConjugate}. Since many of the isomorphisms used in \eqref{eq:IsosToConjugate} are natural (such as the projection formula and the isomorphisms from Lemma \ref{lemmaConjCompat}), the main step is to check the term in parentheses and see that the isomorphism
$$
\pi^{-1}\overline{\cF^\gamma}^{gh}\otimes \overline{\dE^{\mathrm{Re}\,f}}^{gh}\xlongleftarrow{\pi^{-1}\psi_{gh}\otimes (gh)} \pi^{-1}\varrho_{gh!}\cF^\gamma \otimes \dE^{\mathrm{Re}\,f}$$
coincides with
$$
\pi^{-1}\overline{\overline{\cF^\gamma}^g}^h\otimes \overline{\overline{\dE^{\mathrm{Re}\,f}}^g}^h\xlongleftarrow{\pi^{-1}\overline{\psi_g}^h\otimes\overline{g}^h} \pi^{-1}\overline{\varrho_{g  !}\cF^\gamma}^h\otimes \overline{\dE^{\mathrm{Re}\,f}}^h\cong \pi^{-1}\varrho_{g  !}\overline{\cF^\gamma}^h\otimes \overline{\dE^{\mathrm{Re}\,f}}^h\xlongleftarrow{\pi^{-1}\varrho_{g !}\psi_h \otimes h} \pi^{-1}\varrho_{g !}\varrho_{h !}\cF^\gamma \otimes \dE^{\mathrm{Re}\,f}$$
(recall that $g$, $h$ and $gh$ here denote the morphisms induced by the action of the respective Galois group elements, as in the first isomorphism in \eqref{eq:canonicalMorphExp}).
But this holds by \eqref{eq:Cocycle} since $\gamma$ was chosen to be $G$-good.
\end{proof}

We now specify to the case we are mostly interested in, that is, we let $\gamma=(\beta_1,\ldots,\beta_m,\alpha_2,\ldots,\alpha_n)\in[0,1)^{N-1}$, with $\alpha_1=0$ and $\alpha_i\neq \beta_j$ for all $i\in\{1,\ldots,n\}$ and all $j\in\{1,\ldots,m\}$. Then we have the following.
\begin{corollary}\label{cor:Kstruct-directimage}
Under the assumptions of the previous Proposition~\ref{prop:ConjProperty} on $\gamma$, $K$ and $f$, the enhanced ind-sheaf $\SolE{(\Gm)_\infty}(\kappa^+p_+\cE^{\gamma,f})$ admits a $K$-structure. More precisely, there exists $H_K\in \EbRcI{K}{(\Gm)_\infty}$ such that
$$
\SolE{(\Gm)_\infty}(\kappa^+p_+\cE^{\gamma,f})\cong \pi^{-1}\dC_{\Gm}\otimes_{\pi^{-1}K_{\Gm}}H_K.
$$
\end{corollary}
\begin{proof}
Since the $\cD_{\Gm}$-module $\kappa^+p_+\cE^{\gamma,f}$ is concentrated in degree $0$ and has no singularities on $\Gm$ (recall that it is isomorphic to an irregular hypergeometric module by Proposition~\ref{prop:PropHyp}, whose singular points are at $0$ and $\infty$), we know from Lemma \ref{lemmaMochizukiDegreeZero} that also $\SolE{(\Gm)_\infty}(\kappa^+p_+\cE^{\gamma,f})\in \mathrm{E}^0(\mathrm{I}\dC_{(\Gm)_\infty})$. By our preliminary observations in \eqref{eq:SolEHypergeom}, we have
\begin{align*}
	\SolE{(\Gm)_\infty}\big(\kappa^+p_+\cE^{\gamma,f}\big)&\cong \pi^{-1}\dC_{\Gm}\otimes_{\pi^{-1}L_{\Gm}}\EE \kappa_\infty^{-1} \EE p_{\infty !!} \big(\dE^{\gamma,f}[N-1]\big)
\end{align*}
and since the functor $\pi^{-1}\dC_{\Gm}\otimes_{\pi^{-1}L_{\Gm}}(\bullet)$ is exact with respect to the standard t-structure (see Proposition~\ref{prop:tExactnessScalarExt}), it follows that $(\EE \kappa_\infty^{-1} \EE p_{\infty !!} \big(\dE^{\gamma,f}[N-1]\big))\in \mathrm{E}^0(\mathrm{I}L_{(\Gm)_\infty})$.

The object $\EE p_{\infty !!}\dE^{\gamma,f}\in\EbRcI{L}{(\Gm)_\infty}$ carries a $G$-structure by Proposition \ref{prop:ConjProperty}, and so does
$$
\EE \kappa_\infty^{-1} \EE p_{\infty !!} \big(\dE^{\gamma,f}[N-1]\big)\in\mathrm{E}^0(\mathrm{I}L_{(\Gm)_\infty})\cap\EbRcI{L}{(\Gm)_\infty}
$$
by Lemma \ref{lemmaConjCompat}. We are now in the situation to apply Proposition~\ref{prop:KStructure} and obtain the desired object $H_K\in \mathrm{E}^0(\mathrm{I}K_{(\Gm)_\infty})\cap\EbRcI{K}{(\Gm)_\infty}$ satisfying
\begin{align*}
	\SolE{(\Gm)_\infty}(\kappa^+p_+\cE^{\gamma,f})&\cong \pi^{-1}\dC_{\Gm}\otimes_{\pi^{-1}L_{\Gm}} \EE \kappa_\infty^{-1} \mathrm{E}\widetilde{p}_{!!}(\dE^{\gamma,f}[N-1])\\
	&\cong \pi^{-1}\dC_{\Gm}\otimes_{\pi^{-1}L_{\Gm}} (\pi^{-1}L_{\Gm} \otimes_{\pi^{-1}K_{\Gm}} H_K)\\
	&\cong \pi^{-1}\dC_{\Gm}\otimes_{\pi^{-1}K_{\Gm}} H_K.
\end{align*}
\end{proof}

Putting these results together, we arrive at the following first main result.
\begin{theorem}\label{theo:MainTheo}
Let $n>m$, put $N:=n+m$ and let $\gamma=(\beta_1,\ldots,\beta_m,\alpha_2,\ldots,\alpha_n)\in [0,1)^{N-1}$ be given,  where $\alpha_i\neq \beta_j$
for all $i\in\{1,\ldots,n\}$ and for all $j\in \{1,\ldots,m\}$ (with $\alpha_1=0$). Write again
$\cH(\alpha;\beta)=\cH(0,\alpha_2,\ldots,\alpha_n;\beta_1,\ldots,\beta_m)$ for the corresponding irreducible hypergeometric system. Consider $L\subset \dC$ such that $L\supset\dQ(e^{2\pi i \gamma_1},\ldots, e^{2\pi i\gamma_{N-1}})$ and let $K\subset L$ be a subfield satisfying the following properties
\begin{enumerate}
    \item $L/K$ is a finite Galois extension with $G:=\Gal(L/K)$,
    \item $\gamma$ is $G$-good,
    \item the function
    $$
    f:=x_1+\ldots+x_m+\frac{1}{x_{m+1}}+\ldots+ \frac{1}{x_{N-1}}+q\cdot x_1\cdot\ldots\cdot x_{N-1}\in \cO_\dG
    $$
    is $G$-invariant, in the sense that $f\in \cO_\dG^{\im(\varrho)}$, where $\varrho$ is the action from Definition \ref{def:Ggood}.
\end{enumerate}
Then the enhanced ind-sheaf $\SolE{(\Gm)_\infty}\big(\cH(\alpha;\beta)\big)$
has a $K$-structure in the sense of Definition \ref{def:KStructure}. More precisely, there is some
$H_K\in \EbRcI{K}{(\Gm)_\infty}$ such that $\SolE{(\Gm)_\infty}(\cH(\alpha;\beta))\cong \pi^{-1}\dC_{\Gm} \otimes_{\pi^{-1}K_{\Gm}} H_K$.

\end{theorem}
\begin{proof}
This follows by combining Corollary \ref{cor:Kstruct-directimage} with the isomorphism
\eqref{eq:FourierHypergeom} in Proposition \ref{prop:HypGMSystem}.
\end{proof}

In the remainder of this section, we will remove the two conditions $n>m$ and $\alpha_1=0$ that we had to impose so far.
This will be done in two steps.

Consider first the case where we are given numbers $n,m\in \dN_0$ (with $(n,m)\neq (0,0)$), where now we only ask that $n\geq m$. Let $\alpha_1,\ldots,\alpha_n\in  [0,1)$ and $\beta_1,\ldots,\beta_m\in (0,1)$ be given, with $\alpha_i\neq \beta_j$. The case where at least one of the numbers $\beta_j$ equals zero is thus excluded, and will be treated later. We still want to consider irreducible hypergeometric modules, and thus assume that
$\alpha_i\neq \beta_j$ for all $i\in\{1,\ldots,n\}$ and all $j\in\{1,\ldots,m\}$. We let $\gamma:=
(\beta_1,\ldots,\beta_m,\alpha_1,\ldots,\alpha_n)\in [0,1)^N$, and, as before, we consider a field $L\subset \dC$ such that $e^{2\pi i \gamma_1},\ldots,e^{2\pi i \gamma_N}\in L$. Then the following holds.

\begin{corollary}\label{cor:CaseNonZero}
Under these hypotheses, let $K\subset L$ be a field such that
\begin{enumerate}
\item $L/K$ is finite Galois with Galois group $G$,
\item $\gamma$ is $G$-good, where now
the numbers $k,l$ in Definition \ref{def:Ggood} are $k=N$ and $l=1$,
\item the function
$$
f:=x_1+\ldots+x_m+\frac{1}{x_{m+1}}+\ldots+ \frac{1}{x_N}+q\cdot x_1\cdot\ldots\cdot x_N\in \cO_\dG =\cO_{\dG_m^N}\boxtimes\cO_{\Gm}
$$
is $G$-invariant.
\end{enumerate}
Then the enhanced ind-sheaf $\SolE{(\Gm)_\infty}\big(\cH(\alpha;\beta)\big)$
has a $K$-structure. More precisely, there exists
$H_K\in \EbRcI{K}{(\Gm)_\infty}$ such that $\SolE{(\Gm)_\infty}(\cH(\alpha;\beta))\cong \pi^{-1}\dC_{\Gm} \otimes_{\pi^{-1}K_{\Gm}} H_K$.
\end{corollary}

\begin{proof}[Proof of Corollary \ref{cor:CaseNonZero}]
	In \cite[Proposition 5.3.3]{Ka}, N.\ Katz showed that one has
	$$\cH(0,\alpha;\beta)\cong j^+ \FL (j_+\mathrm{inv}_+\cH(\alpha;\beta)),$$
	where $\mathrm{inv}\colon \dG_m\to\dG_m, q\mapsto q^{-1}$ is given by multiplicative inversion, $j\colon \dG_m\hookrightarrow\dA^1$ is the embedding and $\FL(\bullet)$ is the Fourier transform with kernel $e^{qw}$ for $\cD$-modules on $\dA^1$. In particular, $\FL (j_+\mathrm{inv}_+\cH(\alpha;\beta))$ is an extension of $\cH(0,\alpha;\beta)$  and hence there exist natural morphisms
	$$
		j_\dagger\cH(0,\alpha;\beta)\longrightarrow \FL (j_+\mathrm{inv}_+\cH(\alpha;\beta))\longrightarrow j_+\cH(0,\alpha;\beta)
	$$
	whose composition is the canonical morphism $j_\dagger \cH(0,\alpha;\beta)\xlongrightarrow{c}j_+\cH(0,\alpha;\beta)$.
	This follows by applying the canonical
	morphisms of functors $j_\dagger j^\dagger = j_\dagger j^+ \rightarrow \id$ and $\id\rightarrow j_+j^+$ (which are part of the  adjunction triangles) to the module $\FL (j_+\mathrm{inv}_+\cH(\alpha;\beta))$.
	Here $j_+$ (resp.\ $j_\dagger$) denotes the direct image (resp.\ proper direct image) for $\cD$-modules.
		This induces morphisms
	$$\mathrm{inv}^+ j^+\FL^{-1}(j_\dagger\cH(0,\alpha;\beta))\xlongrightarrow{b_1} \cH(\alpha;\beta)\xlongrightarrow{b_2} \mathrm{inv}^+ j^+\FL^{-1}(j_+\cH(0,\alpha;\beta)).$$
	We claim that $b_1$ and $b_2$ are nonzero, and for this it obviously suffices to show that their composition is not the zero map. Moreover, since $\mathrm{inv}^+$ is an involution, it is sufficient to show that the morphism
	$$
	j^+\FL^{-1}(c)\colon j^+\FL^{-1}(j_\dagger\cH(0,\alpha;\beta))  \longrightarrow  j^+\FL^{-1}(j_+\cH(0,\alpha;\beta))
	$$
	induced by the canonical morphism
	$c\colon j_\dagger\cH(0,\alpha;\beta))  \rightarrow  j_+\cH(0,\alpha;\beta)$
	is not zero. Consider first the morphism
	$$
	\FL^{-1}(c)\colon\FL^{-1}(j_\dagger\cH(0,\alpha;\beta))  \longrightarrow  \FL^{-1}(j_+\cH(0,\alpha;\beta)),
	$$
	then we have, denoting by $j_{\dagger+}$ the middle extension,
	$$
	    \im(\FL^{-1}(c))=\FL^{-1}(\im(c))=\FL^{-1}(j_{\dagger+}\cH(0,\alpha;\beta)),
	$$
	where we use that the functor $\FL^{-1}$ is exact (for the standard t-structure on $\textup{Mod}(\cD_{\dA^1})$). Now clearly $j_{\dagger+}\cH(0,\alpha;\beta)$ is not a free $\cO_{\dA^1}$-module
	(since $\cH(0,\alpha;\beta)$ has non-trivial monodromy), therefore
	the module $\FL^{-1}(j_{\dagger+}\cH(0,\alpha;\beta))$ is not supported on $\{0\}\subset \dA^1$ and consequently
	$$
	\im(j^+\FL^{-1}(c))=j^+\im(\FL^{-1}(c))=j^+\FL^{-1}(j_{\dagger+}\cH(0,\alpha;\beta))\neq 0,
	$$
	where we have used that also $j^+$ is exact. This shows that $j^+\FL^{-1}(c)\neq 0$, hence $\mathrm{inv}^+j^+\FL^{-1}(c)\neq 0$,  and therefore neither of the morphisms $b_1$ nor $b_2$ can be zero, which proves the claim made above.
	
	Now notice that $\cH(\alpha;\beta)$ is irreducible, from which it follows that $b_1$ must be an epimorphism and $b_2$ is a monomorphism. In other words, $\cH(\alpha;\beta)$ is the image of the canonical morphism $c=b_2\circ b_1$.
	
	Now we know from \cite{EnhPerv} that the contravariant functor $\SolE{(\dG_m)_\infty}(\bullet)[1]$ is exact with respect to the standard t-structure on $\Dbhol{\dG_m}$ and the middle perversity t-structure on $\EbRcI{\dC}{(\dG_m)_\infty}$, which means that $\SolE{(\dG_m)_\infty}(\cH(\alpha;\beta))$ is the coimage of the canonical morphism
	\begin{equation}\label{eq:canonicalMorphSolE}
		\SolE{(\dG_m)_\infty}\big(\mathrm{inv}^+ j^+\FL^{-1}(j_+\cH(0,\alpha;\beta))\big)\longrightarrow \SolE{(\dG_m)_\infty}\big(\mathrm{inv}^+ j^+\FL^{-1}(j_\dagger\cH(0,\alpha;\beta))\big),
	\end{equation}
	which is a strict morphism in the quasi-abelian category ${}^{1/2}\EE_{\rc}^1(\mathrm{I}\dC_{(\dG_m)_\infty})={}^{1/2}\EE_{\rc}^0(\mathrm{I}\dC_{(\dG_m)_\infty})[-1]$ by Lemma~\ref{lemma:SolEstrictness}. Therefore, by Corollary~\ref{cor:scalarExt-kernel} it suffices to show that the domain and target of \eqref{eq:canonicalMorphSolE} admit $K$-structures and the morphism \eqref{eq:canonicalMorphSolE} is compatible with these.
	
	Since the enhanced solution functor commutes with direct and inverse images (see \ref{lemmaAlgRHCompat}), the morphism \eqref{eq:canonicalMorphSolE} is equal to the morphism
	\begin{equation}\label{eq:canonicalMorphSolE2}
	\EE\,\mathrm{inv}^{-1}\,\EE j^{-1}\, {^\EE}\mathrm{FL}^{-1}\,
	\SolE{\dA^1_\infty}\big(j_+\cH(0,\alpha;\beta)\big)
	\longrightarrow
	\EE\,\mathrm{inv}^{-1}\,\EE j^{-1}\, {^\EE}\mathrm{FL}^{-1}\SolE{\dA^1_\infty}\big(j_\dagger\cH(0,\alpha;\beta)\big),
	\end{equation}
	induced by the canonical morphism
	\begin{equation}\label{eq:canonicalMorphA1}
	\SolE{\dA^1_\infty}(j_+\cH(0,\alpha;\beta))
	\longrightarrow
	\SolE{\dA^1_\infty}(j_\dagger\cH(0,\alpha;\beta)).
	\end{equation}
	Here, we denoted by ${^\EE}\mathrm{FL}^{-1}(\bullet)$ the topological counterpart for enhanced ind-sheaves of $\FL^{-1}$. This functor is given by $\EE p_{1\infty !!}(\dE^{-qw}\conv \EE p_{2\infty}^{-1}(\bullet))[1]$ (with $p_1,p_2\colon \dA^1\times\dA^1\to\dA^1$ the projections to the first and second factor, respectively; see e.g.\ \cite{KS16} for a study of integral transforms in the context of enhanced ind-sheaves). Hence, by Lemma~\ref{lemmaExtensionCompatibility} all the functors for enhanced ind-sheaves appearing in \eqref{eq:canonicalMorphSolE2} are compatible with extension of scalars, and it is therefore sufficient to prove that the morphism \eqref{eq:canonicalMorphA1} is defined over $K$.
	
	Let us show that the right-hand side of \eqref{eq:canonicalMorphA1} has a $K$-structure:
	\begin{align*}
		\SolE{\dA^1_\infty}(j_\dagger\cH(0,\alpha;\beta))&\cong \SolE{\dA^1_\infty}(\dD_{\dA^1}j_+\dD_{\dG_m}\cH(0,\alpha;\beta))\\
		&\cong \DE{\dA^1_\infty}\SolE{\dA^1_\infty}(j_+\dD_{\dG_m}\cH(0,\alpha;\beta))[-2]\\
		&\cong \DE{\dA^1_\infty}\big(\pi^{-1}\dC_{\dG_m}\otimes \SolE{\dA^1_\infty}(j_+\dD_{\dG_m}\cH(0,\alpha;\beta))\big)[-2]\\
		&\cong \DE{\dA^1_\infty}\big(\EE j_{\infty !!}\EE j_\infty^{-1} \SolE{\dA^1_\infty}(j_+\dD_{\dG_m}\cH(0,\alpha;\beta))\big)[-2]\\
		&\cong \DE{\dA^1_\infty}\EE j_{\infty !!} \SolE{(\dG_m)_\infty}(\dD_{\dG_m}\cH(0,\alpha;\beta))[-2]\\
		&\cong \DE{\dA^1_\infty}\EE j_{\infty !!}\DE{(\dG_m)_\infty} \SolE{(\dG_m)_\infty}(\cH(0,\alpha;\beta)).
	\end{align*}
The second and last isomorphisms follow from Lemma~\ref{lemmaAlgRHCompat}(iii) and the third one from Lemma~\ref{lemmaAlgRHCompat}(iv), using that $j_+(\bullet)=\cO_{\dA^1}(*0)\otimes_{\cO_{\dA^1}} j_+(\bullet)$. Since $\SolE{(\dG_m)_\infty}(\cH(0,\alpha;\beta))$ has a $K$-structure by Theorem~\ref{theo:MainTheo}, it follows from Lemma~\ref{lemmaExtensionCompatibility} that $\SolE{\dA^1_\infty}(j_\dagger\cH(0,\alpha;\beta))$ has a $K$-structure.

Finally, let us study the left-hand side of \eqref{eq:canonicalMorphA1}:
Since
$$
j_+\cH(0,\alpha;\beta)=j_+j^+ j_\dagger\cH(0,\alpha;\beta)=(j_\dagger\cH(0,\alpha;\beta))(*0),
$$
we obtain from \cite[Corollary 9.4.11]{DK16} that
$$\SolE{\dA^1_\infty}(j_+\cH(0,\alpha;\beta))\cong \pi^{-1}\dC_{\dG_m}\otimes \SolE{\dA^1_\infty}(j_\dagger\cH(0,\alpha;\beta)),$$
so this object also has a $K$-structure. We can now see that the canonical morphism \eqref{eq:canonicalMorphA1} is induced by the canonical inclusion of sheaves $j_!\dC_{\dG_m}\to\dC_{\dA^1}$, and hence it is compatible with the $K$-structure. This completes the proof.
\end{proof}

After the previous discussion, the only case left is when some of the numbers $\beta_j$ equals zero.
This is treated in the following Theorem, which also summarizes our results.
Notice that the assumption $n\geq m$ that we still need to make is not very restrictive: If $n<m$,
then one replaces the given system $\cH(\alpha;\beta)$ by the system
$$
\cH(-\beta_1,\ldots,-\beta_m;-\alpha_1,\ldots,-\alpha_n)=\widetilde{\kappa}^+\text{inv}_+\cH(\alpha;\beta),
$$
where $\widetilde{\kappa}:\Gm\rightarrow \Gm$, $q\mapsto (-1)^{n-m}\cdot q$,
and formulates the hypothesis for this one.
\begin{theorem}\label{theo:MainTheo2}
Let $n\geq m$, and let $\alpha_1,\ldots,\alpha_n,\beta_1,\ldots,\beta_m\in[0,1)$ be given.
Let $r\in \dN_0$ such that
$\beta_1=\ldots=\beta_r=0$ and $\beta_j\neq 0$ for $j\in\{r+1,\ldots,m\}$. Suppose moreover that $\alpha_i\neq\beta_j$ for all $i,j$. Let $L\subset \dC$ as before (i.e.\ containing all of the numbers $e^{2\pi i\alpha_i}$ and $e^{2\pi i \beta_j}$).
Let $K\subset L$ be finite Galois, with $G=\Gal(L/K)$ such that
\begin{enumerate}
    \item $\gamma:=(\beta_{r+1},\ldots,\beta_m,\alpha_1,\ldots,\alpha_n)\in (0,1)^{N-r}$ is $G$-good, with $k=N-r=n+m-r$, $l=1$.
    \item The function
    \begin{equation}\label{eq:LaurPolMainTheo2}    f:=x_{r+1}+\ldots+x_m+\frac{1}{x_{m+1}}+\ldots+ \frac{1}{x_N}+q\cdot x_{r+1}\cdot\ldots\cdot x_N\in \cO_\dG =\cO_{\dG_m^{N-r}}\boxtimes\cO_{\Gm}
    \end{equation}
    lies in $\cO_\dG^{\im(\rho)}$ where again $\rho$ is the action of the group $G$ from Definition \ref{def:Ggood}. Notice that we use $x_{r+1},\ldots,x_N,q$ as coordinates on $\dG$ here.
\end{enumerate}
Then there exists $H_K\in \EbRcI{K}{(\Gm)_\infty}$ such that $\SolE{(\Gm)_\infty}(\cH(\alpha;\beta))\cong \pi^{-1}\dC_{\Gm} \otimes_{\pi^{-1}K_{\Gm}} H_K$. In particular, the enhanced ind-sheaf $\SolE{(\Gm)_\infty}\big(\cH(\alpha;\beta)\big)$ has a $K$-structure.
\end{theorem}
\begin{proof}
If $r=0$, then the result follows directly from the previous Corollary \ref{cor:CaseNonZero}.
Otherwise, we will proceed by induction on $r$. Let us write
$$
\cH':=\cH(\alpha_1,\ldots,\alpha_n;\underbrace{0,\ldots,0}_{r-1\; \textup{times}},\beta_{r+1},\ldots,\beta_m)
$$
We can assume that $\SolE{(\Gm)_\infty}(\cH')$ is defined over $K$: for $r=1$, this follows again from Corollary \ref{cor:CaseNonZero}, using our assumptions on $\gamma$ and on $f$.
If $r>1$ this is precisely the induction hypothesis.
Now we use the second formula in \cite[Proposition 5.3.3]{Ka}, which states that
$$
\cH(\alpha;\beta)= \kappa^+\text{inv}_+j^+\FL(j_+\cH').
$$
Similarly to the proof of Corollary~\ref{cor:CaseNonZero}, the functors involved in this formula correspond, via $\SolE{}(\bullet)$ to the functors $\mathrm{E}\kappa_\infty^{-1}$, $\mathrm{E}\,\mathrm{inv}_{\infty !!}$, $\mathrm{E}j_\infty^{-1}$, $\mathrm{E}p_{2\infty !!}(\dE^{qw}\conv \mathrm{E}p_{1\infty}^{-1}(\bullet))[1]$ and
$\mathrm{E}j_{\infty !!}$. They all preserve $K$-structures by
Lemma~\ref{lemmaExtensionCompatibility}. Since, as just explained,
$\SolE{(\Gm)_\infty}(\cH')$ is defined over $K$, we therefore obtain that also
$\SolE{(\Gm)_\infty}\left(\cH(\alpha;\beta)\right)$ has a $K$-structure, which is what we had to prove.
\end{proof}
\label{page:proofMainTheoIntro}
We obtain the version of this theorem stated in the introduction (Theorem~\ref{theo:MainTheoIntro}) as a simply consequence.
\begin{proof}[Proof of Theorem \ref{theo:MainTheoIntro}]
Under the hypotheses of Theorem \ref{theo:MainTheoIntro}, let $r\in\{1,\ldots,m\}$
be such that $\beta_1=\ldots=\beta_r=0$ and $\beta_j\neq 0$ for $j\in\{r+1,\ldots,m\}$.
It is assumed by the hypotheses of the theorem that $G:=\Gal(L/K)$ induces actions on
$\{e^{2\pi i\alpha_1},\ldots,e^{2\pi i\alpha_n}\}$ and on $\{e^{2\pi i\beta_1},\ldots,e^{2\pi i\beta_m}\}$, but since $e^{2\pi i \beta_j}=1$ for $1\leq j \leq r$, and this value is fixed by any element of $G$, it reduces to an action of $G$ on the set
$e^{2\pi i \alpha_1},\ldots, e^{2\pi i \alpha_n}$ and on the set
$e^{2\pi i \beta_{r+1}},\ldots, e^{2\pi i \beta_m}$. This action can be looked at as a group homomorphism $G\rightarrow S_{m-r}\times S_n$ (where $S_k$ is the symmetric group on $k$ elements), and therefore yields a natural action $\rho$ of $G$ on $\dG_m^{N-r}\times \Gm$ by permutation of the first $m-r$ and the next $n$ coordinates (and by leaving invariant the last coordinate). Unwinding Definition \ref{def:Ggood}, this means exactly that the vector $\gamma=(\beta_{r+1},\ldots,\beta_m,\alpha_1,\ldots,\alpha_n)$ is $G$-good. More precisely, the eigenvalues $e^{2\pi i \beta_s}$ and $e^{2\pi i \beta_t}$ of the monodromy operator of $\Sol{\dG^{N-r}\times \Gm}(\cO_{\dG^{N-r}\times \Gm}^\gamma)$ corresponding to a loop around the divisors $x_s=0$ and $x_t=0$ (for $t,s\in\{r+1,\ldots,m\}$) are exchanged by both the action of $\rho$ and by $L$-conjugation (and similarly for the monodromy eigenvalues $e^{2\pi i \alpha_s}$ and $e^{2\pi i \alpha_t}$ for $s,t\in\{1,\ldots,n\}$ corresponding to a loop around $x_{m+s}$ and $x_{m+t}$), and therefore the isomorphisms in formula
\eqref{eq:ConjugationCond} as well as the compatibilities in formula \eqref{eq:Cocycle} hold true.

Moreover, the function $f$ from equation \eqref{eq:LaurPolMainTheo2} lies in $\cO_{\dG}^{\im(\rho)}$, since $\rho$ acts via permutation of the first $m-r$ and the of
the last $n$ coordinates individually.
Then the result follows from the previous Theorem \ref{theo:MainTheo2}.
\end{proof}
For an irregular holonomic $\cD$-module, the perverse sheaf of  solutions is not a primary object of study. Nevertheless, it is worth mentioning that in the situation just studied, this object carries a $K$-structure as well.
\begin{corollary}\label{cor:Sol}
Under the assumptions of Theorem~\ref{theo:MainTheo2} the perverse sheaf of solutions $\Sol{\Gm}(\mathcal{H}(\alpha;\beta))[1]$ has a $K$-structure.
\end{corollary}
\begin{proof}
Since $\Sol{\Gm}(\mathcal{H}(\alpha;\beta))\cong \mathsf{sh}_{(\Gm)_\infty}\SolE{(\Gm)_\infty}(\mathcal{H}(\alpha;\beta))$, this follows directly from Corollary~\ref{corSheafificationExtension}.
\end{proof}

In particular, we get the following result in the regular case.
\begin{corollary}\label{cor:LocSystem}
Assume that $n=m$ and the hypotheses of Theorem~\ref{theo:MainTheo2} are satisfied. Then the $\dC$-local system $\Sol{\Gm\setminus\{1\}}(\mathcal{H}(\alpha;\beta)|_{\Gm\setminus\{1\}})$ is the complexification of a $K$-local system.
\end{corollary}
\begin{proof}
It follows directly from Corollary~\ref{cor:Sol} and Lemma~\ref{lemmaExtensionCompatibility} that $\Sol{\Gm}(\mathcal{H}(\alpha;\beta))|_{\Gm\setminus\{1\}}$ has a $K$-structure. Moreover, this object is a local system, since $\cH(\alpha;\beta)|_{\Gm\setminus\{1\}}$ is an integrable connection (its singularities of a hypergeometric system with $m=n$ are at $0$, $1$ and $\infty$). Hence, it is the complexification of a local system over $K$, which follows as in the proof of Lemma~\ref{lemmaGstructureSheaves}.
\end{proof}

\section{Applications}

In this section, we will discuss a few interesting cases in which Theorem \ref{theo:MainTheo2} can be applied.
The first one concerns real structures and is inspired by
\cite[Theorem 2]{Fedorov}, which we will reprove afterwards as a simple corollary.
\begin{theorem}\label{theo:RealStruct}
Let $n\geq m$, consider numbers $\alpha_1,\ldots,\alpha_n,\beta_1,\ldots,\beta_m
\in[0,1)$, with $\alpha_i\neq \beta_j$, and let $s\in\{0,\ldots,n\}$ and $r\in\{0,\ldots,m\}$ such that
\begin{enumerate}
\item $0=\alpha_1=\ldots=\alpha_s<\alpha_{s+1}\leq\ldots\leq \alpha_n<1$,
\item $0=\beta_1=\ldots=\beta_r<\beta_{r+1}\leq\ldots\leq\beta_m<1$,
\item $\alpha_{s+i}+\alpha_{n+1-i}=1$ for all $i\in\{1,\ldots,n-s\}$ and
\item $\beta_{r+j}+\beta_{m+1-j}=1$ for
all $j\in\{1,\ldots,m-r\}$.
\end{enumerate}
(obviously, since $\alpha_i\neq \beta_j$, at most one of the numbers $r$ and $s$ can be positive). Then $\SolE{(\Gm)_\infty}(\cH(\alpha;\beta))$ has an $\dR$-structure in the sense of Definition \ref{def:KStructure}.
\end{theorem}
\begin{proof}
We put
$\gamma:=(\beta_{r+1},\ldots,\beta_m,\alpha_1,\ldots,\alpha_n) \in [0,1)^{n+m-r}$.
Take $L$ to be equal to $\dC$ and $K=\dR$, so that
$G=\Gal(\dC/\dR)=\dZ/2\dZ$. We claim that with these choices, $\gamma$ is $G$-good. Namely,  consider the action $\varrho:G  \longrightarrow  \Aut(\dG)$ (where $\dG:=\dG_m^{n+m-r}\times\Gm
=\Spec[x_{r+1}^\pm,\ldots,x_{m+n}^\pm,q^\pm]$)
such that $\varrho_{[1]}(x_{r+j})=x_{m+1-j}$ and $\varrho_{[1]}(x_{m+s+i})=x_{n+m-i}$ (it is readily checked that $\varrho_{[1]}$ is an involution, thus defining an action of $G$). Then assumptions 3.\ and 4.\ imply the condition in equation \eqref{eq:ConjugationCond} for $g=[1]\in \dZ/2\dZ$, notice that
in this case, the $g$-conjugate of $\cF^\gamma$ is simply the ordinary conjugate $\overline{\cF^\gamma}$. Hence $\gamma$ is $G$-good. Moreover, it is clear that the Laurent polynomial
$$
f=x_{r+1}+\ldots+x_m+\frac{1}{x_{m+1}}+\ldots+ \frac{1}{x_{m+n}}+q\cdot x_{r+1}\cdot\ldots\cdot x_{m+n}
$$
is invariant under $G$ (more precisely, $f\in \cO_\dG^{\im(\varrho)}$) since $G$ acts simply by exchanging pairs of the first $m-r$ and the last $n$ coordinates. Hence we can apply Theorem \ref{theo:MainTheo2}, which tells us that
$\SolE{\Gm}(\cH(\alpha;\beta))$ has an $\dR$-structure, i.e.\ is obtained via extension of scalars from an enhanced ind-sheaf defined over $\dR$.
\end{proof}

As a consequence of Corollary~\ref{cor:LocSystem}, we can now easily get back (the Betti structure part of) Fedorov's result \cite[Theorem 2]{Fedorov} here.
\begin{corollary}
Let numbers $\alpha_1,\ldots,\alpha_n$ and $\beta_1,\ldots,\beta_n$
in $[0,1)$ be given and assume that they satisfy the assumptions of the previous theorem. Then the local system on $\dP^1\setminus\{0,1,\infty\}$ associated to the corresponding hypergeometric equation via the Riemann--Hilbert correspondence is the complexification of a local system of real vector spaces.
\end{corollary}

Next we consider the case when all $\alpha_i,\beta_j$ are rational. Then the field $L$ from above can be chosen to be cyclotomic, more precisely,
let $c\in \dZ\backslash \{0\}$ such that $c\alpha_i,c\beta_j\in \dZ$, and put $L:=\dQ(\zeta)$, where $\zeta$ is a primitive $c$-th root of unity, so that
$\SolE{(\Gm)_\infty}(\cH(\alpha;\beta))$ is a priori defined over $L$. Let $H=\Gal(L/\dQ)\cong(\dZ/c\dZ)^*$. For any $g\in H$, and for any $\delta\in [0,1)$ with
$e^{2\pi i\delta}\in L$ we write $\rho_g(\delta)=\widetilde{\delta}\in[0,1)$ if
$$
g.e^{2\pi i \delta}= e^{2\pi i \widetilde{\delta}}.
$$
In this situation, put $M:=\{\beta_j\}_{j=1,\ldots,m}$, $N:=\{\alpha_i\}_{i=1,\ldots,n}$.
\begin{lemma}\label{lem:CycloFields}
Let $G < H$ be a subgroup such that $\rho_g(M)\subset M$ and
$\rho_g(N)\subset N$ for all $g\in G$.
Then $\SolE{(\Gm)_\infty}(\cH(\alpha;\beta))$ has a  $K$-structure in the sense of Definition \ref{def:KStructure}, where $K:=L^G$ is the fixed field of $G$.
\end{lemma}
\begin{proof}
First we remark that the inclusions $\rho_g(M)\subset M$ and
$\rho_g(N)\subset N$  are automatically equalities, and that we obtain an action
$$
\rho: G \longrightarrow S(M)\times S(N)\cong S_m\times S_n.
$$
where we denote by $S(M)$ resp. $S(N)$ the group of permutations of the sets $M$ and $N$, respectively. Notice that if $r\in\{1,\ldots,m\}$ is as before, i.e.
$\beta_j=0$ for $j=1,\ldots,r$ and $\beta_j\neq 0$ for $j>r$, then since necessarily $\rho(0)=0$, this action factors over
$S(\{\beta_{r+1},\ldots,\beta_m\})\times S(N)\cong S_{m-r}\times S_n$.

By construction we have
$$
\overline{\cF^\gamma}^g = \cF^{\rho^{-1}_g(\gamma)}
$$
for each $g\in G$, where $\rho_g(\gamma):=
(\rho_g(\beta_{r+1}),\ldots, \rho_g(\beta_m),
\rho_g(\alpha_1),\ldots,\rho_g(\alpha_n))$.
Again, since $G$ acts on $\dG_m^{m-r+n}$ via symmetry groups in the first $m-r$ and the last $n$ coordinates, we have that the function $$
f=x_{m+r}+\ldots+x_m+\frac{1}{x_{m+1}}+\ldots+\frac{1}{x_{m+n}}+
q\cdot x_{r+1}\cdot\ldots\cdot x_{m+n}
$$
lies in $\cO^{\im(\varrho)}_\dG$. Then the result follows by applying Theorem \ref{theo:MainTheo2}.
\end{proof}

Notice that if the hypotheses of Theorem \ref{theo:RealStruct} are satisfied and if we suppose moreover that $\alpha_i,\beta_j\in \dQ$, then Theorem \ref{theo:RealStruct} follows as a special case from Lemma \ref{lem:CycloFields}, since the fixed field $K$ will automatically be contained in $\dR$, and hence $\SolE{(\Gm)_\infty}(\cH(\alpha;\beta))$
acquires an $\dR$-structure as well.

Finally, if in Lemma \ref{lem:CycloFields} we have $\rho_g(M)\subset M$ and $\rho_g(N)\subset N$ for all $g\in H$, then we automatically get that $\cH(\alpha;\beta)$ is defined over $\dQ$. This condition can actually be rephrased in a nicer way.

\begin{theorem}\label{theo:RatStruct}
Let $s\in\{0,\ldots,n\}$ and $r\in\{0,\ldots,m\}$ be as in Theorem \ref{theo:RealStruct}. Suppose that there exist non-negative integers $e,f$
and positive integers
$r_1,\ldots, r_e$ and $s_1,\ldots,s_f$ such that $n-s=\varphi(r_1)+\ldots+\varphi(r_e)$ and $m-r=\varphi(s_1)+\ldots+\varphi(s_f)$ where  $\varphi$ is Euler's $\varphi$-function.
If now we have
$$
\prod_{i=s+1}^n (q\partial_q-\alpha_i)=\prod_{i'=1}^e \prod_{d\in (\dZ/r_{i'}\dZ)^*} \left(q\partial_q-\frac{d}{r_{i'}}\right)
$$
and
$$
\prod_{j=r+1}^m (q\partial_q-\beta_j)=\prod_{j'=1}^f \prod_{d\in (\dZ/s_{j'}\dZ)^*} \left(q\partial_q-\frac{d}{s_{j'}}\right),
$$
then $\SolE{(\Gm)_\infty}(\cH(\alpha;\beta))$ has a $\dQ$-structure.
\end{theorem}
\begin{proof}
It is elementary to verify that the assumption implies that the group $G$ in Lemma \ref{lem:CycloFields} can be taken to be the full Galois group $\Gal(L/\dQ)\cong (\dZ/c\dZ)^*$ (notice that $c$ is divisible by $\text{lcm}(r_1,\ldots,r_e,s_1,\ldots,s_f)$), and then $K=L^G=\dQ$.
\end{proof}

A special case of this result is worth mentioning, since it is related to various examples of mirror symmetry for toric orbifolds.
\begin{corollary}\label{cor:RatStructuresWeightedProjSp}
Let $m,n$ and $r,s$ be as above, and suppose that there are $p,q\in \dZ_{\geq 0}$ and
$w_1,\ldots, w_p, v_1,\ldots,v_q \in \dZ_{>0}$ such that $n-s=w_1+\ldots+w_p-p$ and $m-r=v_1+\ldots+v_q-q$ and such that
$$
\prod_{i=s+1}^n (q\partial_q-\alpha_i)=\prod_{i'=1}^p \prod_{d=1}^{w_{i'}-1} \left(q\partial_q-\frac{d}{w_{i'}}\right)
$$
and
$$
\prod_{j=r+1}^m (q\partial_q-\beta_j)=\prod_{j'=1}^q \prod_{d=1}^{v_{j'}-1} \left(q\partial_q-\frac{d}{v_{j'}}\right).
$$

Then $\SolE{(\Gm)_\infty}(\cH(\alpha;\beta))$ has a $\dQ$-structure.
\end{corollary}
\begin{proof}

We have the following identity for any $w\in \dZ_{>0}$
$$
\begin{array}{rcl}
\D \prod_{d=1}^{w-1} \left(q\partial_q -\frac{d}{w}\right)
&=&
\D \prod_{k<w, k|w} \; \prod_{d<w,\textup{gcd}(d,w)=k}  \left(q\partial_q -\frac{d}{w}\right)\\ \\
&=&
\D \prod_{k<w, k|w} \; \prod_{d<w,\textup{gcd}(d,w)=k}  \left(q\partial_q -\frac{d/k}{w/k}\right)\\ \\
&=&
\D \prod_{k<w, k|w} \; \prod_{b<w/k,\textup{gcd}(b,w/k)=1}  \left(q\partial_q -\frac{b}{w/k}\right)
\stackrel{(*)}{=}
\D \prod_{k<w, k|w} \; \prod_{b<k,\textup{gcd}(b,k)=1}  \left(q\partial_q -\frac{b}{k}\right) \\ \\
&=&
\D \prod_{k<w, k|w} \; \prod_{b\in(\dZ/k\dZ)^*}  \left(q\partial_q -\frac{b}{k}\right),
\end{array}
$$
where the equality $(*)$ comes from the fact that
the map from $\{0<k<w\,|\,k|w\}$ to itself sending $k$ to
$w/k$ is a bijection.

By applying this identity to the two monomial operators
$$
\D \prod_{i'=1}^p \prod_{d=1}^{w_{i'}-1} \left(q\partial_q-\frac{d}{w_{i'}}\right)
\quad\quad\textup{and}\quad\quad
\prod_{j'=1}^q \prod_{d=1}^{v_{j'}-1} \left(q\partial_q-\frac{d}{v_{j'}}\right).
$$
and by using the previous theorem, we obtain the desired result.
\end{proof}

\begin{remark}
Suppose that we are given $(\alpha;\beta)\in [0,1)^N$ as before such that the following variant of the assumptions of the previous theorem holds: There exist
$p,q\in \dZ_{\geq 0}$ and
$w_2,\ldots, w_p, v_1,\ldots,v_q\in \dZ_{>0}$ such that $n=w_1+\ldots+w_p$ and $m=v_1+\ldots+v_q$ and such that
$$
\prod_{i=1}^n (q\partial_q-\alpha_i)=\prod_{i'=1}^p \prod_{d=0}^{w_{i'}-1} \left(q\partial_q-\frac{d}{w_{i'}}\right)
$$
(where $w_1=0$, so that necessarily $\alpha_1=0$)
and
$$
\prod_{j=1}^m (q\partial_q-\beta_j)=\prod_{j'=1}^q \prod_{d=0}^{v_{j'}-1} \left(q\partial_q-\frac{d}{v_{j'}}\right).
$$
Notice that here both some of the numbers $\alpha_i$ and some of the numbers $\beta_j$ are equal to zero, so that
the corresponding module $\cH(\alpha;\beta)$ is no longer necessarily irreducible.
We now consider the following Laurent polynomial:
$$
\widetilde{f}:=x_1+\ldots+x_q+\frac{1}{x_{q+1}}+\ldots+\frac{1}{x_{q+p-1}}+q\cdot x_1^{v_1}\cdot\ldots\cdot x_q^{v_q}
\cdot x_{q+1}^{w_2}\cdot\ldots\cdot x_{p+q-1}^{w_p} \in \cO_{\widetilde{\dG}},
$$
where this time $\widetilde{\dG}=\dG_m^{q+p-1}\times\Gm$. Let again
$p:\widetilde{\dG}\twoheadrightarrow \Gm$ be the projection to the last factor, then it can be shown along the lines of Proposition \ref{prop:HypGMSystem} that we have
$$
\kappa^+\cH^0 p_+ \cE^{\widetilde{f}}_{\widetilde{\dG}}  \cong \cH(\alpha;\beta).
$$
This isomorphism is essentially well-known, e.g. the Laurent polynomial $\widetilde{f}$ appears, in the case where $q=0$, as the Landau--Ginzburg model for the quantum cohomology of weighted projective spaces, see \cite{DS2} for a thorough discussion of this example.

As a consequence, one can show directly that this (reducible) module $\cH(\alpha;\beta)$ has a $\dQ$-structure since it is constructed via standard functors from an object (namely $\cE^{\widetilde{f}}_{\widetilde{\dG}}$) which already has a $\dQ$-structure. This is in contrast to the cases discussed before, where we start with the module $\cE^{\gamma,f}$, on which there is no a priori Betti structure.

One can also relate (for the case of parameters $\alpha,\beta$ satisfying the hypotheses of this remark) the two approaches by comparing the Laurent polynomials $f$ and $\widetilde{f}$. This yields an interesting geometric explanation for these two different approaches to the existence of Betti structures. We plan to discuss these issues in a subsequent work.
\end{remark}

\section{Consequences for Stokes matrices}
\label{sec:Stokes}

We apply our results to questions regarding the Stokes matrices for hypergeometric system at the irregular singular point. In Section 9.8 of \cite{DK16}, the authors explain how the Stokes matrices or Stokes multipliers are encoded in the enhanced ind-sheaf of the solutions.

We assume to be in the situation of Theorem \ref{theo:MainTheo2} so that the enhanced solutions carry a $K $-structure
\[
\SolE{(\Gm)_\infty}(\cH(\alpha;\beta))\cong \pi^{-1}\dC_{\Gm} \otimes_{\pi^{-1}K_{\Gm}} H_K
\]
for some $H_K\in \mathrm{E}^\mathrm{b}_{\rc}(\mathrm{I}K_{(\Gm)_\infty})$.

Recall that we considered parameters $\alpha_1, \ldots, \alpha_n, \beta_1, \ldots, \beta_m \in \dC$ with $n>m $. Then $\cH(\alpha;\beta) $ is irregular singular at infinity and if we write $d:=n-m$ it is ramified of degree $d $. Let us denote by $\rho:\Gmu \to \Gm$ the ramification map $\rho:u \mapsto u^d=q$. As usual, we will consider the pull-back of $\cH(\alpha;\beta)$ with respect to $\rho$ and study the Stokes matrices of the resulting enhanced ind-sheaf with the induced $K$-structure (see Lemma \ref{lemmaExtensionCompatibility}(ii)):
\begin{equation}\label{eq:Kstrsol}
\EE \rho^{-1} \SolE{(\Gm)_\infty}(\cH(\alpha;\beta))\cong \pi^{-1}\dC_{\Gmu} \otimes_{\pi^{-1}K_{\Gm}} \EE \rho^{-1} H_K.
\end{equation}
Let us write $\widetilde{H}_K := \EE \rho^{-1} H_K \in \mathrm{E}^\mathrm{b}_{\rc}(\mathrm{I}K_{(\Gmu)_\infty})$ (cf.\ \cite[Proposition 4.9.11]{DK16} for the compatibility of $\dR$-constructibility with pull-backs).

The pull-back $\cH(\alpha;\beta)$ is of slope one and there is a finite set $C_1 \subset \dC^\times$ such that the formal exponential factors of $\cH(\alpha;\beta)$ are the elements of
\[
\{e^0\} \cup \{ e^{cu} \mid c \in C_1 \}
\]
(see \cite{SabStokes} for these notions). Let us write $C:=\{0\} \cup C_1$.

\begin{remark}\label{rem:Katz}
With additional assumptions on the parameters $(\alpha;\beta)$, the exponential factors can be determined rather easily. If the non-resonant parameters satisfy that $d\alpha_j \not \in \dZ$ for all $j $ and that the module is not Kummer induced (see \cite[Kummer Recognition Lemma 3.5.6]{Ka}), a theorem of N. Katz \cite[Theorem 6.2.1]{Ka} relates $\cH(\alpha;\beta)$ with the Fourier-Laplace transform of a regular singular hypergeometric module. Applying the stationary phase formula of C. Sabbah \cite{SaStat}, one then deduces that $C_1$ is given as $C_1=\{ d\cdot \zeta \mid \zeta \in \mu_d \}$, where $\mu_d$ is the group of $d$-th roots of unity.
\end{remark}

We would like to apply results from \cite{MocCurveTest}. There is a natural pre-order for subanalytic functions $f,g $ on a bordered space $(M^\circ,M)$ defined as
\begin{equation}\label{eq:preorder}
f \prec g :\Leftrightarrow f-g \text{ is bounded from above on $U \cap M^\circ$ for any relatively compact subset $U$ of $M $.}
\end{equation}
It induces an equivalence relation by setting $f \sim g :\Leftrightarrow f\prec g $ and $g \prec f$. We will write $[f]$ for the equivalence class of a function.

Let $\Delta$ be a small open neighbourhood of $\infty$ and let $\varpi\colon \widetilde{\Delta}\to \Delta$ be the oriented real blow-up of $\Delta$ at $\infty$. Let us write $\Delta^\circ \vcentcolon= \Delta \smallsetminus \{ \infty \}$ with its inclusion $\iota\colon \Delta^\circ \hookrightarrow \widetilde{\Delta}$, and consider the bordered space $\bDelta \vcentcolon=(\Delta^\circ, \widetilde{\Delta})$.  Since we are interested in the local situation at infinity, we will restrict all sheaves to $\Delta$ (or $\widetilde{\Delta}$). For example, we consider the enhanced ind-sheaf \(\SolE{(\Delta^\circ,\Delta)}((\rho^+\cH(\alpha;\beta))|_{\Delta^\circ})\) instead of the full version on $(\Gmu)_\infty$. Let us remark that we consider $\varpi$ as a morphism of bordered spaces
\(
\varpi\colon \bDelta=(\Delta^\circ,\widetilde{\Delta}) \to (\Delta^\circ,\Delta)
\)
also.

As explained in \cite[Section 9]{DK16}, we know that we can cover $\Delta^\circ$ by sectors $\Sigma_k$ such that we have trivializations of the enhanced solutions of $\rho^+\cH(\alpha;\beta)$ of the form
\[
    \pi^{-1}\dC_{\Sigma_k}\otimes \SolE{(\Delta^\circ,\Delta)}(\rho^+\cH(\alpha;\beta))\cong \pi^{-1}\dC_{\Sigma_k} \otimes \bigoplus_{c\in C} \big(\dE^{\mathrm{Re}(cu)}_{(\Delta^\circ,\Delta),\dC}\big)^{r_c}
\]
with the index set $C$ from above. We write $\dE^{\mathrm{Re}(cu)}_{(\Delta^\circ,\Delta),\dC}$ here for the enhanced ind-sheaf $\dE^{\mathrm{Re}(cu)}_\dC$ defined in \eqref{eq:DefExp} in order to emphasize the bordered space on which it lives. Since
\[
\EE \varpi^{-1} \dE^{\mathrm{Re}(cu)}_{(\Delta^\circ,\Delta),\dC} \cong
\dE^{\mathrm{Re}(cu)}_{\bDelta,\dC} ,
\]
we deduce the trivializations
\begin{equation}\label{eq:sectorK}
    \pi^{-1}\dC_{\Sigma_k}\otimes \mathrm{E} \varpi^{-1}  \SolE{(\Delta^\circ,\Delta)}(\rho^+\cH(\alpha;\beta))\cong \pi^{-1}\dC_{\Sigma_k} \otimes \bigoplus_{c\in C} \big(\dE^{\mathrm{Re}(cu)}_{\bDelta,\dC}\big)^{r_c}.
\end{equation}
We will now work on the bordered space $\bDelta$ and hence omit the subscript by simply writing $\dE^{\mathrm{Re}(cu)}_{\dC}$ again.

It is important to remark that the induced filtration
\begin{equation}\label{eq:enhStokesFiltr}
F_\mathfrak{a} \big(\pi^{-1}\dC_{\Sigma_k}\otimes \mathrm{E} \varpi^{-1} \SolE{(\Delta^\circ,\Delta)}(\rho^+\cH(\alpha;\beta))\big) \cong \pi^{-1}\dC_{\Sigma_k} \otimes \bigoplus_{c\in C: [cu]\prec \mathfrak{a}} \big(\dE^{\mathrm{Re}(cu)}_\dC\big)^{r_c}
\end{equation}
indexed by classes $\mathfrak{a}$ of subanalytic functions is well-defined, i.e.\ does not depend on the choice of the isomorphism in \eqref{eq:sectorK} (cf.\
\cite[Lemma 5.15]{MocCurveTest}).

Since we have pole order at most one at infinity in the exponential factors, the following arguments yield that we can obtain these splittings on two sectors, each of width slightly greater than $\pi$: By classical analysis (cf.\ \cite{balseretal}) this is well-known for asymptotic solutions in two such sectors. Recall the notion of $\cD^\cA$-modules on the real oriented blow-up $\widetilde\Delta$ from \cite[§7.2]{DK16}. The classical result induces a corresponding splitting as $\cD^\cA$-modules and we deduce the existence of a splitting as in \eqref{eq:sectorK} for two sectors of width slightly greater than $\pi$ from \cite[Proposition 3.5]{ItoTakeuchi} (see also \cite[Proposition 3.1]{AHGauss} for details in the one-dimensional case).

Let us choose sectors $S_\pm$ (in $\widetilde{\Delta}$) such that the $\varpi^{-1}(\infty)\cap S_\pm $ cover $\varpi^{-1}(\infty)$, and let
\[
\sigma_+ \cup \sigma_- = S_+ \cap S_-
\]
be the union of the two smaller sectors $\sigma_\pm$, the overlaps of the sectors $S_\pm$. The choice of the sectors $S_\pm$ has some impact on the Stokes matrices one wants to compute -- in principle it amounts to the action of a braid group.

Let us denote by $\cL:=
\big(\mathsf{sh}_{(\Gmu)_\infty}\SolE{(\Gmu)_\infty}(\rho^+\cH(\alpha;\beta))\big)\big|_{\Delta^\circ}$ the local system of solutions on the punctured disc and by $\widetilde{\cL}:=\iota_\ast \cL$ its extension to the boundary. The Stokes filtration on the local system $\widetilde{\cL}|_{\varpi^{-1}(\infty)}$ is the filtration inherited from the filtration in \eqref{eq:enhStokesFiltr} via the sheafification functor. On the sectors $S_\pm$, we have splittings of these filtrations as in \eqref{eq:sectorK}, say
\[
\psi_{\pm}\colon \widetilde{\cL}|_{(\varpi^{-1}(\infty)\cap S_\pm)} \stackrel{\cong}{\to}
\bigoplus_{c\in C} (\cE_{c,\pm})^{r_c}
\]
where $\cE_{c,\pm} $ is the constant rank one local system on the interval $\varpi^{-1}(\infty)\cap S_\pm$ coming from $\iota_*\mathsf{sh}_{\bDelta}\dE^{\mathrm{Re}(cu)}_\dC$. The Stokes matrices are defined to be the matrices representing the transition isomorphims on the overlaps
\[
\cS_+ = \big( \psi_- \circ \psi_+^{-1})|_{\varpi^{-1}(\infty)\cap\sigma_+} \text{ and }
\cS_- = \big( \psi_- \circ \psi_+^{-1})|_{\varpi^{-1}(\infty)\cap\sigma_-}.
\]
(Let us remark that there are different conventions and the Stokes matrices are sometimes also defined as the inverse isomorphisms, which is not important for our purposes. Here, we did not describe the orientation of the sectors and their overlaps explicitly.) In order to prove that one can arrive at Stokes matrices with entries in the subfield $K \subset \dC$, we have to show that the local system $\widetilde{\cL}_{\varpi^{-1}(\infty)}$, its Stokes filtration and splittings can be defined over $K$. We are indebted to T.\ Mochizuki for pointing out the idea how to prove this.

First, note that both sides of \eqref{eq:sectorK} have a $K$-structure,
so that we can write this isomorphism in the form
\begin{equation}\label{eq:tildeHK}
 \pi^{-1}\dC_{\Delta^\circ} \otimes_{\pi^{-1}K_{\Delta^\circ}} \pi^{-1}K_{\Sigma_k} \otimes \EE \varpi^{-1}\widetilde{H}_K \cong
\pi^{-1}\dC_{\Delta^\circ} \otimes_{\pi^{-1}K_{\Delta^\circ}}
\pi^{-1}K_{\Sigma_k} \otimes \bigoplus_{c\in C} \big(\dE^{\mathrm{Re}(cu)}_K\big)^{r_c}.
\end{equation}

Let us denote by $\widetilde{\cL}_K:=\iota_\ast (\mathsf{sh}_{(\Gmu)_\infty}\widetilde{H}_K)|_{\Delta^\circ}$ the associated $K$-structure of $\widetilde{\cL}$.

We know that $\EE\varpi^{-1}\widetilde{H}_K$ is an $\dR$-constructible enhanced ind-sheaf (again by \cite[Proposition 4.9.11]{DK16}) and we deduce from \cite[Lemma 4.9.9]{DK16} that there exists a subanalytic stratification $\widetilde{\Delta}=\bigsqcup_{\lambda \in \Lambda} S_\lambda$ refining $\widetilde{\Delta}=\varpi^{-1}(\infty) \sqcup \Delta^\circ$ such that the following holds: For each stratum $S_\lambda \subset \Delta^\circ$, there exist a finite set of $\dR \cup \{\infty\} $-valued subanalytic functions $f_{\lambda,j}< g_{\lambda,j}$, say for $j=1, \ldots,m $, and isomorphisms
\begin{equation}\label{eq:decompStrat}
\pi^{-1} K_{S_\lambda} \otimes \EE\varpi^{-1}\widetilde{H}_K \cong \pi^{-1} K_{S_\lambda} \otimes \bigoplus_{j=1}^m K^\mathrm{E}_{\bDelta} \conv K_{f_{\lambda,j}\le t < g_{\lambda,j}}
\end{equation}
where (analogously to \eqref{eq:DefExp}) we write
\[
K_{f\le t <g} \vcentcolon= K_{\{ (u,t) \in \widetilde{\Delta} \times \mathsf{P} \mid u \in \Delta^\circ, t \in \dR, f(u)\le t < g(u) \}}.
\]

For each $j $, the pair $(f_{\lambda,j},g_{\lambda,j})$ then is non-equivalent in the sense of \cite[§5.2.2]{MocCurveTest}.

Let us now consider the situation around points on the boundary of the real blow-up: For all but finitely many points $p \in \varpi^{-1}(\infty)$, we find one-dimensional strata $S_\eta \subset \varpi^{-1}(\infty)$ containing $p$ and $S_\lambda \subset \Delta^\circ$ such that their union contains an open neighbourhood $U_p $ of $p $ in $\widetilde{\Delta}$ and such that \eqref{eq:decompStrat} holds over $S_\lambda$ and consequently also over $U_p^\circ=U_p \cap \Delta^\circ$. Let $Z \subset \varpi^{-1}(\infty)$ be the finite set of point where this does not hold, i.e.\ zero-dimensional strata in $\varpi^{-1}(\infty)$ or limit points of (real) one-dimensional strata in $\Delta^\circ$.
Consider a point $p \in \varpi^{-1}(\infty) \smallsetminus Z$ and apply the notations as above. Since
\[
\pi^{-1} \dC_{\Delta^\circ}\otimes_{\pi^{-1} K_{\Delta^\circ}}\pi^{-1}K_{U_p^\circ}\otimes
K^\mathrm{E}_{\bDelta} \conv K_{f_{\lambda,j}\le t < g_{\lambda,j}} \cong \pi^{-1} \dC_{U_p^\circ} \otimes \dC^\mathrm{E}_{\bDelta} \conv \dC_{f_{\lambda,j}\le t < g_{\lambda,j}}
\]
for all $j$, we deduce from \eqref{eq:decompStrat} the isomorphism
\begin{equation}\label{eq:decompStratC}
\pi^{-1} \dC_{\Delta^\circ} \otimes_{\pi^{-1} K_{\Delta^\circ}}\pi^{-1}K_{U_p^\circ} \otimes \EE\varpi^{-1}\widetilde{H}_K \cong \pi^{-1}\dC_{U_p^\circ}\otimes \bigoplus_{j=1}^m \dC^\mathrm{E}_{\bDelta} \conv \dC_{f_{\lambda,j}\le t < g_{\lambda,j}}.
\end{equation}
If $U_p^\circ$ is chosen small enough, it is contained in one of the sectors $\Sigma_k$ from \eqref{eq:tildeHK} and since
\[
\dE^{\mathrm{Re}(cu)}_\dC=\dC^{\mathrm{E}}_{\bDelta} \conv \dC_{-\mathrm{Re}(cu)\le t},
\]
we combine \eqref{eq:sectorK} and \eqref{eq:decompStratC} to obtain the isomorphism
\begin{equation}\label{eq:KMIs}
\pi^{-1}\dC_{U_p^\circ}\otimes\bigoplus_{j=1}^m \dC^\mathrm{E}_{\bDelta} \conv \dC_{f_{\lambda,j}\le t < g_{\lambda,j}} \cong \pi^{-1}\dC_{U_p^\circ}\otimes \bigoplus_{c\in C} \dC^{\mathrm{E}}_{\bDelta} \conv (\dC_{-\mathrm{Re}(cu) \le t})^{r_c}.
\end{equation}

On the basis of \cite{MocCurveTest}\footnote{We refer to the arXiv version of Mochizuki's paper. However, the paper has been reorganized in the meantime and will be published in two parts. The notation $\mathrm{Sub}_{\not\sim}^{\langle 2 \rangle, \ast}$ is the one from the reorganized article (part II). We are grateful to Takuro Mochizuki for providing the new versions. The notation of \cite{MocCurveTest} is slightly different but similar enough not to create confusion.}, let us denote by
\(
\mathrm{Sub}_{\not\sim}^{\langle 2 \rangle, \ast}(U_p^\circ,U_p)
\)
the set of either pairs $(f,g) $ of non-equivalent subanalytic functions (with $f(p)<g(p) $ pointwise) or of pairs of one subanalytic function together with $\infty$ on the bordered space $(U_p^\circ,U_p)$.

The pre-order $\prec$ from \eqref{eq:preorder} induces a pre-order on the
set $\mathrm{Sub}_{\not\sim}^{\langle 2 \rangle, \ast}(U_p^\circ,U_p)$ by setting $(f_1,g_1) \prec (f_2,g_2)$ if and only if $f_1 \prec f_2 $ and $g_1 \prec g_2$ -- where of course $f \prec \infty $ for all subanalytic $f $. The quotient with respect to the induced equivalence relation is denoted by
\(
\overline{\mathrm{Sub}}_{\not\sim}^{\langle 2 \rangle, \ast}(U_p^\circ,U_p).
\)
Now, both sides of \eqref{eq:KMIs} are associated to finite multi-subsets $(I,\mathfrak{m})$ of $\mathrm{Sub}_{\not\sim}^{\langle 2 \rangle, \ast}(U_p^\circ,U_p)$, i.e.\ finitely many elements $\mathfrak{a}$ of the latter set together with a multiplicity $m_\mathfrak{a} \in \dN$ for each, namely
\begin{itemize}
    \item $(I_\mathrm{left},\mathfrak{m}_\mathrm{left}) $ consisting of the restrictions of the pairs $(f_{\lambda,j},g_{\lambda,j})$ to $(U_p^\circ,U_p)$ and the multiplicities induced by equivalent pairs for the left hand side of \eqref{eq:KMIs}, and
    \item $(I_\mathrm{right},\mathfrak{m}_\mathrm{right}) $ being the multiset of the pairs $(-\mathrm{Re}(ct),\infty)$ with multiplicity $r_c $.
\end{itemize}
Each side of \eqref{eq:KMIs} is constructed in the obvious way from these multi-subsets. If we mimic the notation from \cite{MocCurveTestII} and write\footnote{Note that T.~Mochizuki more generally considers graded multi-sets where an additional grading information refers to a shift of the enhanced ind-sheaves as building blocks for $\mathrm{K}_M(I,\mathfrak{m})$. We don't need these gradings here, since all sheaves are concentrated in one degree.}
\[
\mathrm{K}_{(U_p^\circ,U_p)}(I,\mathfrak{m}) \vcentcolon= \bigoplus_{(f,g)\in I} \dC^\mathrm{E}_{(U_p^\circ,U_p)} \conv (\dC_{f \le t < g})^{m_{(f,g)}},
\]
the isomorphism \eqref{eq:KMIs} after pull-back via the embedding $(U_p^\circ,U_p) \hookrightarrow \bDelta $ reads as
\[
\mathrm{K}_{(U_p^\circ,U_p)}(I_\mathrm{left},\mathfrak{m}_\mathrm{left}) \cong
\mathrm{K}_{(U_p^\circ,U_p)}(I_\mathrm{right},\mathfrak{m}_\mathrm{right}).
\]
Now, due to \cite[§5.2.6]{MocCurveTest}\footnote{In loc.~cit.~the statement is referred to as a direct analogy to Lemma 5.15. In the reorganized article it is worked out in all details in \cite[Lemma 3.29]{MocCurveTestII}.} we conclude that the induced multi-subsets of $\overline{\mathrm{Sub}}_{\not\sim}^{\langle 2 \rangle, \ast}(U_p^\circ,U_p)$ coincide and so do the canonical filtrations.

Hence, we obtain an isomorphism
\begin{equation}\label{eq:Kdecomp}
\pi^{-1} K_{U_p^\circ} \otimes \mathrm{E}\varpi^{-1}\widetilde{H}_K \cong
\pi^{-1} K_{U_p^\circ} \otimes \bigoplus_{c \in C} K^\mathrm{E}_{\bDelta} \conv \big(\dE^{\mathrm{Re}(cu)}_K\big)^{r_c}
\end{equation}
and an induced filtration $F^p_{\mathfrak{a}}(\widetilde{\cL}_{K,p})$ on the stalk $\widetilde{\cL}_{K,p}$ which induces the Stokes filtration on $\widetilde{\cL}_p$ after extension of scalars from $K$ to $\dC$.

For a point $p\in Z$ consider two nearby points $p_1,p_2 \in \varpi^{-1}(\infty)$ on each side of $p$, i.e.\ in the components of a sufficiently small punctured interval at $p$. Then we have canonical isomorphisms of the stalks $\widetilde{\cL}_{p_j} \cong \widetilde{\cL}_p$ as well as for the stalks of the $K $-structure. With respect to the first isomorphisms, the Stokes filtrations on $\widetilde{\cL}$ are related by
\begin{equation}\label{eq:Fp1p2p}
F^p_\mathfrak{a} (\widetilde{\cL}_p) = F^{p_1}_\mathfrak{a} (\widetilde{\cL}_{p_1}) \cap F^{p_2}_\mathfrak{a} (\widetilde{\cL}_{p_2}),
\end{equation}
(note that both stalks on the right hand side are equal if $p$ is no Stokes direction of a pair $(cu,\mathfrak{a})$ of exponential factors, anyway.) Hence, if we define $F^p_\mathfrak{a}(\widetilde{\cL}_{K,p})$ analogously to \eqref{eq:Fp1p2p}, we obtain a filtration also on $\widetilde{\cL}_{K,p}$ inducing the one on $\widetilde{\cL}_p$.

In summary, there exist filtrations $F^p_\mathfrak{a}(\widetilde{\cL}_K)$ on any stalk $p \in \varpi^{-1}(\infty)$ inducing the Stokes filtration after scalar extension from $K$ to $\dC$.

\begin{lemma}\label{lem:gradedLocSys}
The graded objects
\[
\mathrm{Gr}^{F^p_\bullet}_\mathfrak{a}(\widetilde{\cL}_{K,p}) =
F^p_\mathfrak{a}(\widetilde{\cL}_{K,p})/ \sum_{\mathfrak{b} \prec \mathfrak{a}} F^p_\mathfrak{b}(\widetilde{\cL}_{K,p})
\]
glue to a local system of $K$-vector spaces on $\varpi^{-1}(\infty)$.
\end{lemma}
\begin{proof}
First let us remark that each interval $I \subsetneq \varpi^{-1}(\infty)$ induces canonical isomorphisms
\begin{equation}\label{eq:stalksiso}
    \widetilde{\cL}_p \cong H^0(I,\widetilde{\cL}) \cong \widetilde{\cL}_q
\end{equation}
for all $p,q \in I$. The same holds for the local system $\widetilde{\cL}_K$.

We know that the claim of the lemma is true over $\dC$, the graded objects associated to the filtrations as $\dC$-sheaves are local systems of $\dC$-vector spaces -- cf.\ e.g.\ \cite[Section 6.1]{DKmicFS} or \cite[Proposition 2.7]{SabStokes}, going back to ideas of Deligne and Malgrange (see \cite[Section IV.2]{Mal7}). Consequently, if one chooses complementary vector spaces $G^p_\mathfrak{a} \subset F^p_{\mathfrak{a}}(\widetilde{\cL}_p)$ for each $p \in \varpi^{-1}(\infty)$ such that the canonical morphism
\[
G^p_\mathfrak{a} \hookrightarrow F^p_{\mathfrak{a}}(\widetilde{\cL}_p) \to \mathrm{Gr}^p_{\mathfrak{a}}(\widetilde{\cL}_p)
\]
is an isomorphism, then each $p$ has an open neighbourhood $I_p$ such that
\(
\widetilde{\cL}_q = \bigoplus_{c\in C} G^p_{[cu]}
\)
is a splitting of the filtration for each $q \in I_p$ where we use the identifications from \eqref{eq:stalksiso}.

In the same way, we can choose complementary $K$-vector spaces $G^p_{K,\mathfrak{a}} \subset F^p_{K,\mathfrak{a}} $ for each $p$. Then $G^p_\mathfrak{a}:=G^p_{K,\mathfrak{a}} \otimes_K \dC$ is a choice as above and we know that for each $p$ the $\dC$-vector spaces  $G^p_\mathfrak{a}$ induce a local splitting of the filtration on $\widetilde{\cL}_q $ for $q$ in a neighbourhood of $p$. Therefore, the same is true for the $K$-vector spaces $G^p_{K,\mathfrak{a}}$. The claim of the lemma easily follows.
\end{proof}

We want to convince ourselves that we can find local splittings of the filtration on $\widetilde{\cL}_K$ over the same intervals as it is the case for $\widetilde{\cL}$.

\begin{lemma}\label{lemma:splittingK}
Suppose we have a splitting $\widetilde{\cL}|_I= \bigoplus_{c\in C} \cG_{I,[cu]}$ of the Stokes filtration over an interval $I \subsetneq \varpi^{-1}(\infty)$. Then there is a splitting
\[
\widetilde{\cL}_K|_I= \bigoplus_{c\in C} \cG_{K,I,[cu]}
\]
over $I$ inducing the given splitting after extension of scalars from $K$ to $\dC$.
\end{lemma}
\begin{proof}
Let us denote by $\mathrm{Gr}^F_{K,[cu]}(\widetilde{\cL}_K)$ the local systems associated to the filtered local system $\widetilde{\cL}_K$ on $\varpi^{-1}(\infty)$ as in Lemma~\ref{lem:gradedLocSys}.

Let $I$ be an interval as in the assumptions. The filtration $F^p_\bullet(\widetilde{\cL}_{K,p})$ induces a filtration on $H^0(I,\widetilde{\cL}_K)$ -- recall the identifications \eqref{eq:stalksiso} and its variant for $\widetilde{\cL}_K$. We have a surjective morphism
\[
F^p_{[cu]}(H^0(I,\widetilde{\cL}_K)) \twoheadrightarrow H^0(I, \mathrm{Gr}^F_{K,[cu]}(\widetilde{\cL}_K))
\]
for each $p \in I$, hence we also get a morphism
\begin{equation}\label{eq:bigcapsurjective}
    \bigcap_{p\in I} F^p_{[cu]}(H^0(I,\widetilde{\cL}_K)) \longrightarrow H^0(I, \mathrm{Gr}^F_{K,[cu]}(\widetilde{\cL}_K)).
\end{equation}
It is easy to see that it suffices to show that the latter is surjective, since then we can find a subspace $G_{K,[cu]} \subset \bigcap_{p\in I} F^p_{[cu]}(H^0(I,\widetilde{\cL}_K))$ such that the induced morphism
\[
G_{K,[cu]} \to H^0(I, \mathrm{Gr}^F_{K,[cu]}(\widetilde{\cL}_K))
\]
is an isomorphism. Then, if we denote by $\cG_{K,I,[cu]}$ the constant local system of $K $-vector spaces over $I$ associated to $G_{K,[cu]}$, these define a splitting of $\widetilde{\cL}_K$ over $I$.

To prove that \eqref{eq:bigcapsurjective} is surjective, observe that the given splitting $\widetilde{\cL}|_I= \bigoplus_{c\in C} \cG_{I,[cu]}$ yields that
\[
H^0(I,\cG_{I,[cu]}) \subset F^p_{[cu]}(H^0(I,\widetilde{\cL}))
\]
for all $p \in I$ and
\(
H^0(I,\cG_{I,[cu]}) \to H^0(I,\mathrm{Gr}^F_{[cu]}(\widetilde{\cL}))
\)
is an isomorphism. Consequently, the natural morphism
\begin{equation}\label{eq:bigcapsurjectiveC}
    \bigcap_{p\in I} F^p_{[cu]}(H^0(I,\widetilde{\cL})) \twoheadrightarrow H^0(I, \mathrm{Gr}^F_{[cu]}(\widetilde{\cL}))
\end{equation}
is surjective. Since \eqref{eq:bigcapsurjectiveC} is the obtained from \eqref{eq:bigcapsurjective} by extension of scalars from $K$ to $\dC$, and since this extension is a right-exact functor, it follows that the morphism \eqref{eq:bigcapsurjective} is surjective as well.
\end{proof}

We can now state and prove the final result of this section.
\begin{theorem}
Assume that we are in the situation of Theorem \ref{theo:MainTheo2} and that  moreover we have $n>m$. Put again $d\vcentcolon=n-m$, let $\rho\colon u \to u^d=q$ be the local ramification map at infinity of degree $d$ and consider the pull-back $\rho^+ \cH(\alpha;\beta)$ of the hypergeometric system.

Then there is a representation of the Stokes matrices for $\rho^+ \cH(\alpha;\beta)$ with values in the field $K$.
\end{theorem}
\begin{proof}
We pick up the notation from above. By Lemma \ref{lemma:splittingK}, we know that the local system $\widetilde{\cL}_K$ of $K$-vector spaces splits on an interval $I \subsetneq \varpi^{-1}(\infty)$ as a local system of $K$-vector spaces  whenever $\widetilde{\cL}$ splits on $I$.
In our situation, we have splittings of $\widetilde{\cL}$ over two intervals $I_\pm= \varpi^{-1}(\infty)\cap S_\pm$ and the Stokes matrices are the connecting isomorphisms between these splittings on the intersection.
We deduce from Lemma \ref{lemma:splittingK} that the splittings and hence the connecting isomorphisms arise from the same construction over $K$.
\end{proof}
Notice that a related statement for the cases $K=\dR$ and $K=\dQ$ and the assumptions as in Remark \ref{rem:Katz} was found in  \cite[Corollary 6.3 and 6.4]{Hien}, using rather different methods.

\bibliographystyle{amsalpha}

\def\cprime{$'$}
\providecommand{\bysame}{\leavevmode\hbox to3em{\hrulefill}\thinspace}
\providecommand{\MR}{\relax\ifhmode\unskip\space\fi MR }
\providecommand{\MRhref}[2]{%
  \href{http://www.ams.org/mathscinet-getitem?mr=#1}{#2}
}
\providecommand{\href}[2]{#2}

\vspace{1cm}

Davide Barco\\
Fakult\"at f\"ur Mathematik\\
Technische Universit\"at Chemnitz\\
09107 Chemnitz\\
Germany\\
davide.barco@mathematik.tu-chemnitz.de

\vspace{1cm}

\nd
Marco Hien\\
Institut f\"ur Mathematik \\
Universit\"at Augsburg\\
86135 Augsburg\\
Germany

marco.hien@math.uni-augsburg.de
\vspace{1cm}

\nd
Andreas Hohl\\
Universit{\'e} de Paris and Sorbonne Universit{\'e}, CNRS\\
Institut de Math{\'e}matiques de Jussieu-Paris Rive Gauche (IMJ-PRG)\\
75013 Paris\\
France

andreas.hohl@imj-prg.fr
\vspace{1cm}

\nd
Christian Sevenheck\\
Fakult\"at f\"ur Mathematik\\
Technische Universit\"at Chemnitz\\
09107 Chemnitz\\
Germany\\
christian.sevenheck@mathematik.tu-chemnitz.de

\end{document}